\documentclass[11pt]{article}

\usepackage{lineno}
\usepackage{color,tikz}
\usepackage{mathrsfs}
\usepackage{graphicx}
\usepackage{mathtools}
\usepackage{caption,subcaption}
\usepackage{amsmath}
\usepackage{amsthm}
\usepackage{amssymb}
\usepackage{textcomp}
\usepackage{euscript}
\usepackage{eufrak}
\usepackage{xcolor}
\usepackage{tcolorbox}
\usepackage[titletoc,title]{appendix}
\usepackage{upgreek}
\usepackage{hyperref}
\usepackage{float}
\usepackage{longtable}

\makeatletter
\newcommand{\@giventhatstar}[2]{\left(#1\,\middle|\,#2\right)}
\newcommand{\@giventhatnostar}[3][]{#1(#2\,#1|\,#3#1)}
\newcommand{\giventhat}{\@ifstar\@giventhatstar\@giventhatnostar}
\makeatother

\usepackage{float}
\usepackage[scaled]{helvet}

\usepackage{algorithm,algpseudocode}
\algnewcommand{\Inputs}[1]{%
  \State \textbf{Inputs:}
  \Statex \hspace*{\algorithmicindent}\parbox[t]{.8\linewidth}{\raggedright #1}
}
\algnewcommand{\Initialize}[1]{%
  \State \textbf{Initialize:}
  \Statex \hspace*{\algorithmicindent}\parbox[t]{.8\linewidth}{\raggedright #1}
}

\renewcommand{\comment}[1]{}

\newcommand{\pp}{\mathbb{P}}
\newcommand{\dd}{\mathrm{d}}
\newcommand{\ee}{\mathbb{E}}

\newcommand{\rbdd}{\overline{r}}

\usepackage[utf8]{inputenc}

\usepackage[margin=1in]{geometry}

\newcommand{\reals}{\mathbb{R}}

\newcommand{\mbb}{\mathbb}

\newcommand{\lt}{\left}
\newcommand{\rt}{\right}
\newcommand{\wt}{\widetilde}
\newcommand{\wh}{\widehat}

\newcommand{\pois}{\mathsf{Poi}}
\newcommand{\bern}{\mathsf{Bern}}
\newcommand{\unif}{\mathsf{Unif}}

\newcommand{\pno}{\mathcal{P}(\lambda)}
\newcommand{\lfs}{\mathbf{x}}
\newcommand{\state}{\mathbf{X}}

\newcommand{\argmin}{\operatornamewithlimits{argmin}}
\newcommand{\pnon}{\mathcal{P}_n(\lambda)}

\newcommand{\comDCone}{\mathcal{T}_{\mathsf{DC1}}}

\newcommand{\comAR}{\mathcal{T}_{\mathsf{NAR}}}

\newtheorem{theorem}{Theorem}

\newtheorem{proposition}{Proposition}

\newtheorem{remark}{Remark}
\newtheorem{corollary}{Corollary}
\newtheorem{lemma}{Lemma}
\title{Rejection and Importance Sampling based Perfect Simulation for Gibbs hard-sphere models}
\author{
  Moka, S. B.\\ {\normalfont University of Queensland}\\
  \and
  Juneja, S.\\{\normalfont TIFR, Mumbai}\\
  \and
  Mandjes, M. R. H.\\ {\normalfont University of Amsterdam }
}
\date{}
\begin{document}
\maketitle

\begin{abstract}
Coupling from the past (CFTP) methods have been used to generate perfect samples from finite Gibbs hard-sphere models, an important class of spatial point processes, which is a set of spheres with the centers on a bounded region that are distributed as a homogeneous Poisson point process (PPP) conditioned that spheres do not overlap with each other. We propose an alternative importance sampling based rejection methodology  for the perfect sampling of these models.  We analyze the asymptotic expected running time complexity of the proposed method when the intensity of the reference PPP increases to infinity while the (expected) sphere radius decreases to zero at varying rates. We further compare the performance of the proposed method analytically and numerically with a naive rejection algorithm and popular dominated CFTP algorithms. Our analysis relies upon identifying large deviations decay rates of the non-overlapping probability of spheres whose centers are distributed as a homogeneous PPP.
\end{abstract}


\section{Introduction}
Perfect sampling, that is, generating unbiased samples from a target distribution (also referred to as perfect simulation or exact sampling), is an important and exciting area of research in stochastic simulation. In this paper, {we introduce and investigate a novel methodology for generating perfect samples of finite Gibbs hard-sphere models}, which are an important family of Gibbs point processes. Roughly, a Gibbs hard-sphere model can be described as a set of spheres such that their centers constitute a Poisson point process on a bounded Euclidean space conditioned that
no two spheres overlap with each other.
The proposed methodology combines {\it importance sampling} (IS) and {\it acceptance-rejection} (AR) techniques to achieve  substantial performance improvement in certain important regimes of interest.
 In statistical physics, there is a large body of work related to the Gibbs hard-sphere models; see, e.g., \cite{Pathria72, MM40, Adams74, AS13, LM66, SRJ53, JM13, BT16}.
These models are important also in modelling adsorption of latexes or proteins on solid surfaces \cite[and references therein]{TST90, SVS00}. 
For the analysis of wireless communication networks, it is common to use the Gibbs hard-sphere models to model base-stations in a cellular network because no two base-stations are to be normally placed closer than a certain distance from each other \cite{SKM87, Haenggi12}.
Our results can be used to assess the stationary behaviour of Code Division Multiple Access (CDMA) wireless networks.\\

{\em Literature Review:}
{The existing literature offers several perfect sampling methods for Gibbs hard-sphere models. Among these,  
the dominated coupling from the past (dominated CFTP) methods are most prominent and they are based on the seminal paper by Propp \& Wilson \cite{PW96}}; see \cite{KM00, KW98, MH12, Kendall15}. 
Another well-known perfect sampling method for the Gibbs hard-sphere models is called the backward-forward algorithm (BFA) by Ferrari et al. \cite{FFG02}; also see \cite{MH16, GNL00}.
{To see some of the applications of perfect sampling for these models, refer to \cite{BKM08, BM06, MPB06}.
For other related literature on perfect sampling for spatial point processes, refer to \cite{JMR10, HVM99}.}
As mentioned in \cite{GNL00}, all the existing methods are, in some sense, complementary to each other.
They take advantage of an important property that the distribution of a Gibbs hard-sphere model can be realized as an invariant measure of a spatial birth-and-death process,
call it the {\em interaction process}. For example, the main ingredient of the dominated CFTP method is to construct a birth-and-death process backward in time starting from its steady-state at time zero such that it dominates the interaction process, and then use {\it thinning} on the dominating process to construct coupled upper and lower bound processes forward in time such that the coalescence of these two bounding processes assures a perfect sample from the target measure, which is the invariant measure of the interaction process. The BFA is based on the construction of the {\it clan of ancestors} that uses thinning of a dominating process and extends the applicability to infinite-volume measures. 
A crucial drawback of the naive AR and the dominated CFTP methods is that they are guaranteed to be efficient only if the intensity of the Gibbs hard-sphere model is close to the intensity of the reference Poisson point process; see \cite{MH16} for details. In addition, most of the dominated CFTP methods suffer from the so-called {\it impatient-user bias} (a bias that is induced when a user aborts long runs of the algorithm); see \cite{Fill98}, \cite{FMMR00} and \cite{Thon99}.\\

{\em Our Contributions:}
Acceptance-rejection methods are free of the impatient-user bias and involve neither thinning nor coupling (which are crucial for the other methods).
Despite being an obvious alternative to the existing methods, to the best of our knowledge, in the context of Gibbs point processes,
the use of AR methods is still largely unexplored (except brief discussions, e.g., in \cite{FJMM15} and \cite{MH16}).
AR methods for Gibbs hard-sphere models are amenable to further algorithmic enhancements that may substantially decrease the expected running time of the algorithm.
The proposed methodology provides one such enhancement. To highlight the significance of the proposed methodology,
we compare its running time complexity with that of both the naive AR and the dominated CFTP methods.
This effectiveness analysis is based on our large deviations analysis of the {\em non-overlapping probability}. A brief summary of our results is as follows.

\begin{itemize}
 \item Our first key contribution is that we conduct a large deviations analysis of the probability of spheres not overlapping with each other when their centers constitute a homogeneous Poisson point process (PPP). More specifically, we consider a homogeneous {\em marked} PPP on $[0,1]^d$ with intensity $\lambda$ where the points are the center of spheres with independently and identically distributed (iid) radii as marks which are independent of the centers and identical in distribution to  $R/\lambda^\eta$ for a positive bounded random variable  $R$ and a constant ${\eta > 0}$. We establish large deviations of the probability of spheres do not overlap with each other,  as ${\lambda \nearrow \infty}$. This analysis is useful in the study of the asymptotic behavior of the expected running time complexities of the proposed and the existing perfect sampling methods for the Gibbs hard-sphere models. This analysis may also be of independent interest.
       
 \item {Our second key contribution is that we propose a novel IS based AR algorithm for generating perfect samples of the Gibbs hard-sphere model obtained by considering the homogeneous marked PPP conditioned on no overlap of the spheres. This is achieved by partitioning the underlying configuration space and arriving at an appropriate change of measure on each partition. Applicability of the proposed algorithm is illustrated  in two scenarios.
       In the first scenario, all the spheres are assumed to be of a fixed size (i.e., $R$ is a fixed positive constant). We develop a grid based IS technique under which spheres are generated sequentially such that the chance of spheres overlapping is small and the corresponding likelihood ratio has a better deterministic upper bound that improves the acceptance probability in each iteration of the algorithm.
       In the second scenario, we consider the general case where spheres have iid radii. In this scenario, we divide the underlying configuration space into two sets. On one set, the sum of the volumes of spheres is bounded from below and on the other set, the volume sum takes small values so that the set consists of  {\em rare} configurations only. For the first set, we develop a grid based IS method that is similar to the one stated above, and for the second set, we use an {\it exponential twisting} on the sphere volume distribution. In both the scenarios, the new method provably substantially improves the performance of the algorithm compared to the naive AR method.}
       
 \item We analytically and numerically compare the performance of the proposed IS based AR method with  that of some of the dominated CFTP methods.
         The numerical results support our analytical conclusions that the proposed method is substantially efficient compared to the existing methods over the {\em high density regime} where $\eta d \leq 1$ and $\lambda$ is large.
\end{itemize}

{\em Organization:} 
Section~\ref{sec:prel} provides a definition of the hard-sphere model. The large deviations of the non-overlapping probability is presented in Section~\ref{sec:LDR}.
In Section~\ref{sec:AR_methods}, we first review a naive AR method and analyze its expected running time complexity, and we then propose and analyze the IS based AR method. In Section~\ref{sec:dcftp}, a review of the well-known dominated CFTP methods for the hard-sphere models is given.
Section~\ref{sec:NumExp} illustrates the efficiency of the proposed methodology using numerical experiments.
Section~\ref{sec:ConAR} is a brief conclusion of the paper. All proofs are presented in Appendix~\ref{Proofs}.

\section{Preliminaries}
\label{sec:prel}

First we introduce some notation. $X \sim F$ denotes that the distribution of a random object $X$ is $F$. $\pois(\lambda)$ and $\bern(p)$ denote, respectively,
Poisson distribution with mean $\lambda > 0$ and Bernoulli distribution with success probability $p$.
The uniform distribution on $[0,1]$ is denoted by $\unif(0,1)$. For an event $A$, the indicator function $I(A)$ is equal to $1$ if $A$ occurs, otherwise it is equal to $0$.
A measure $\mu_1$ is absolutely continuous with respect to a measure $\mu_2$ on a measurable set $A$ if $\mu_1(B\cap A) = 0$ for any measurable $B$ such that
$\mu_2(B\cap A) = 0$. {For any probability measure $\mu$, $\pp_{\mu}(A)$ denotes the probability of an event $A$ under $\mu$,
and $\ee_{\mu}[\cdot]$ denotes the associated expectation. We drop the subscript $\mu$ when it is not relevant.} 
For any non-negative real valued functions $f$ and $g$, write $f(x) = O(g(x))$ if $\limsup_{x\rightarrow \infty} f(x)/g(x)~\leq~c$ for some constant $c > 0$, write $f(x) = \varOmega(g(x))$ if $g(x) = O(f(x))$, and write $f(x) = o(g(x))$ if $\limsup_{x\rightarrow \infty} f(x)/g(x) = 0 $.
Write $f(x) = \Theta(g(x))$ {if both} $f(x) = O(g(x))$ and $f(x) = \varOmega(g(x))$ are true. For any real value $x$, the largest integer $n$ such that $n \leq x$ is denoted by $\lfloor x\rfloor$ and the smallest integer $n$ such that $n \geq x$ is denoted by $\lceil x\rceil$. The set of all the non-negative integers is denoted by $\mbb{N}_0$.\\

A random finite subset $\state = \{X_1, \dots, X_N\}$ of an observation window $W \subset \reals^d$ is called a {\em Poisson point process} (PPP) with a finite intensity measure $\nu$ on $W$ if $N \sim \pois(\nu(W))$ and for every $n \in \mbb{N}_0$, conditioned on $N = n$, the points $X_1, \dots, X_n$ are iid with distribution $\nu(\dd x)/ \nu(W)$. A PPP  on $[0,1]^d$ is called $\lambda$-{\it homogeneous} PPP with intensity $\lambda > 0$ if the intensity measure $\nu(\dd x) = \lambda\, \dd x$,  {where $\dd x$ is Lebesgue measure on $W$.}
To each point $X_i$ of the $\lambda$-homogeneous PPP on $[0,1]^d$, we associate a mark  {which is a non-negative number interpreted as the radius of }a sphere centered at $X_i$.
In particular, a {\it $\lambda$-homogeneous marked  PPP} on $[0,1]^d$ is a PPP on $W = [0,1]^d \times [0, \infty)$ with the intensity measure ${\nu(\dd x \times \dd r) = \lambda \dd x \times F(\dd r)}$ where $F$ is the distribution of each radius. That is, the centers constitute a $\lambda$-homogeneous PPP on $[0,1]^d$ which is independent of the radii, and the radii are iid with distribution $F$. A realization of the marked PPP with $n$ points is denoted by $\lfs = \{(y_1, r_1), \dots, (y_n, r_n)\}$, where $r_i \geq 0$ is the radius of the sphere centered at $y_i \in [0,1]^d$. Define $\mathscr{G} = \cup_{n \in \mbb{N}_0}\mathscr{G}_n$ where
\[
\mathscr{G}_n = \big\{ \lfs = \{(y_1, r_1), \dots, (y_n, r_n)\} :  (y_i, r_i) \in [0,1]^d\times [0, \infty),\, \text{ for }\, i = 1, \dots, n \big\}.
\]

Now we  {define a }{\em Gibbs hard-sphere model}. Suppose that $\mu^0$ is the distribution of a $\lambda$-homogeneous marked PPP as defined above with $F$ being the distribution of $R/\lambda^\eta$ for a constant $\eta > 0$ and a non-negative random variable $R$. Let $\mathscr{A} \subset \mathscr{G}$ be the set of all configurations with no two spheres overlapping with each other. Then the distribution $\mu$ of  the Gibbs hard-sphere model is absolutely continuous with respect to $\mu^0$ with the Radon-Nikodym derivative given by
\begin{equation}
 \label{eqn:dist_mu_HS}
  \frac{d\mu}{d\mu^0}(\lfs) = \frac{I\lt(\lfs \in \mathscr{A}\rt)}{\pno}, \quad \lfs \in \mathscr{G},
\end{equation}
where the normalizing constant $\pno$ is the non-overlapping probability given by
\begin{align}
 \label{eqn:non-overlap}
 \pno = \pp_{\mu^0} \lt( \state \in \mathscr{A} \rt).
\end{align}

We refer to the Gibbs hard-sphere  {model as a} {\em torus-hard-sphere model} if the boundary of the underlying space $[0,1]^d$ is periodic, that is, a sphere $S(x, a)$ centered at $x \in [0,1]^d$ with radius $r$ is defined by 
$$S(x, r) = \lt\{ (y_1\, \mathsf{mod}\, 1, \dots, y_d\, \mathsf{mod}\, 1): y = (y_1, \dots, y_d)\in \reals^d, \, \,\|x- y \| < r\rt\},$$ 
where $\|\cdot \|$ is the  $d$-dimensional Euclidean norm and '$\mathsf{mod}$' denotes the modulo operation \cite{Boute92}. If the boundary is not periodic, we refer to the  {model as} a {\em Euclidean-hard-sphere model}.\\

From now onwards, the phrase `hard-sphere model' refers to either of these two models and we assume that $R$ is bounded from above by a constant $\rbdd > 0$.  In particular, if $R$ is a constant, we take $\rbdd = R$. Furthermore, we assume that $2\rbdd/\lambda^\eta < 1$ to avoid certain trivial difficulties such as the possibility of a sphere on the torus overlapping with itself.

\section{Large Deviations Results}
\label{sec:LDR}
In this section, we obtain large deviations results for the non-overlapping probability $\pno$. We use these results 
for analyzing the running time complexity of both the naive and importance sampling based acceptance-rejection methods. Hereafter, $\gamma = {\pi^{d/2}}/{\Gamma(d/2 + 1)}$, where $\Gamma(\cdot)$ is the gamma function. Note that the volume of a sphere with radius $r$ is given by $\gamma r^d$.
Define $ m_1 := {\mbb E}[(R + \wh R)^{d}]$, where $\wh R$ is independent and identical in distribution to $R$, and let
\begin{align}
\label{eqn:def_gamm_prime}
\gamma' = \begin{cases}
            \gamma, &\text{ if  $[0,1]^d$ is treated as the torus},\\
            \gamma/2^d, &\text{otherwise}.
                  \end{cases}
\end{align}
 
 \begin{theorem}
\label{lem:non-ovr-rand}
 The non-overlapping probability $\pno$ satisfies
\begin{align*}
 \lim_{\lambda \rightarrow \infty}\pno &= \begin{cases}  1, &\text{ if }\, \eta d > 2,\\
                                                     \exp\lt( -\frac{\gamma m_1}{2}\rt), &\text{ if }\, \eta d = 2,
 \end{cases} \\
\lim_{ \lambda \rightarrow \infty}\lt[\frac{1}{\lambda^{2-\eta d}}\log \pno\rt]  &= - \frac{\gamma m_1}{2}, \,\,\text{ if }\,\, 1 < \eta d < 2,\\
\text{ and }\, \, \, \,\lim_{\lambda \rightarrow \infty}\left[\frac{1}{\lambda}\log \pno  \right] &=   -1,\, \text{ if }\, 0 < \eta d < 1.
\end{align*}
When $\eta d = 1$, the limit $\delta := \lim_{\lambda \rightarrow \infty}\left[\frac{1}{\lambda}\log \pno  \right]$ exists and $-1 \leq \delta < 0$.
Furthermore, $\delta \nearrow 0$ if $\gamma m_1 \searrow 0$, and $\delta \leq  -\frac{1}{2} \lt(1 - \frac{1}{\gamma' \rbdd^d} \rt)^2$ if $R \equiv \rbdd$ and $\gamma' \rbdd^d > 1$.
In addition, for the torus-hard-sphere model,
\begin{align*}
 \lim_{ \lambda \rightarrow \infty}\lt[\pno\exp\lt( \frac{\gamma m_1}{2}\lambda^{2 - \eta d}\rt) \rt]  &= 1, \,\,\text{ if }\,\, 5/3 < \eta d < 2.
\end{align*}
\end{theorem}

An important and fundamental characteristic of a Gibbs point process is its {\em intensity}; see, for example, \cite[and references therein]{MMSWD01} and \cite{BN12}. Roughly speaking, the intensity of a Gibbs point process is the expected number of points of the process per unit volume. There is an interesting connection between the regimes considered in Theorem~\ref{lem:non-ovr-rand} and the asymptotic intensity of the torus-hard-sphere model. To see this, assume that each sphere has a fixed radius $\rbdd/\lambda^\eta$. Since the underlying space is $[0,1]^d$, the intensity $\rho(\lambda)$ of the model is exactly equal to the expected total number of points in a realization of the model. Equivalently, we may consider the fraction of the volume $\mathsf{VF}(\lambda)$ occupied by the spheres, given by 
$\mathsf{VF}(\lambda) = \rho(\lambda) \gamma \rbdd^d \lambda^{- \eta d}$.
For the torus-hard-sphere model, the volume fraction $\mathsf{VF}(\lambda)$ is bounded from above by $\rho^{\max} \gamma$, where $\rho^{\max}$ is the closest packing density defined by 
$\rho^{\max} = \lim_{n \rightarrow \infty} N_n/(n+1)^d,$
with $N_n$ being the maximal number of mutually disjoint unit radius spheres which are included in the hypercube $[-(n+1/2), (n+1/2)]^d$; see \cite{MMSWD01}. Proposition~\ref{prop:Pack_int} describes asymptotic behavior of $\mathsf{VF}(\lambda)$ as $\lambda \to \infty$ for different values of $\eta d$. In particular, the regime with $\eta d > 1$ is a {\em low density regime} while the regime with $\eta d < 1$ is a {\em high density regime}. In the high density regime, the intensity of the hard-sphere model is much smaller than the intensity $\lambda$ of the reference PPP.
\begin{proposition}
\label{prop:Pack_int}
For the torus-hard-sphere model with a fixed radius $R = \rbdd$,
\begin{align*}
&\lim_{\lambda \nearrow \infty} \frac{\mathsf{VF}(\lambda)}{\gamma \rbdd^d \lambda^{1 - \eta d}} = 1, \quad \text{ if }\, \eta d > 1, \\
&\lim_{\lambda \nearrow \infty} \mathsf{VF}(\lambda) = \rho^{\max} \gamma,   \quad \text{ if }\, \eta d < 1,\\
&\lim_{\lambda \nearrow \infty} \mathsf{VF}(\lambda) <\rho^{\max} \gamma,   \quad \text{ if }\, \eta d = 1.
\end{align*}
\end{proposition}

\section{Acceptance-Rejection Based Algorithms}
\label{sec:AR_methods}
In Section~\ref{sec:exact}, we present a naive acceptance-rejection (AR) algorithm for generating perfect samples of the hard-sphere model and analyze its expected running time complexity. We then proceed to present and analyze our importance sampling (IS) based AR algorithm where the key idea is to partition the configuration space $\mathscr{G}$ so that a well chosen IS technique can be  implemented on each partition.  One such IS for the hard-sphere model is the reference IS presented in Section~\ref{sec:RIS} where spheres are generated sequentially such that, whenever possible, the center of each sphere is selected uniformly over the region on $[0,1]^d$ that guarantees no overlap with the existing spheres.  However, generating samples from this IS measure can be computationally challenging when $d \geq 2$. The grid based IS introduced in Sections~\ref{sec:improv-const} and \ref{sec:improv-random}  overcomes this difficulty by imitating the reference IS, and interestingly, it is more efficient than the reference IS. \\

In every algorithm presented in this paper, the running time complexity is calculated under the assumption that checking overlap of a newly generated sphere with an existing sphere is done in a sequential manner.  That is, if there are $n$ existing spheres, the expected running time complexity of the overlap check is proportional to $n$. However, if enough computing resources are available, the overlap check can be done in parallel so that its running time complexity is a constant. We omit the discussion of this parallel overlap check because it is easy to modify the results to accommodate the parallel case, and also the key conclusions of the paper do not change.

\subsection{Naive AR Algorithm}
\label{sec:exact}
Algorithm~\ref{alg:ext1} is a naive AR algorithm for generating perfect samples of the Gibbs hard-sphere model. {The basic idea of the algorithm is standard \cite{Dev86}}, and its correctness is straightforward and hence omitted. 
\begin{algorithm}
  \caption{Naive AR Method}
  \label{alg:ext1}
  \begin{algorithmic}[1]
   \Repeat
    \State{Generate $N \sim \pois(\lambda)$}
    \State{$\state \leftarrow \varnothing$}
    \If {$N \neq 0$}
    \State{$i \leftarrow 0$}
    \Repeat
        \State{$i \leftarrow i +1$}
    \State{Generate $Y_i$ independently and uniformly distributed on $[0,1]^d$}
    \State{Generate a copy $R_i$ of $R$ independently of everything else}
    \State{$\state \leftarrow \state \cup \{(Y_i, R_i/\lambda^\eta)\}$}
    \Until {$i = N$ or $\state \notin \mathscr{A}$}
    \EndIf
      \Until {$\state \in \mathscr{A} $}\\
    \Return{$\state$}
  \end{algorithmic}
\end{algorithm}

 Let ${\mathcal{T}}_{\mathsf{NAR}}$ be the expected running time complexity of Algorithm~\ref{alg:ext1}, where the running time complexity denotes
the number of elementary operations performed by the algorithm; every elementary operation takes at most a fixed amount of time.
Note that the acceptance probability of each iteration is $\pno$. Then 
the expected total number of iterations of the algorithm is $1/\pno$.
Suppose $C_{\mathsf{itr}}(\lambda)$ is the expected running time complexity of an iteration. Then,
\begin{align}
\label{eqn:TAR}
 \comAR = \frac{C_{\mathsf{itr}}(\lambda)}{\pno}.
\end{align}

We now establish bounds on ${\mathcal{T}}_{\mathsf{NAR}}$,
and then establish its asymptotic behavior as $\lambda \nearrow \infty$ using Theorem~\ref{lem:non-ovr-rand}. In each iteration of Algorithm~\ref{alg:ext1},  spheres are generated in a sequential order until we see an overlap or a configuration with $N$ non-overlapping spheres. 
The key to prove Proposition~\ref{prop:AR_meth} is to establish that the expected number of spheres generated per iteration is $\Theta\lt(\lambda^{\min\{\eta d, 2 \}}\rt)$.
\begin{proposition}
\label{prop:AR_meth}
The expected running time complexity $C_{\mathsf{itr}}(\lambda)$ of an iteration of the naive AR algorithm, Algorithm~\ref{alg:ext1}, satisfies
\begin{align}
\label{eqn:TAR-vs-Pno}
C_{\mathsf{itr}}(\lambda) = \Theta\lt( \lambda^{\min\{\eta d, 2 \}}\rt).
\end{align}
Furthermore, the expected total running time $\comAR$ satisfies:
\[
 \comAR = \begin{cases}
 \Theta\lt(\lambda^{2}\rt), & \quad \text{ if }\, \eta d \geq 2,\\
 \Theta\lt(\lambda^{\eta d}\,\exp\Big(\lt(\gamma m_1/2 +o(1)\rt)\lambda^{2 - \eta d }\Big)\rt), & \quad \text{ if }\, 1 < \eta d < 2,\\
 \Theta\Big(\lambda^{\eta d}\, \exp\lt(\delta\lambda\rt)\Big), \text{ for some }\, 0 < \delta \leq  1,& \quad\text{ if }\,\eta d = 1,\\
 \Theta\lt( \lambda^{\eta d}\, \exp\Big((1 + o(1))\lambda\Big)\rt),& \quad \text{ if }\,0 < \eta d < 1.
 \end{cases}
\]
\end{proposition}

\begin{remark}
 From \eqref{eqn:TAR-vs-Pno} and Theorem~\ref{lem:non-ovr-rand}, we see that for large values of $\lambda$ and for $\eta d < 2$, $\comAR$  is mainly governed by $\pno$, which can be very small for large $\lambda$. This suggests that any rejection based perfect sampling algorithm with a significant improvement in the acceptance probability will have a significantly improved running time complexity.
\end{remark}

\subsection{Importance Sampling Based Acceptance-Rejection Algorithm}
\label{sec:AR_GSC}
A sequence of tuples ${\lt\{ \lt(D_{n,k}, \mu_{n,k}, \sigma_{n,k}\rt)_{k = 1}^K\rt\}_{n \in \mbb{N}_0}}$ with some $K  \leq \infty$ is called  {\em stable} IS sequence if for each ${n \in \mbb{N}_0}$,  ${\lt(D_{n,k} \rt)_{k = 1}^{K}}$ is a partition of $\mathscr{G}_n$, and
$(\mu_{n,k})_{k = 1}^{K}$  a sequence of probability measures  such that $\mu^0$  is absolutely continuous with respect to  $\mu_{n,k}$ on $D_{n,k} \cap \mathscr{A}$ and the corresponding likelihood ratio ${L_{n,k}(\lfs_n) := \frac{d\mu^0}{d\mu_{n,k}}(\lfs_n)}$ satisfies
 \begin{align*}
L_{n,k}(\lfs_n)  \leq \sigma_{n,k} \leq 1, \text{ if }\, \lfs_n \in D_{n,k} \cap \mathscr{A},
 \end{align*}
 for $k = 1, \dots, K$. Under the stability condition, for every measurable subset ${\mathscr{B} \subseteq  \mathscr{G}}$,
\begin{align}
 \mu(\mathscr{B}) &\propto \pp_{\mu^0}(\state \in \mathscr{B} \cap \mathscr{A}) = \sum_{n \in \mbb{N}_0} e^{-\lambda}\frac{\lambda^n}{n!} \Bigg(\sum_{k=1}^{K} \pp_{\mu^0}\lt(\state_n \in D_{n,k} \cap \mathscr{B} \cap \mathscr{A}\rt) \Bigg) \nonumber\\
        &= \sum_{n \in \mbb{N}_0} e^{-\lambda}\frac{\lambda^n }{n!} \Bigg(\sum_{k=1}^{K} \, \ee_{\mu_{n,k}}\big[ I\lt(\state_n \in D_{n,k} \cap \mathscr{B} \cap \mathscr{A}\rt)L_{n,k}(\state_n)\big] \Bigg) \nonumber\\
        &= \sum_{n \in \mbb{N}_0} e^{-\lambda}\frac{\lambda^n \wt \sigma(n)}{n!} \Bigg(\sum_{k=1}^{K} \frac{\sigma_{n,k}}{\wt \sigma(n)}\, \ee_{\mu_{n,k}}\Bigg[ \frac{I(\state_n \in D_{n,k} \cap \mathscr{B} \cap \mathscr{A})L_{n,k}(\state_n)}{\sigma_{n,k}}\Bigg] \Bigg)\nonumber \\
        &=  \sum_{n \in \mbb{N}_0} \frac{\lambda^n \wt \sigma(n)}{n!} \Bigg(\sum_{k=1}^{K} \frac{\sigma_{n,k}}{\wt \sigma(n)}\, \pp_{\mu_{n,k}}\big( J = 1, \state_n \in D_{n,k} \cap \mathscr{B} \cap \mathscr{A} \big) \Bigg),\label{eqn:mu_IS}
\end{align}
where ${\wt \sigma(n) := \sum_{k=1}^{K} \sigma_{n,k}}$, ${U \sim \mathsf{Unif}(0,1)}$ and $ J \sim \bern\left(\frac{L_{n,k}(\state_n)}{\sigma_{n,k}}\right)$.
Let $M$ be a non-negative integer valued random variable with the
pmf defined by,
\begin{align}
 \label{eqn:dist_M_gen_stab}
 \pp\lt(M = m\rt) = \frac{1}{C_{\lambda}}\frac{\lambda^m \wt \sigma(m)}{m!},\,\, m \in \mbb{N}_0,
\end{align}
where $C_{\lambda} := \sum_{n=0}^\infty \frac{\lambda^n \wt \sigma(n)}{n!}$.
The pmf \eqref{eqn:dist_M_gen_stab} is well defined because $\ee\lt[\wt \sigma(N)\rt]$ is finite under the stability condition.
Now consider Algorithm~\ref{alg:ext_gen_stab}. 
\begin{algorithm}
  \caption{Importance Sampling Based AR method}
  \label{alg:ext_gen_stab}
  \begin{algorithmic}[1]
  \Repeat
    \State{Generate a sample of $M$ with pmf \eqref{eqn:dist_M_gen_stab}}
    \State{Generate $J_1$ with pmf $\pp(J_1 = k) = \sigma_{M,k}/\wt \sigma(M)$, $k = 1, \dots, K$}
    \State{Generate a realization $\state$ of $M$ points under the measure $\mu_{M, J_1}$} \label{line:X_M}
    \State{Generate $ J_2 \sim \bern\lt(\frac{L_{M, J_1}(\state) I\lt(\state \in D_{M,J_1} \cap \mathscr{A}\rt)}{\sigma_{M,J_1}}\rt)$} \label{line:bern}
    \Until {$J_2 = 1$}\\
    \Return{$\state$}
    
  \end{algorithmic}
\end{algorithm}
\begin{proposition}
\label{prop:P_acc_IS}
Algorithm~\ref{alg:ext_gen_stab} generates a perfect sample of the Gibbs hard-sphere model. Furthermore, let $N \sim \pois(\lambda)$. Then the probability of accepting the configuration generated in an iteration of Algorithm~\ref{alg:ext_gen_stab} is given by
\begin{align}
\label{eqn:Pacc_IS}
 P_{\mathsf{acc}}(\lambda) = \frac{\pno}{\ee[\wt \sigma(N)]}.
\end{align}
\end{proposition}
We omit the proof of Proposition~\ref{prop:P_acc_IS} because the correctness easily follows from \eqref{eqn:mu_IS}, and \eqref{eqn:Pacc_IS} holds from the observation that
\begin{align*}
 P_{\mathsf{acc}}(\lambda) = \frac{1}{C_{\lambda}}\sum_{n \in \mbb{N}_0} \frac{\lambda^n \wt \sigma(n)}{n!} \Bigg(\sum_{k=1}^{K} \frac{\sigma_{n,k}}{\wt \sigma(n)}\, \ee_{\mu_{n,k}}\Bigg[ \frac{L_{n,k}(\state_n)}{\sigma_{n,k}}; \state_n \in D_{n,k} \cap  \mathscr{A}\Bigg] \Bigg).
\end{align*}
{Note that the expected number of iterations of Algorithm~\ref{alg:ext_gen_stab} is $1/P_{\mathsf{acc}}(\lambda)$. Corollary~\ref{cor:no_of_itr} is an important and trivial consequence of Proposition~\ref{prop:P_acc_IS}. 
\begin{corollary}
\label{cor:no_of_itr}
For all stable IS sequences $\lt\{ \lt(D_{n,k}, \mu_{n,k}, \sigma_{n,k}\rt)_{k = 1}^K\rt\}_{n \in \mbb{N}_0}$ with the same $\ee[\wt \sigma(N)] = \sum_{k =1}^{K}\ee\lt[\sigma_{N,k}\rt]$, the expected number of iterations of Algorithm~\ref{alg:ext_gen_stab} is the same.
\end{corollary}}

Suppose that $\wt C_{\mathsf{itr}}(\lambda)$  is the  expected running time complexity of an iteration of Algorithm~\ref{alg:ext_gen_stab}. Then the expected total running time of the algorithm is given by
\begin{align}
\label{eqn:ISAR_cost}
{\mathcal{T}}_{\mathsf{ISAR}} = \frac{\wt C_{\mathsf{itr}}(\lambda)\ee[\wt \sigma(N)]}{\pno},
\end{align}
where ${N \sim \pois(\lambda)}$.
Recall that the acceptance probability of the naive AR method is $\pno$.
It is reasonable to seek a valid stable IS sequence  ${\lt\{ \lt(D_{n,k}, \mu_{n,k}, \sigma_{n,k}\rt)_{k = 1}^K\rt\}_{n \in \mbb{N}_0}}$
so that ${\wt C_{\mathsf{itr}}(\lambda)\ee[\wt \sigma(N)]}$ is smaller than ${C_{\mathsf{itr}}(\lambda)}$.
In Subsections~\ref{sec:improv-const} and \ref{sec:improv-random}, we present applications of Algorithm~\ref{alg:ext_gen_stab} where $\mathcal{T}_{\mathsf{ISAR}}$ is indeed much smaller than $\comAR$.

\begin{remark}[Extension of IS Based AR to General Gibbs Point Processes]
Suppose that $\mu$ is the distribution of a Gibbs point process that is absolutely continuous with respect to $\mu^0$ with the corresponding Radon-Nikodym derivative given by
$ {\frac{d\mu}{d\mu^0} \lt(\lfs \rt) = \frac{\exp\lt( - \beta\, V(\lfs)\rt)}{Z},\, \lfs \in \mathscr{G}},$ where the constant $\beta \in \reals$ is known as inverse temperature,
$V$ is called non-negative potential function, and the normalizing constant
${Z  = \ee_{\mu^0} \lt[\exp\lt( - \beta\, V(\state)\rt) \rt]}$. If the stability condition holds true when $I(\lfs _n\in \mathscr{A})$ is replaced by $\exp\lt( - \beta\, V(\lfs_n)\rt)$, then Algorithm~\ref{alg:ext_gen_stab} can  generate perfect samples from $\mu$ if in line \ref{line:bern} of the algorithm, 
\[
J_2 \sim \bern\lt(\frac{L_{M, J_1}(\state) \exp\lt( - \beta\, V(\state)\rt)I\lt(\state \in D_{M,J_1} \rt)}{\sigma_{M,J_1}}\rt).
\]
To see that the hard-sphere model is a special case of such a Gibbs point process, take  $\beta > 0$ and assume that ${V(\lfs) = 0}$ if $\lfs$ is a non-overlapping configuration of spheres, otherwise, ${V(\lfs)  = \infty}$.
\end{remark}

\subsection{Reference Importance Sampling}
\label{sec:RIS}
We now introduce an IS measure, called reference IS and denoted by $\wt \mu_n$ for each $n$, so that  $\lt\{ \lt(\mathscr{G}_n, \wt \mu_n, \sigma_n\rt)\rt\}_{n \in \mbb{N}_0}$ is a stable IS sequence (with $K = 1$) that can be used in Algorithm~\ref{alg:ext_gen_stab} for generating perfect samples of the hard-sphere model for an appropriate choice of the sequence $\lt\{\sigma_n : n \in \mbb{N}_0\rt\}$. Under $\wt \mu_n$, first generate iid sequence $R_1, \dots, R_n$ identical in distribution to $R$, and then $n$ spheres are generated sequentially as follows. Generate the center of the first sphere uniformly distributed on $[0,1]^d$. Suppose that $i-1$ spheres are already generated. For the $i^{th}$ sphere generation, a subset $\mathcal{B}_i \subseteq [0,1]^d$ is called blocking region if $\mathcal{B}_i$ is the largest set such that the center $Y_i$ of the $i^{th}$ sphere falling in this region (that is, $Y_i \in \mathcal{B}_i $) would result in an overlap of the $i^{th}$ sphere with  one of the existing $i-1$ spheres.  
The center of the $i^{th} $ sphere is generated with uniform distribution over the non-blocking region $[0,1]^d\setminus \mathcal{B}_i$. If for any sphere $i \leq n$, the entire space is blocked (that is, $\mathcal{B}_i = [0,1]^d$), we select the centers of spheres $ i, \dots, n$ arbitrarily. Figure~\ref{fig:block} illustrates this for $d =2$ and $n = 1,2$. In conclusion, $\wt \mu_n$ is the distribution of an output of Algorithm~\ref{alg:Ref_IS}.

\begin{algorithm}
  \caption{Reference Importance Sampling}
  \label{alg:Ref_IS}
  \begin{algorithmic}[1]
  \State{{\bf Input:} The total number of spheres $n$}
  \State{$\state \leftarrow \varnothing$}
  \If {$n \neq 0$}
  \State{$\mathcal{B}_0 = \varnothing$ and $i \leftarrow 0$}
  \Repeat
      \State{$i \leftarrow i + 1$}
    \State{Generate a copy $R_i$ of $R$ independently of everything else so far generated}
        \If {$\mathcal{B}_{i} = [0,1]^d$}        
            \State{Select the center $Y_i$ of the $i^{th}$ sphere arbitrarily over $[0,1]^d$}
    \Else
       \State{Identify the non-blocking region $\mathcal{B}^{\mathsf{c}}_i$} \label{step:findingBc}
       \State{Generate $Y_i$ uniformly distributed over $\mathcal{B}^{\mathsf{c}}_i$}    
    \EndIf   
    \State{$\state \leftarrow \state \cup \{(Y_i, R_i/\lambda^\eta)\}$}  
    \Until{$i = n$}
%
    
 \EndIf\\
 
    \Return{$\state$}
    
  \end{algorithmic}
\end{algorithm}
\begin{figure}[h]
    \begin{center}
    \begin{subfigure}{0.47\textwidth}
        \begin{center}
        \includegraphics[height=1\textwidth]{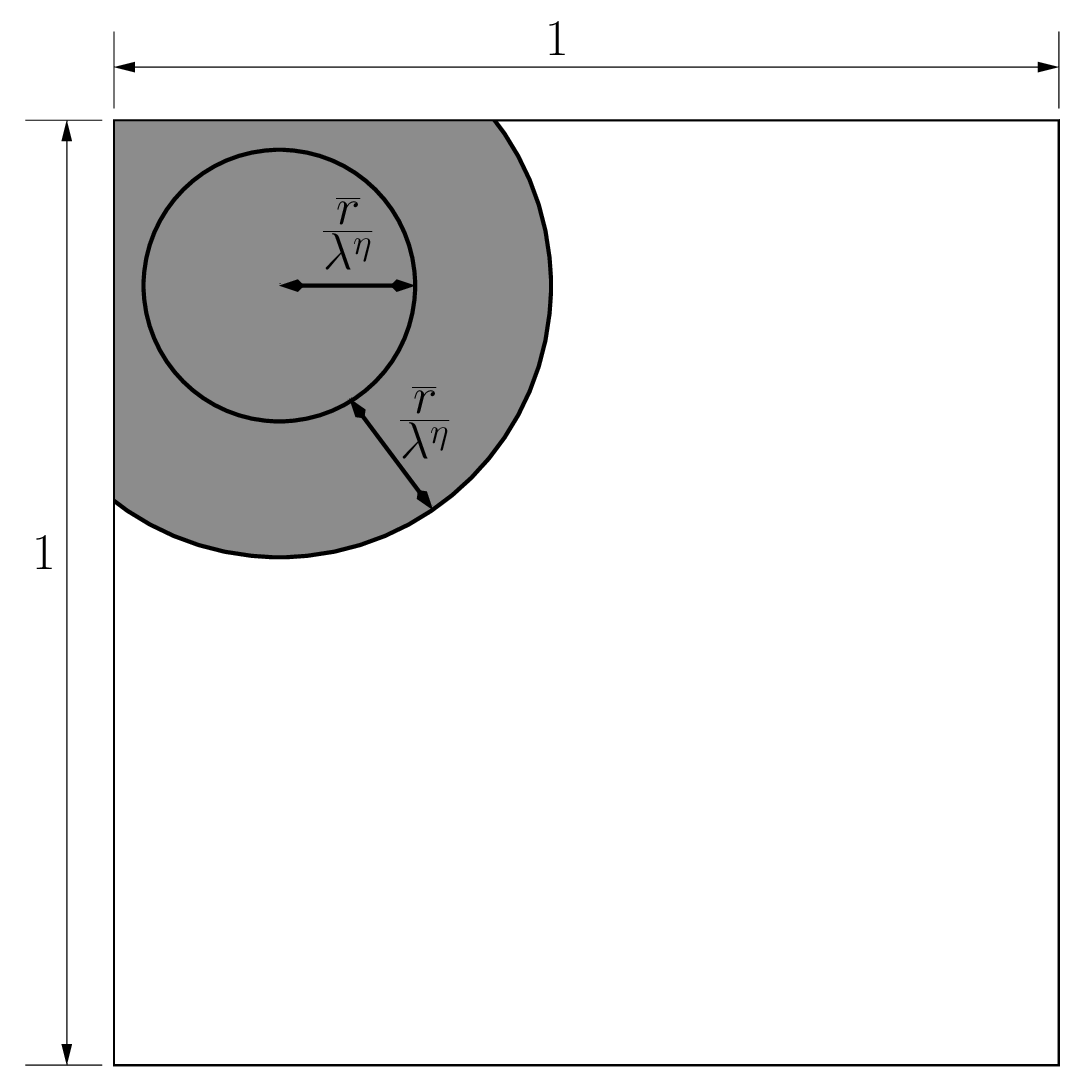} 
        \end{center}
        \caption{}
    \end{subfigure}%
    ~
    \begin{subfigure}{0.47\textwidth}
        \begin{center}
        \includegraphics[height=1\textwidth]{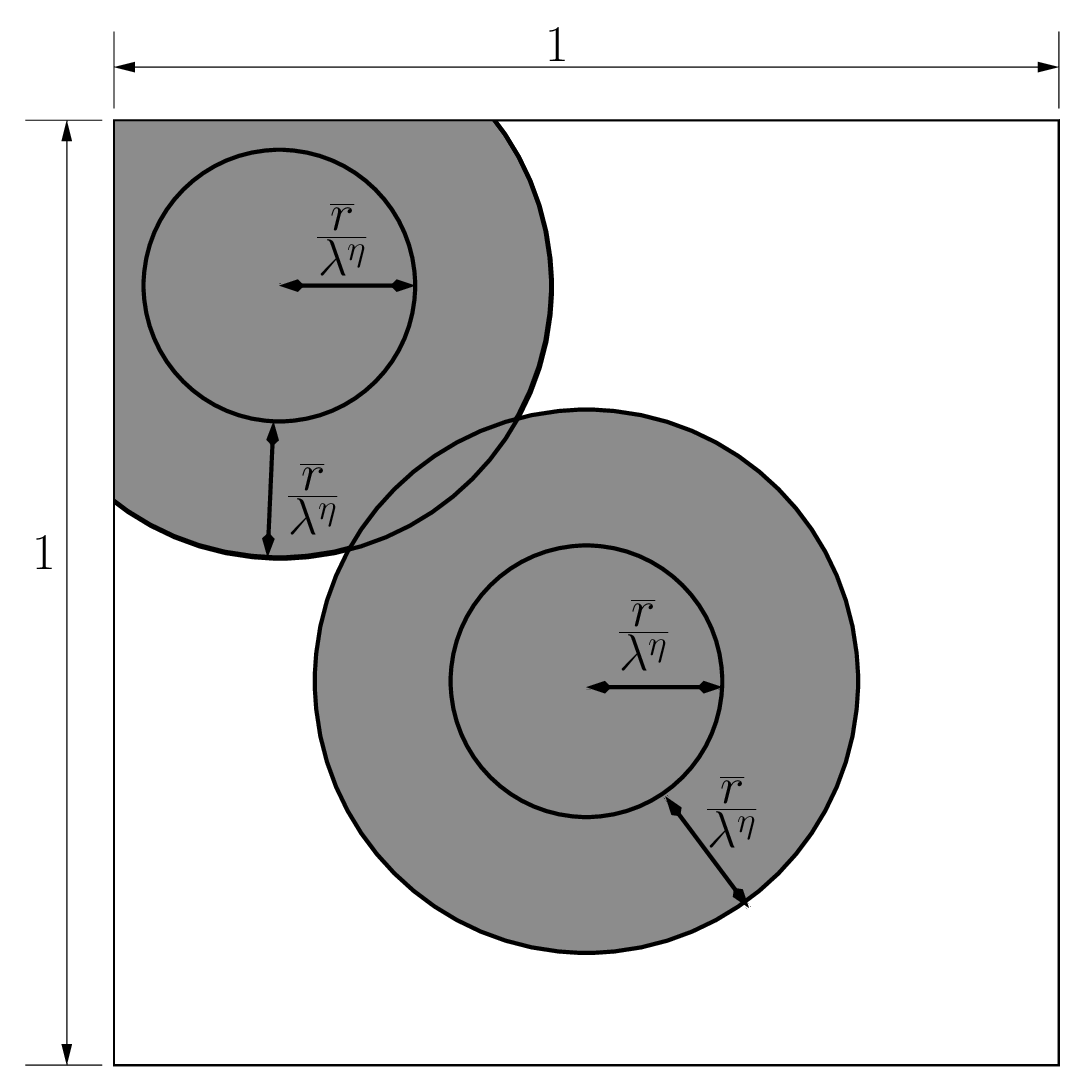} %
        \end{center}
        \caption{}
    \end{subfigure}
    \caption{Illustration of the reference IS method for a Euclidean-hard-sphere model on $[0,1]^2$ with spheres of fixed radius $\rbdd/\lambda^\eta$.
                    In (a) (respectively, (b)), the grey region represents the blocking area when generating the second circle (respectively, when generating 
                    the third circle). }
    \label{fig:block}
   \end{center}
\end{figure}

 Observe that $\mu^0$ is absolutely continuous with respect to $\wt \mu_n$ on $\mathscr{G}_n \cap \mathscr{A}$, and the associated likelihood
ratio satisfies 
\begin{align}
\label{eqn:LHRatio1}
 \wt L_n(\lfs_n) = \frac{d \mu^0}{d\wt \mu_n}(\lfs_n) = \prod_{i=1}^{n}\Big(1 - B_{i}\Big),
\end{align}
for all $\lfs_n \in \mathscr{G}_n \cap \mathscr{A}$ and $n \in \mbb{N}_0$, where {$B_{i}$ is the volume of $\mathcal{B}_i$} and $ \wt L_0 = 1$. Note that $\wt L_n(\lfs_n) = 0$ if and only if $\lfs_n \notin \mathscr{A}$ because for any $\lfs_n \notin \mathscr{A}$, there exists $i \leq n$ such that $B_i = 1$.\\

Observe that the blocking volume added by the $i^{th}$ sphere is at least $\gamma' \lt(R_i/\lambda^{\eta}\rt)^d$ when it does not overlap with any of the existing spheres.
This is because, for the torus-hard-sphere model, the entire volume within an accepted sphere is added to blocking volume, and for
the Euclidean-hard-sphere model, at least $1/2^d$ fraction of an accepted sphere is added to the blocking  volume. Thus,
\begin{equation}
\label{eqn:bni}
B_{i}\geq  \frac{\gamma'}{ \lambda^{\eta d}} \sum_{j =1}^{i-1} R_j^d,
\end{equation}
for every configuration $\lfs_{i-1} \in  \mathscr{G}_{i-1} \cap \mathscr{A}$. 
In particular, if all the spheres are of the same size with a fixed radius $\rbdd$, 
\begin{align}
\label{eqn:LHRatio_bdd}
 I(\lfs_n \in \mathscr{A}) \wt L_n(\lfs_n) \leq \prod_{i = 1}^{n} \lt( 1 - (i-1)\frac{\gamma'}{ \lambda^{\eta d}}  \rbdd^d\rt)^+ =: \delta_{n},
\end{align}
for all $n \in \mbb{N}_0$ and $\lfs_n \in  \mathscr{G}_n $, where $x^+ = \max(0,x)$ and $\delta_{0} = 1$. Then the stability condition is satisfied with $K = 1$, $D_{n,1} = \mathscr{G}_n $, $\mu_{n,1} = \wt \mu_n$  and $\sigma_{n,1} = \delta_n$ for $n \in \mbb{N}_0$. Thus, Algorithm~\ref{alg:ext_gen_stab} generates perfect samples of the fixed radius hard-sphere model, and from Proposition~\ref{prop:P_acc_IS}, the corresponding acceptance probability $$P_{\mathsf{acc}}(\lambda) = \frac{\pno}{\ee[\wt \sigma_N]} = \frac{\pno}{\ee[\delta_{N}]}.$$

\begin{remark}
\label{rem:tessellation}
\normalfont
When the dimension $d = 1$, spheres become line segments and thus it is easy to generate samples from the IS measure $\wt \mu_n$. However, for $d \geq 2$, generating samples under the reference IS is difficult because every time a new sphere is generated, we need to know  the volume of the blocking region created by the existing spheres and then we need to generate a point uniformly on this non-blocking region; see line \ref{step:findingBc} in Algorithm~\ref{alg:Ref_IS}.
One possible way to implement the reference IS is by combining a well-known method called {\em power tessellation} and a simple rejection method in two steps: i) Using the power tessellation, we can compute the blocking volumes exactly; see, e.g, \cite{Aurenhammer87} and \cite{MH08}.  ii) Then, use a simple acceptance-rejection method where repeatedly a point  is generated independently and uniformly on $[0,1]^d$ until it falls within the non-blocking region. Unfortunately,  implementing  the power tessellation method is computationally prohibitive. Besides, even if we have an efficient implementation of the power tessellation method, the above simple rejection step can be expensive when the non-blocking region is small. 
Fortunately, we can overcome both these difficulties by using a simple grid on $[0,1]^d$. From \eqref{eqn:ISAR_cost}, it is evident that if there are two IS methods with the same $\ee[\wt \sigma_N]$, it is computationally preferable to use the method that has smaller per iteration expected running time, $\wt C_{\mathsf{itr}}(\lambda)$. In Subsection~\ref{sec:improv-const}, we introduce a hyper-cubic grid based IS method that continues to generate perfect samples  while the blocking regions are closely approximated by grid cells. With a careful choice of the cell-edge length, we make sure that the inequality \eqref{eqn:LHRatio_bdd} holds for the grid IS as well (and thus, $\ee[\wt \sigma_N]$ is same as that of the reference IS). As a consequence of Corollary~\ref{cor:no_of_itr}, the expected iterations of Algorithm~\ref{alg:ext_gen_stab} is the same as that of the reference IS method. 
However, the grid method is easy to implement and has a much smaller expected iteration cost $\wt C_{\mathsf{itr}}(\lambda)$ compared to that of the reference IS. The choice of the hyper-cubic grid is just an option as it simplifies the implementation. However, the method can be implemented using other kinds of grids. In two dimensional case, for example, it is possible to use a hexagonal grid for implementing the IS method.
\end{remark}

\subsection{Grid Based Importance Sampling for Fixed Radius Case}
\label{sec:improv-const}
Consider the hard-sphere model with a fixed radius $\rbdd/\lambda^\eta$. Generation of $n$ spheres {under the} following grid based IS measure $\wh \mu_n$ starts by partitioning the underlying space $[0,1]^d$ into a hyper-cubic grid with a cell-edge length $ \varepsilon > 0$  such that $1/\varepsilon$ is an integer. The centers of the spheres are generated in a sequential order: Suppose that $i-1$ spheres with centers $Y_1, \dots, Y_{i-1}$ are already generated. At the time of $i^{th}$ sphere generation, a cell $C$ in the grid  is labeled as {\em fully-blocked} if the cell is completely inside a sphere with radius $2\rbdd/\lambda^\eta$ centered at an existing point, that is, $C \subseteq S(Y_j, 2\rbdd/\lambda^\eta)$ for some $j \leq i-1$; otherwise, the cell is labeled as {\em non-fully-blocked}. A non-fully-blocked cell $C$ is called {\em partially-blocked} if  $C \cap S(Y_j, 2\rbdd/\lambda^\eta) \neq \varnothing$ for some $j \leq i-1$; otherwise, it is called {\em non-blocked}. The center $Y_i$ of the $i^{th}$ sphere is selected uniformly over the non-fully-blocked cells, because selecting $Y_i$ over a fully-blocked cell will certainly result in  the $i^{th}$ sphere overlapping with an existing sphere. We  then check for overlap only if $Y_i$ is generated over a partially-blocked cell, because the overlap is not possible if $Y_i$ is generated over a non-blocked cell.  If either there is an overlap or all the cells are fully-blocked by the existing spheres, the centers $Y_i, \dots, Y_n$ of the remaining spheres are selected arbitrarily (such a selection results in an overlapping configuration). Otherwise, for  the next sphere $i+1$ generation, we repeat the same procedure by relabeling the non-fully-blocked cells by considering spheres $1, \dots, i$ as the existing spheres. Note that at the beginning of each iteration all the cells are labeled as non-blocked. Also note that since all the spheres have the same radius, for relabeling of the cells, we  only need to focus on the cells that might interact with the  last sphere generated.  See Figure~\ref{pic:impSim} for an illustration of this sequential procedure. \\

 \begin{figure}[h]
\centering
\includegraphics[width=0.7\textwidth, height=0.7\textwidth]{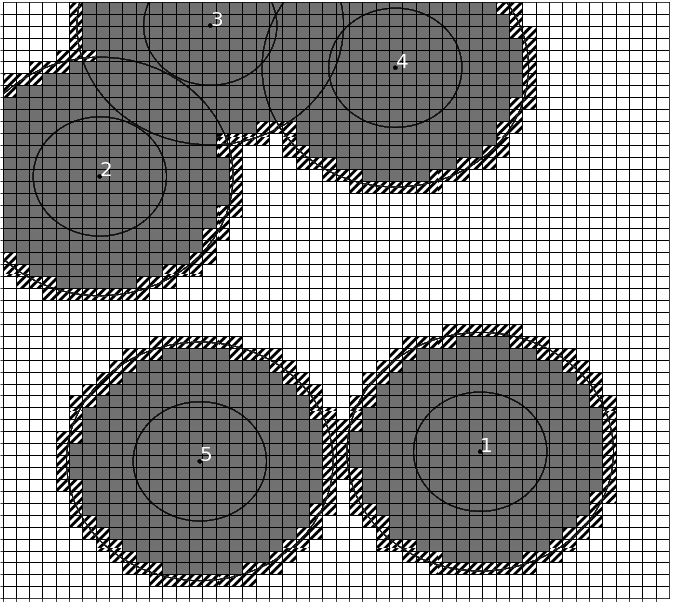}
\caption{A realization with 5 circles on the unit square $[0,1]^2$ generated using the grid based IS method for a Euclidean-hard-sphere model with a fixed radius (smaller circles). The grid size is $50\times 50$ and the radius is $0.1$. The bigger circle around each point is the actual region blocked by the circle.
For the $6^{th}$ circle generation, grey cells are fully-blocked, hatched cells are partially blocked, and white cells are non-blocked.}
\label{pic:impSim}
\end{figure}

Suppose that $\wh \mu_n$ is the probability measure under which $n$ spheres are generated by the above procedure. 
Then $\mu^0$ is absolutely continuous with respect to $\wh \mu_n $ on $ {\mathscr{G}}_n \cap \mathscr{A}$ and the corresponding likelihood ratio 
\begin{align*}
\wh L_n(\lfs_n)  := \frac{d\mu^0}{d\widehat \mu_n}(\lfs_n) = \prod_{i=1}^n\lt(1 -  \wh B_i\rt), \quad \lfs_n \in  {\mathscr{G}}_n \cap \mathscr{A},
\end{align*}
where $\wh B_i$ is the volume of fully-blocking cells for the $i^{th}$ sphere generation, that is, $\wh B_i$ equal to the product of the number of fully-blocked cells and $\varepsilon^d$. 
 To apply Algorithm~\ref{alg:ext_gen_stab} for the fixed radius hard-sphere model, take $K = 1$ and for each ${n \in \mbb{N}_0}$, take ${D_{n,1} =   {\mathscr{G}}_n}$, ${\mu_{n,1} = \wh \mu_n}$ and ${\sigma_{n,1} =\delta_n}$. Thus, ${\wt \sigma(n) = \delta_{n}}$ and ${L_{n,1}(\lfs_n) = \wh L_n(\lfs_n)}$ for all ${\lfs_n \in  {\mathscr{G}}_n \cap \mathscr{A}}$. \\

\noindent
{\bf Selection of the cell-edge length $\boldsymbol{\varepsilon}$:} Observe that the longest diagonal length of a cell is $\sqrt{d}\, \varepsilon$. Since we focus only on the non-overlapping configurations, in the implementation, we generate a sphere only if all the existing spheres are non-overlapping. Suppose that the cell-edge length $\varepsilon$ is selected so that 
$\displaystyle \sqrt{d}\, \varepsilon \leq \rbdd/\lambda^\eta.$
Then for the $i^{th}$ sphere generation, every cell that has non-empty intersection with $S(Y_j, \rbdd/\lambda^\eta)$, for any ${j =1, \dots, i-1}$, has to be fully-blocked, because such a cell is a subset of $S\lt(Y_j, 2\rbdd/\lambda^\eta\rt)$. Thus, the non-overlapping condition of the existing spheres imply that 
$\cup_{j = 1}^{i -1} S(Y_j, \rbdd/\lambda^\eta)$ is a subset of the union of the fully-blocked cells, and hence ${\wh B_i \geq \frac{\gamma' (i -1)\rbdd^d}{\lambda^{\eta d}}}.$
Thus, for $n \geq 1$,
\begin{align}
\label{eqn:lhr_grid_bdd}
 I(\lfs_n \in \mathscr{A})\wh L_n(\lfs_n) \leq \delta_{n} =  \prod_{i=1}^{n} \lt( 1 - (i-1) \frac{\gamma' \rbdd^d}{ \lambda^{\eta d}}  \rt)^+,\quad \lfs_n \in  {\mathscr{G}}_n.
\end{align}
This upper bound is same as that we obtained in the case of the reference IS; see the inequality \eqref{eqn:LHRatio_bdd}.  Since the acceptance probability ${P_{\mathsf{acc}}(\lambda)= \pno/\ee[\delta_N]}$ is the same for both the grid IS and the reference IS methods, we need to choose the cell-edge length $\varepsilon \leq \rbdd/\lambda^\eta$ so that the expected per iteration running time $\wt C_{\mathsf{itr}}(\lambda)$ is minimum. It is easy to see that the higher the value of $\varepsilon$, the smaller $\wt C_{\mathsf{itr}}(\lambda)$ due to the following reasons:
\begin{enumerate}
\item Labelling of the cells is faster if they are bigger in size;
\item Increase in the cell size increases the chances of overlap of the new sphere with the existing spheres, and hence on average each iteration generates fewer spheres; 
\end{enumerate}
In conclusion, we choose $\varepsilon = 1/ \lfloor \lambda^\eta/\rbdd\rfloor$ for the implementation of the grid IS method. \\

To reduce the per iteration complexity of the algorithm, we make some changes to the steps \ref{line:X_M} and \ref{line:bern} in Algorithm~\ref{alg:ext_gen_stab}. Observe that a realization $\state_n$ generated under $\wh \mu_n$ is accepted only if ${\state_n \in \mathscr{A}}$ and ${J =  1}$, where ${J \sim \bern\lt(\wh L_n(\state_n)/\delta_n\rt)}$. In the implementation, we generate an iid sequence $U_1, \dots, U_n \sim \unif(0,1)$ independent of everything else so far generated, and take
$J_i = I\lt( U_i \leq \frac{1 - \wh B_i}{(1 - (i-1) \gamma' \rbdd^d\lambda^{-\eta d})}\rt)$ for $i \leq n$. Since $J$ and the product $\prod_{i = 1}^n J_i$ are Bernoulli random variables with the same success probability $L(\state_n)/\delta_n$, 
to reduce the per iteration cost, we generate the $i^{th}$ sphere only if $J_i = 1$ and the existing spheres do not overlap with each other.\\

Algorithm~\ref{alg:Grid_IS} implements the grid based IS for a given $n$ with the above mentioned enhancements. 
Algorithm~\ref{alg:ext_gen_stab} is restated as Algorithm~\ref{alg:ext11}.
\begin{algorithm}[H]
  \caption{Grid Based Importance Sampling for Fixed Radius}
  \label{alg:Grid_IS}
  \begin{algorithmic}[1]
  \State{{\bf Input:} The total number of spheres $n \geq 1$ and a grid on $[0,1]^d$} 
  \State{{\bf Output:} $(\state, \mathsf{Status}) \in \mathscr{G} \times \{\mathsf{True}, \mathsf{False}\}$. Where $\mathsf{Status} = \mathsf{True}$ if $\state \in \mathscr{A}$ and $\mathsf{Status} = \mathsf{False}$ otherwise}
  \State{Label every cell as non-blocked}
%
  \State{$\state \leftarrow \varnothing, i \leftarrow 0$ and $\wh B \leftarrow 0$}
  \Repeat
    \State{$i \leftarrow i + 1$}
    \State{Generate $U \sim \unif(0,1)$}    
    \If { $U > \frac{1 - \wh B}{1 - (i-1) \gamma' \rbdd^d\lambda^{-\eta d} }$} 
            \State{\Return $(\state, \mathsf{False})$}
    \Else
         \State{Generate $Y_i$ uniformly distributed over the non-fully-blocked cells}
         \Statex{\hspace{1.1cm} (and independently of everything else so far generated)}
         \If{$Y_i$ is on a partially-blocked cell and there is an overlap} 
            \State{ \Return $(\state, \mathsf{False})$}
      	\EndIf
       \State{Update the cell labels}
       \State{Compute the volume $\wh B$ of the fully-blocked cells}      	
    \EndIf
   \State{$\state \leftarrow \state \cup \{(Y_i, \rbdd/\lambda^\eta)\}$}
    \Until{$i = n $}\\    
    \Return{$(\state, \mathsf{True})$}
  \end{algorithmic}
\end{algorithm}

\begin{algorithm}
  \caption{Perfect Sampling for hard-sphere model using Grid Based IS}
  \label{alg:ext11}
  \begin{algorithmic}[1]
    \State{Partition $[0,1]^d$ into a hypercube grid with cell-edge length $\varepsilon = 1/ \lfloor \lambda^\eta/\rbdd\rfloor$}
  \Repeat
    \State{Generate a sample of $M$ with pmf \eqref{eqn:dist_M_gen_stab}}
    \If {$M = 0$}
    	\State{$(\state, \mathsf{Status}) \leftarrow (\varnothing, \mathsf{True})$}
    \Else
	    \State{Obtain  an output $(\state, \mathsf{Status})$ from Algorithm~\ref{alg:Grid_IS} with $M$ and the grid as input}
    \EndIf
    \Until{$\mathsf{Status}  = \mathsf{True}$}\\
    \Return{$\state$}
  \end{algorithmic}
\end{algorithm}

\begin{remark}[The pmf of $M$]
\normalfont
Note that, for the current setup, the pmf of $M$, given by \eqref{eqn:dist_M_gen_stab}, becomes $\pp\lt(M = m\rt) = \frac{1}{C_{\lambda}}\frac{\lambda^m \delta_m}{m!}$, $m \in \mbb{N}_0$, where  the normalizing constant $C_{\lambda} = \sum_{n \in \mbb{N}_0} \frac{\lambda^n \delta_n}{n!}$. The support of the pmf is finite because $\delta_m = 0$ for all $m \geq \lambda^{\eta d}/(\gamma' \rbdd^d) + 1$. To increase the performance of the algorithm, we can further truncate the support of the pmf. Using the maximum packing density, we can obtain an integer~$m_{\max}$ such that $\state \notin \mathscr{A}$ for all $m \geq m_{\max}$ and configurations $\state$ with $|\state| = m$. In that case, we can take  $\pp\lt(M = m\rt) = \frac{1}{C_{\lambda}}\frac{\lambda^m \delta_m}{m!}$, $0 \leq m \leq m_{\max}$, with $C_\lambda = \sum_{n = 0}^{m_{\max}} \frac{\lambda^n \delta_n}{n!}$. For example, refer to \cite{MMSWD01} for finding maximum packing densities for $d = 2$ and $d = 3$. 
\end{remark}

We now focus on the expected running time analysis of Algorithm~\ref{alg:ext11}. By Proposition~\ref{prop:P_acc_IS}, the acceptance probability 
$P_{\mathsf{acc}}(\lambda)$ of Algorithm~\ref{alg:ext11} is 
$\pno/\ee[\wt \sigma(N)] = \pno/\ee[\delta_N].$
A proof of Proposition~\ref{prop:imp_samp} is given in Section~\ref{sec:proof_imp_samp}.
\begin{proposition}
\label{prop:imp_samp}
 For the fixed radius hard-sphere model, there exists a constant $c > 0$  such that
 \begin{align}
 \label{eqn:ISAR_cost_grid}
  \mathcal{T}_{\mathsf{ISAR}} \leq c \, \ee\lt[ \delta_N\rt] \frac{\lambda^{\min\{\eta d, 1\}}}{\pno},
 \end{align}
where $N \sim \pois(\lambda)$.
Furthermore,
 \begin{align*}
  \limsup_{\lambda \nearrow \infty }\lt[ \frac{1}{\lambda^{2 - \eta d}} \log \ee\lt[\delta_N\rt] \rt] &\leq  - \frac{\gamma' \rbdd^d}{2}, \,\, \text{ if }\, \eta d > 1, \,\, \text{ and} \\
  \limsup_{\lambda \nearrow \infty }\lt[ \frac{1}{\lambda} \log \ee\lt[\delta_N\rt] \rt] & \leq - b,  \, \text{ if }\, 0 < \eta d \leq 1, \text{ for some constant }\, b > 0.
 \end{align*}
\end{proposition}
The following result is a trivial consequence of Propositions~\ref{prop:AR_meth} and \ref{prop:imp_samp}.

\begin{corollary}
\label{cor:TIS_vs_TAR}
For the fixed radius hard-sphere model, if $\eta d \geq 2$, both $\mathcal{T}_{\mathsf{ISAR}}$ and  $\comAR$ are of the same order, and  if $0 < \eta d < 2$, there exists a constant $c > 0$ such that
 $\displaystyle\mathcal{T}_{\mathsf{ISAR}} \leq c\,  \ee\lt[\delta_N\rt] \comAR.$
\end{corollary}

\begin{remark}[Better choice of $\delta_n$ for the Euclidean-hard-sphere model]
\label{rem:ChSigOneDim}
\normalfont
If {the spheres} are Euclidean, further improvements in the choice of $\delta_n$ can be obtained by accounting for  boundary effects. For instance, for $d=2$, the four corners of $[0,1]^2$ are covered by at most $4$ circles, each of which contributing a blocking area of at least $\gamma'\rbdd^2/\lambda^{2 \eta} = \pi\rbdd^2/4\lambda^{2 \eta}$, while each of the remaining circles contributing a blocking area of at least $2\gamma'\rbdd^2/\lambda^{2 \eta} = \pi\rbdd^2/2\lambda^{2 \eta}$. Let  $b_0 = 0$, $b_i =   (i - 1)\frac{\pi \rbdd^2}{4 \lambda^{2\eta}}$ for $1\leq i \leq 5,$ and $b_i = \frac{\pi \rbdd^2}{\lambda^{2\eta}} + (i-4)\frac{\pi\rbdd^2}{2\lambda^{2 \eta}}$ for $i \geq 6$.
Then, for this particular scenario, a better choice of $\delta_n$ in \eqref{eqn:LHRatio_bdd} (as well as in \eqref{eqn:lhr_grid_bdd}) is 
$\delta_n = \prod_{i=1}^{n} \lt( 1 - b_n  \rt)^+,\, n \in \mbb{N}_0$.
\end{remark}

\subsection{Random Radii Case}
\label{sec:improv-random}
We now consider another application of Algorithm~\ref{alg:ext_gen_stab} for the hard-sphere model {when under the marked PPP the radii of the spheres are iid.}
For the fixed radius case presented in Section~\ref{sec:improv-const}, the proposed IS method ensured a uniform bound $\delta_n$
on the likelihood ratio over ${\mathscr{G}}_n$ for every $n \in \mbb{N}_0$, as shown in \eqref{eqn:lhr_grid_bdd}. Such upper bounds are possible for a random radii hard-sphere model if the radii are bounded below by a positive constant. Furthermore, a similar analysis can be established when the spheres are replaced with iid convex shapes such that each shape occupies a minimum positive volume. However, when the radii are not bounded from below almost surely, the associated blocking volumes can be arbitrarily small.
We address this issue by partitioning ${\mathscr{G}}_n$ into two sets $D_{n,1}$ and $ D_{n,2}$ for each $n$ so that the IS on $D_{n,1}$ is a grid based IS method that is similar to Algorithm~\ref{alg:Grid_IS} and the IS on $D_{n,2}$ is obtained by {exponentially twisting the distribution of $R^d$ to put high probability mass on configurations with lower volume spheres.}\\

We first introduce the exponential twisting of the distribution, say $G$, of $R^d$. Recall that $R$ is assumed to be a bounded non-negative random variable. Without loss of generality further assume that ${\alpha := \ee[R^d] > 0}$. 
Thus the {\it logarithmic moment generating function} of $R^d$ defined  by 
${\Lambda(\theta) := \log\lt(\ee\lt[\exp({\theta R^d}) \rt] \rt)}$ is finite for every $\theta \in \reals$. 
Furthermore, the derivative 
\[
\Lambda'(\theta) = \frac{\dd \Lambda(\theta)}{\dd \theta} = \frac{\ee\lt[R^d\exp({\theta R^d}) \rt] }{\ee\lt[\exp({\theta R^d}) \rt]}
\]
is finite and positive for all $\theta \in \reals$ and in particular, ${\Lambda'(0) = \alpha}$. In fact, using the results in Chapter~2 of \cite{DZ10}, it can be seen that 
$\Lambda(\theta)$ is strictly convex. As a consequence, $\Lambda'(\theta)$ is strictly increasing and hence
\[
\alpha_{\min} := \lim_{\theta \to -\infty} \Lambda'(\theta) < \alpha.
\]
Let $\wh \theta$ be such that 
${\Lambda'(\wh \theta) = \varrho}$ for some ${\varrho \in (\alpha_{\min}, \alpha)}$. 
Therefore, ${\wh \theta < 0}$. Now consider the distribution  $\wt G$ obtained by  exponentially twisting $G$ by the amount $\wh \theta$, that is, 
${\dd \wt G(t)} = \exp\lt( \wh \theta t - \Lambda(\wh \theta)\rt)\dd G(t).$
Fix a constant $a \in (0,1)$ and for each integer $n \geq 1$, define
\[
 H_n := \lt\{ (t_1, t_2, \dots, t_{\lceil n a \rceil}) \in \reals_+^{\lceil n a \rceil}: \frac{1}{\lceil n a \rceil}\sum_{i= 1}^{\lceil n a \rceil} t_i < \varrho \rt\}.
\]
We later  see that $a = 1/2$ is a good choice for increasing performance of the algorithm.
Let $\Lambda^*(\cdot)$ be the Legendre-Fenchel transform of $\Lambda$, that is, ${\Lambda^*(t) = \sup_{\theta \in \reals} \{\theta t -  \Lambda(\theta)\}}$. {This corresponds} to the large deviations rate function
associated with the empirical average of iid samples from~$G$. 
From the definition of $\wh \theta$ and the fact that $\Lambda(\theta)$ is strictly convex, {${\Lambda^*(\varrho) = \wh \theta  \varrho - \Lambda(\wh \theta) >0}$}. 
Since $\wh \theta < 0$, {for all ${(t_1, t_2, \dots, t_{\lceil n a \rceil}) \in H_n}$},
\begin{align*}
\exp\lt( \wh \theta \sum_{i=1}^{\lceil n a \rceil} t_i - \lceil n a \rceil \Lambda(\wh \theta)\rt)  &= \exp\lt( \wh \theta \sum_{i=1}^{\lceil n a \rceil} (t_i - \varrho) + \lceil n a \rceil\, \Lambda^*(\varrho)\rt)\\ &\geq \exp\lt( \lceil n a \rceil\, \Lambda^*(\varrho)\rt),
\end{align*}
and  thus,
\begin{align}
\label{eqn:LHRatio_ET}
 \prod_{i=1}^{\lceil n a \rceil} \frac{ \dd G}{\dd \wt G}( t_i) \leq \exp\lt( - \lceil n a \rceil\, \Lambda^*(\varrho)\rt) \leq \exp\lt( - n a \, \Lambda^*(\varrho)\rt).
\end{align}

Recall the definition of the distribution $\mu$ of the hard-sphere model given by \eqref{eqn:dist_mu_HS}.
To apply Algorithm~\ref{alg:ext_gen_stab}, select $K = 2$ and define
\begin{align*}
 D_{n,1} = \lt\{ \lfs_n = \{(x_1, r_1/\lambda^{\eta d}), \dots, (x_n, r_n/\lambda^{\eta d}) \} \in \mathscr{G}_n: (r_1^d, \dots, r^d_{\lceil n a \rceil}) \in H^{\mathsf{c}}_n\rt\},
\end{align*}
and $D_{n,2} =  {\mathscr{G}}_n\setminus D_{n,1},$
 for each $n$ where $H^{\mathsf{c}}_n$ is the complement of $H_n$ within $[0, \rbdd]^{\lceil n a \rceil}$. To apply Algorithm~\ref{alg:ext11}, we are now left with identifying the IS measures $\mu_{n,1}$ and $\mu_{n,2}$, and the corresponding bounds $\sigma_{n,1}$ and $\sigma_{n,2}$ for each ${n \in \mbb{N}_0}$. \\

The measure $\mu_{n,1}$  on $ D_{n,1}$ is again a grid based IS method similar to the grid method introduced for the fixed radius case in Subsection~\ref{sec:improv-const}. First, iid copies $\{R_1, \dots, R_n\}$ of $R$ are generated. Then, we construct a new grid and label each cell every time a new sphere is generated as follows. For the generation of the $i^{th}$ sphere with radius $R_i/\lambda^\eta$, we take the cell-edge length ${\varepsilon = 1/ \lceil \lambda^\eta/R_i\rceil}$. A cell $C$ in the grid is labeled as fully-blocked if ${C \subseteq S(X_j, (R_j + R_i)/\lambda^\eta)}$ for an existing sphere ${j \leq i-1}$ with the center $X_j$ and the radius $R_j/\lambda^\eta$; otherwise, the cell is labeled as non-fully-blocked. A non-fully-blocked cell $C$ is called partially blocked if ${C \cap S(X_j, (R_j + R_i)/\lambda^\eta) \neq \varnothing}$ for some ${j \leq i-1}$; otherwise, it is called non-blocked. Then the next center $X_i$ is generated uniformly over the non-fully-blocking cells. Just like in the case of fixed radius, $X_1$ is generated uniformly over $[0,1]^d$ and we check the possibility of the overlap of $i^{th}$ sphere with an existing sphere only if $X_i$ falls on a partially-blocked cell.\\

       The measure $\mu^0$ is absolutely continuous with respect to $\mu_{n,1}$ on ${D_{n,1} \cap \mathscr{A}}$ and the associated likelihood ratio $L_{n,1}$ is given by  
       $$L_{n,1}(\lfs_n) =  \prod_{i=1}^{n}\Big(1 - B_{i}\Big), \quad \lfs_n \in D_{n,1} \cap \mathscr{A},$$
       where $\wh B_i$ is the volume of all the fully-blocked cells for the $i^{th}$ sphere generation. By \eqref{eqn:bni} and the fact that the cell-edge length is $1/ \lceil \lambda^\eta/R_i\rceil$, we have
       ${B_{i} \geq \min\lt(1, \frac{\gamma' \lceil n a \rceil}{ \lambda^{\eta d}} \varrho\rt)}$ on ${D_{n,1} \cap \mathscr{A}}$ for all ${i \geq \lceil n a \rceil + 1}$ because
       $\frac{1}{\lceil n a \rceil}\sum_{j =1}^{\lceil n a \rceil} R_j^d \geq \varrho$ over the set $H^{\mathsf{c}}_n$. Consequently, 
       \[
       I(\lfs_n \in \mathscr{A}) L_{n,1}(\lfs_n) \leq \lt[\lt(1 - \frac{\gamma' \lceil n a \rceil}{ \lambda^{\eta d}} \varrho \rt)^+\rt]^{n(1 - a)} =: \sigma_{n, 1}, \quad \lfs_n \in D_{n,1}.
       \]

      The measure $\mu_{n,2}$ is induced by the following procedure: Generate iid samples $R^d_1, \dots, R^d_{\lceil n a \rceil}$ from $\wt G$, and
       independently of this, generate iid samples $R^d_{\lceil n a \rceil + 1}, \dots, R^d_n$ from $G$. For $i =1, \dots, n$, the radius of the $i^{th}$ sphere is $R_i/\lambda^d$ and the center generated uniformly distributed over the non-blocking region created by the existing $i-1$ spheres. 
       Since $R^d_1, \dots, R^d_{\lceil n a \rceil}$ are sampled from $\wt G$, by \eqref{eqn:LHRatio_ET},
       \[
       I(\lfs_n \in \mathscr{A})L_{n,2}(\lfs_n) = \prod_{i=1}^{\lceil n a \rceil} \frac{\dd G}{\dd\wt G}(r^d_i) \leq \exp\lt( - n a \, \Lambda^*(\varrho)\rt) =: \sigma_{n,2}, 
       \text{  for all }\,\lfs_n \in D_{n,2}.
       \]
       
In summary,  $\lt\{ \lt(D_{n,k}, \mu_{n,k}, \sigma_{n,k}\rt)_{k = 1}^2\rt\}_{n \in \mbb{N}_0}$ is a stable IS sequence, and hence Algorithm~\ref{alg:ext_gen_stab} generates perfect samples from $\mu$. However, to reduce the per iteration complexity (as in the fixed radius case), we make some modification to the algorithm. Algorithm~\ref{alg:Grid_IS_RR} is similar to Algorithm~\ref{alg:Grid_IS} and Algorithm~\ref{alg:ext_gen_stab} is restated as Algorithm~\ref{alg:ext_gen_stab_RR}. 
  \begin{algorithm}[h]
  \caption{Grid Based Importance Sampling for Random Radii Case}
  \label{alg:Grid_IS_RR}
  \begin{algorithmic}[1]
  \State{{\bf Input:} The total number of spheres $n \geq 1$} 
  \State{{\bf Output:} $(\state, \mathsf{Status}) \in \mathscr{G} \times \{\mathsf{True}, \mathsf{False}\}$. Where $\mathsf{Status} = \mathsf{True}$ if $\state \in \mathscr{A}$ and $\mathsf{Status} = \mathsf{False}$ otherwise}
  \State{$i \leftarrow 0$}
  \State{$\state \leftarrow \varnothing$}   
  \Repeat
    \State{$i \leftarrow i + 1$}
    \State{Generate a copy $R_i$ of $R$ independently of everything else so far generated}
    \State{Construct a grid on $[0,1]^d$ with the cell-edge length $\varepsilon = 1/ \lceil \lambda^\eta/R_i\rceil$}
    \State{Identify the label of each cell in the new grid}   
    \State{Compute the volume $\wh B$ of the fully-blocked cells and generate $U \sim \unif(0,1)$}
    \If { $U > \frac{1 - \wh B}{\lt(1 - \gamma' \lceil n a \rceil \varrho  \lambda^{ -\eta d}  \rt)^{1- a} }$} 
            \State{\Return $(\state, \mathsf{False})$}                     
    \Else          
         \State{Generate $Y_i$ uniformly distributed over the non-fully-blocked cells}
         \Statex{\hspace{1.1cm} (and independently of everything else so far generated)}          
         \If{$Y_i$ is on a partially-blocked cell and there is an overlap} 
            \State{ \Return $(\state, \mathsf{False})$} 
	\EndIf
    \EndIf
    \State{$\state \leftarrow \state \cup \{(Y_i, R_i/\lambda^\eta)\}$}
  \Until{$i = n$}\\
\Return{$(\state, \mathsf{True})$}
    
  \end{algorithmic}
\end{algorithm}

\begin{algorithm}[H]
  \caption{Perfect Sampling for hard-sphere model with Random Radii}
  \label{alg:ext_gen_stab_RR}
  \begin{algorithmic}[1]
  \Repeat   
    \State{Generate a sample of $M$ with pmf \eqref{eqn:dist_M_gen_stab}}   
    \State{Generate $J$ with pmf $\pp(J = k) = \sigma_{M,k}/\wt \sigma(M)$, $k = 1, 2$}
    \If {$J  = 1$}    
    \State{Obtain an output $(\state, \mathsf{Status})$ of  Algorithm~\ref{alg:Grid_IS_RR} with $M$ as input}       
    \Else
    \State{Generate $\state$ under $\mu_{M,2}$} 
    \If {$ \bern\lt(\frac{L_{M, 2}(\state) I\lt(\state \in D_{M,2} \cap \mathscr{A}\rt)}{\sigma_{M,2}}\rt) = 0$} 
    \State{$\mathsf{Status}  \leftarrow  \mathsf{False}$}
%
    \EndIf
    \EndIf    
    \Until {$\mathsf{Status} = \mathsf{True}$}\\
    \Return{$\state$}
  \end{algorithmic}
\end{algorithm}

We now focus on the running time complexity of Algorithm~\ref{alg:ext_gen_stab_RR}. Notice that $\wt \sigma(n) = \sigma_{n,1} + \sigma_{n,2}$ for each $n \in \mbb{N}_0$.
By Proposition~\ref{prop:P_acc_IS}, $P_{\mathsf{acc}}(\lambda) = \pno/\ee\lt[\wt \sigma(N)\rt]$ with $N \sim Poi(\lambda)$.
Observe that $\sigma_{n,1} \leq \exp\lt( - \frac{\gamma' n^2\, a(1 - a)}{\lambda^{\eta d}} \varrho \rt)$.
The proof of Proposition~\ref{prop:imp_samp} can be extended to {the current} scenario to show that
\begin{align*}
  \limsup_{\lambda \nearrow \infty }\lt[ \frac{1}{\lambda^{2 - \eta d}} \log \ee\lt[ \sigma_{N,1}\rt] \rt] &\leq  - \gamma' \, a(1 - a)\varrho,  \,\, \text{ if }\, \eta d > 1, \,\, \text{ and}\\
  \limsup_{\lambda \nearrow \infty }\lt[ \frac{1}{\lambda} \log \ee\lt[\sigma_{N,1}\rt] \rt] & \leq - b,  \, \text{ if }\, 0 < \eta d \leq 1, \text{ for some constant }\, b > 0.
 \end{align*}
It is now clear that a good choice for $a$ is $1/2$ because it maximizes $a(1 - a)$.
Furthermore, using the moment generating function of Poisson random variables, we have
\[
 \ee\lt[ \sigma_{N,2}\rt] \leq \exp\lt( - \lambda \lt(1 - e^{ -\Lambda^*(\varrho)/2} \rt)\rt).
\]
Recall that ${\mathcal{T}_{\mathsf{ISAR}} \leq  \ee\lt[\wt  \sigma(N)\rt] \wt C_{\mathsf{itr}}(\lambda)/\pno}$. The per iteration complexity $\wt C_{\mathsf{itr}}(\lambda)$ mainly determined by relabelling of cells in the new grid for each sphere generation. 
The grid size for the $i^{th}$ sphere generation is an order of $\lambda^{\eta d}/R_i^d$ and the total number of spheres generated in each iteration is at most an order of $\lambda^{\min\{\eta d, 1\}}$. Therefore, for $\eta d \leq 2$, we can show that $\wt C_{\mathsf{itr}}(\lambda)$ is of order $\lambda^{\min\{\eta d, 1\}}\ee[1/R^d]$.  

\begin{remark}
\normalfont
If $\varrho$ is selected to equal
$\argmin_{\varrho \in (\alpha_{\min}, \alpha)} \,\, (\sigma_{n,1} + \sigma_{n,2})$ for each $n =  1, 2, \dots$, then $\ee\lt[\wt  \sigma(N)\rt]$ is minimum.  Note that   $\sigma_{n,1}$ decreases and $\sigma_{n,2}$ increases as functions of $\varrho$.
The above decompositions were chosen to illustrate ideas simply. More complex decompositions are easily created for further performance improvement. For instance, we could have defined $H_n$ above as
$$H^{\mathsf{c}}_n := \lt\{ (r_1, \dots, r_{n}) \in \reals_+^{n}: \frac{1}{m}\sum_{i= 1}^{m} r_i \geq \varrho_m, \,\forall m \leq n \rt\},$$
and then arrived at appropriate $\{\varrho_m\}_{m \leq n}$ and appropriate changes of measures
for configurations in  $H_n$ and  $H_n^{\mathsf{c}}$. While this should lead to substantial performance improvement,
it also significantly  complicates the analysis.
\end{remark}

\section{Dominated CFTP Methods}
\label{sec:dcftp}
In this section, we review some of the well-known dominated CFTP algorithms for the hard-sphere models. 
We refer to  \cite{KM00} for a general description of  the dominated CFTP for Gibbs point processes (this method was first proposed for area-interaction processes by Kendall \cite{KW98}).\\

\newcommand{\mD}{\mathbf{D}}
\newcommand{\mZ}{\mathbf{Z}}
\newcommand{\mU}{\mathbf{U}}
\newcommand{\mL}{\mathbf{L}}
Let $\mD = \{\mD(t): t \in \reals\}$ be the so-called dominating birth-and-death process on $[0,1]^d$ with births arriving as a Poisson process with rate $\lambda$, where each birth is a uniformly and independently generated marked point on $[0,1]^d$ that denoting the center of a sphere with the mark being its radius. Each birth is alive for an independent random time exponentially distributed with mean one. It is well known that the steady-state distribution of $\mD$ is $\mu^0$. 
Furthermore, it is easy to generate the dominating process $\mD$ both forward and backward in time so that ${\mD(t) \sim \mu^0}$ for all $t$. To see this, let ${\dots < t_{-2} < t_{-1} < 0 < t_1 < t_2 < \dots}$ be the event instants of the process $\mD$, where an event can be either a birth or a death. Assume that with each birth there is an additional mean one exponentially distributed independent mark to determine its life time. Since the births are arriving as a Poisson process, the interarrival times are exponential with mean $1/\lambda$. Generate ${\mD(0) \sim \mu^0}$, determine the next event instant $t_1$ and take ${\mD(t) = \mD(0)}$ for ${0\leq t < t_1}$. If the next event is a birth, generate a new independent (marked) point; otherwise, remove the existing point with the smallest lifetime. Continue the same procedure starting with $\mD(t_1)$ to generate the process over $[t_1, t_2)$, and so on.
For generating the dominating process $\mD$ backward in time, observe that $\mD$ is time-reversible, and hence we can generate
$\{\mD(t): -T \leq t \leq 0\}$ for any finite ${T > 0}$ just by generating an independent copy $\{\wt{\mD}(t):\, 0 \leq t \leq T\}$ of the dominating process $\left\{\mD(t):\, 0 \leq  t \leq T\right\}$
and taking ${\mD(-t) = \wt{\mD}(t)}$ for ${0 \leq t \leq T}$. \\

Since the distribution $\mu$ of the hard-sphere model is absolutely continuous with respect to $\mu^0$, using {\it coupling}, it is possible to construct a spatial birth-and-death process ${\mZ = \{\mZ(t) : t \in \reals\}}$, called the {\it interaction} process, such that ${\mZ(t) \subseteq \mD(t)}$ and ${\mZ(t) \sim \mu}$ for all ${t \in \reals}$; see \cite{KM00}. Each iteration of any dominated CFTP method essentially involves the following two steps: 
\begin{enumerate}
\item Fix $n > 0$ and  construct ${\{\mD(t): t_{-n} \leq t \leq 0\}}$ backward in time starting at time zero with ${\mD(0) \sim \mu^0}$
\item Then, as detailed in Sections~\ref{sec:kendall}-\ref{sec:huber}, use {\em thinning} on the dominating process $\{\mD(t): t_{-n} \leq t \leq 0\}$ to obtain an upper bounding process ${\{\mU_n(t) : t \geq t_{-n}  \}}$ with ${\mU_n(t_{-n}) = D(t_{-n})}$ and a lower bounding process ${\{\mL_n(t_{-n}) : t \geq t_{-n} \}}$ with ${\mL_n(t_{-n}) = \varnothing}$ forward in time such that for ${t \geq t_{-n}}$, ${\mL_n(t)\subseteq \mZ(t) \subseteq \mU_n(t) \subseteq \mD(t)}$ and ${\mL_m(t)\subseteq \mL_n(t) \subseteq \mU_n(t) \subseteq \mU_m(t)}$ for $m \leq n$.
\end{enumerate}
 If $\mU_n$ and $\mL_n$ {\em coalescence} at time $0$, that is, ${\mU_n(0) = \mL_n(0)}$, then  $\mU_n(0)$ is a perfect sample from the target distribution $\mu$. If there is no coalescence, then repeat the steps by increasing~$n$ and extending the dominating process ${\{\mD(t): -t_{-n} \leq t \leq 0\}}$ further backward to time $t_{-n}$ and repeat the same procedure. It is well known that a good strategy for increasing $n$ is doubling it after every iteration.  The criteria for thinning depends on the coupling used for constructing $\mZ$. However, the dominating process $\mD$ depends only on $\lambda$. In summary, a dominated CFTP algorithm is described by Algorithm~\ref{alg:kendall}.
\begin{algorithm}
  \caption{Dominated CFTP}
  \label{alg:kendall}
  \begin{algorithmic}[1]
    \State{Generate ${\{\mD(t) : t_{-1} \leq t \leq 0\}}$ with ${\mD(0) \sim \mu^0}$}
    \State{$n \leftarrow 1$}
    \Repeat
        \State{$n \leftarrow 2*n$}
    \State{Extend $\mD$ backwards from ${\{\mD(t): t_{-n/2} \leq t \leq 0\}}$ to ${\{\mD(t): t_{-n} \leq t \leq 0\}}$ \label{step:3_kendall}}
    \State{Construct $\left\{ \mL_n(t): t_{-n} \leq t \leq 0\right\}$ and $\left\{ \mU_n(t): t_{-n} \leq t \leq 0\right\}$}
    \Until {$\mL_n(0) = \mU_n(0)$}\\
    \Return {$\mL_n(0)$}
  \end{algorithmic}
\end{algorithm}

Consider the {\it backward coalescence} time {${N^* = \min\left\{n \in \mbb{N}_0: L_n(0) = U_n(0) \right\}}$.}
The average running time complexity of Algorithm~\ref{alg:kendall} depends on the number of operations involved within $N^*$, which further depends on the construction of the interaction process and the bounding  processes. 
At each iteration, the length of the dominating process $\mD$ is doubled on average backwards in time. Hence, on average the running time complexity doubles at each iteration.
From the definition of $N^*$, the length of the last iteration is $2^{\lceil\log_2 N^*\rceil} \geq N^*$.
Let ${N^f = \min\lt\{n \in \mbb{N}_0: \mL_0(t_n) = \mU_0(t_n) \rt\}}$ be the {\it forward} coalescence time.
Due to the reversibility of the dominating process, it can be shown that $N^*$ and $N^f$ are identical in distribution \cite{AG07},
and hence the expected computational effort for constructing the dominating, upper bounding  and lower bounding processes up to the forward coalescence time $N^f$,
starting from time $0$, is a lower bound on the expected running time of the algorithm.\\

Below we consider three dominated CFTP methods applicable for the hard-sphere models.

\subsection{Method 1}
\label{sec:kendall}
This method is based on \cite{KM00}. 
Note that the {\em Papangelou conditional intensity} of the hard-sphere model is given by
\begin{align}
\label{eqn:Papangelou_HS}
 \ell(\lfs, x) := \frac{I\lt(\lfs\cup\{x\} \in \mathscr{A}\rt) }{I\lt(\lfs \in \mathscr{A}\rt)} = I\lt(\lfs\cup\{x\} \in \mathscr{A}\rt),
\end{align}
with the convention that $0/0 = 0$. The interaction process $\mZ = \{\mZ(t): t \in (-\infty, \infty)\}$ is constructed as follows: Suppose $x$ is a birth to $\mD$ that sees $\mZ$ in
a state $\lfs \in  {\mathscr{G}}$.  Then $x$ is added to $\mZ$ if and only if $\ell(\lfs, x) = 1$.
Every death in $\mD$ {\em reflects} in $\mZ$, that is, if there is a death of a point $y$ in $\mD$, then $y$ is removed from the process $\mZ$ as well if it is present.
It can be shown that $\mu$ is the unique invariant  probability measure of $\mZ$; see, e.g., \cite{GNL00} or \cite{FFG02}.\\

For each $n \geq 1$, the bounding processes are constructed as follows: 
As mentioned earlier take ${\mL_n(t_{-n}) = \varnothing}$ and ${\mU_n(t_{-n}) = \mD(t_{-n})}$. Suppose that ${\lfs^l = \mL_n(t_{i})}$  and ${\lfs^u = \mU_n(t_{i})}$ for ${-n \leq i <  0},$
then assign $\mL_n(t) = \lfs^l$  and $\mU_n(t) = \lfs^u$ for $t_{i} < t < t_{i+1}$. In case it is a birth $x$ in
the dominating process $\mD$ at time $t_{i+1}$, set $\mL_n(t_{i+1}) = \lfs^l\cup\{ x\}$ if $ \ell(\lfs^u,x) = 1$;
otherwise, it will remain unchanged, that is, $\mL_n(t_{i+1}) = \lfs^l$. Similarly, set $\mU_n(t_{i+1}) = \lfs^u \cup \{ x\}$ if $\ell(\lfs^l,x) = 1$;
otherwise, set $\mU_n(t_{i+1}) = \lfs^u$. Every death in the dominating process reflects in both lower and upper bound processes.
Note that a birth is accepted by the lower bounding  process if the resulting state of the upper bounding  process is in $\mathscr{A}$. Similarly, a birth in $\mD$ is accepted in the upper bounding  process if the resulting state of the lower bounding  process is in $\mathscr{A}$.

\begin{theorem}
\label{thm:loss-system}
 The expected running time complexity $\comDCone$ of the above dominated CFTP algorithm satisfies
\begin{align}
\label{eqn:TDC_vs_TAR}
\comDCone \geq c\, \frac{\lambda}{\pno},
\end{align}
for some constant $c > 0$. In particular,
\[
 \comDCone = \begin{cases}
 \varOmega\Big(\lambda\Big), & \quad \text{ if }\, \eta d \geq 2,\\
 \varOmega\lt(\lambda\exp\left(\Big(\frac{\gamma \mu_1}{2} +o(1)\Big)\lambda^{2 - \eta d }\right)\rt), & \quad \text{ if }\, 1 < \eta d < 2,\\
 \varOmega\lt(\lambda\exp \lt(\Big(1 +o(1)\Big)\lambda\rt)\rt),& \quad \text{ if }\,0 < \eta d \leq 1.
 \end{cases}
\]
\end{theorem}

As highlighted by the numerical results in Section~\ref{sec:NumExp}, the lower bound \eqref{eqn:TDC_vs_TAR} is a loose bound, because the bound is established by considering the running time complexity only up to the time at which the
 lower bounding process receives its first arrival. This can be much smaller than the running time complexity until the coalescence of the upper and lower bound processes.

\subsection{Method 2}
\label{sec:Alt-DCFTP}
This method is an improved version of Method~1, again based on \cite{KM00}. 
Observe that at any given time $t \in \reals$, the interaction process $\mZ(t)$ can have only non-overlapping spheres. This suggests a better way of constructing the bounding processes that we describe now. For each $n \geq 1$, just like in Method~1, start with ${\mL_n(t_{-n}) = \varnothing}$ and ${\mU_n(t_{-n}) = \mD(t_{-n})}$ to guarantee that ${\mL_n(t_{-n}) \subseteq \mZ(t_{-n}) \subseteq \mU_n(t_{-n})}$. Suppose that the event at $t_{i+1}$ is an arrival of sphere $x$. Irrespective of whether $\mU_n(t_i) \in \mathscr{A}$ or not, if $x$ is not overlapping with any sphere in the upper bounding process $\mU_n(t_i)$, then it can not overlap with any sphere in $\mZ(t_i)$ and hence is accepted to $\mZ$. Thus, we add $x$ to both the bounding processes. (Observe that in Method~1, such an $x$ is added to both the bounding processes only if $\mU_n(t_i) \in \mathscr{A}$ because of the Papangelou conditional intensity \eqref{eqn:Papangelou_HS}.) If $x$ overlaps with any sphere in the lower bounding process $\mL_n(t_i)$, then it must overlap with a sphere in $\mZ(t_i)$ as well, and hence it is not added to any of the bounding processes $\mL$ and $\mU$. If $x$ does not overlap with any sphere in $\mL_n(t_i)$, but does overlap with a sphere in $\mU_n(t_i)$, its presence in the process $\mZ$ cannot be ruled out, and hence we keep it in the upper bounding process, but not in the lower bounding process. 
Finally, every death in $\mD$ is reflected in both the bounding processes $\mL_n$ and $\mU_n$.
Under this construction, the lower bounding process accepts births more often and hence the upper bounding process accepts births less often when compared with the construction in Section~\ref{sec:kendall}. As a result the running time of Method 2 is shorter than that of Method~1.

\subsection{Method 3}
\label{sec:huber}
A different approach for dominated CFTP for repulsive pairwise interaction processes has been proposed by Huber \cite{MH12}. Here, we discuss main ingredients of the method for hard-sphere model;
refer to \cite{MH12} and \cite{MH16} for more details. 
In this method, the interaction process $\mZ$ is different from $\mZ$ in Sections~\ref{sec:kendall}-\ref{sec:Alt-DCFTP} and is known as spatial birth-death {\it swap} process
whose invariant distribution is again the distribution $\mu$ of the hard-sphere model.
In addition to births and deaths of spheres, this process also allows {\it swap} moves; here a swap move is an event where an existing sphere is replaced by an arrival if it is the only sphere that is overlapping with the arrival. The lower and upper bounding processes are constructed as follows:
As usual let ${\mU_n(t_{-n}) =\mD(t_{-n})}$ and ${\mL_n(t_{-n}) = \varnothing}$. For any ${0 < k < n}$, if $t_{-k}$ is an instant of a death in the dominating process $\mD(t)$ then the death is reflected in
both the upper and lower bound processes. Now suppose that $x \in \mD(t_{-k})$ is born at $t_{-k}$.
\begin{itemize}
 \item[] {\bf Case 1:} If no sphere in $\mU_n(t_{-k})$ is overlapping with $x$, then the arrival sphere $x$ is added to both $\mU_n(t_{-k})$ and $\mL_n(t_{-k})$.
                       If only one sphere $y$ in $\mU_n(t_{-k})$ is overlapping with $x$, then $y$  is removed from $\mU_n(t_{-k})$ (from $\mL_n(t_{-k})$ if it is present) and $x$ is added to both $\mU_n(t_{-k})$ and $\mL_n(t_{-k})$.
 \item[] {\bf Case 2:} There are at least two spheres in $\mL_n(t_{-k})$ overlapping with $x$. Then $x$ is rejected by both $\mU_n(t_{-k})$ and $\mL_n(t_{-k})$.
 \item[] {\bf Case 3:} There is at most one sphere in $\mL_n(t_{-k})$ and at least two spheres in $\mU_n(t_{-k})$ overlapping with $x$.  Then $x$ is added to $\mU_n(t_{-k})$ (but not to $\mL_n(t_{-k})$). If $y \in \mL_n(t_{-k})$ is the one that is overlapping with $x$, then  remove $y$ from $\mL_n(t_{-k})$.
\end{itemize}

\section{Simulations}
\label{sec:NumExp}
We compare the performance of all the methods discussed in this paper using numerical experiments, and illustrate the effectiveness of the proposed IS based AR method over certain regimes where the other methods fail to work. For this, we consider {the torus-hard-sphere model} with a {fixed radius $\rbdd/\lambda^\eta$ on $2$-dimensional square $[0,1]^2$}. Thus, $\eta d = 2\eta$. In the first two experiments, by  fixing values of $\eta$ and $\rbdd$, we estimate the complexities of the algorithms as a function of the intensity $\lambda$ of the reference PPP by computing a sample average of the number of spheres (or, circles in this case) generated per generation of a perfect sample of the hard-sphere model. Instead of estimating the expected running time complexities, we take this approach to keep the discussion independent of the underlying data structures and programming language used in the implementation of the algorithms. In addition, we estimate the non-overlapping probability $\pno$ using the conditional Monte Carlo rare event estimation for {\em Gilbert graphs} proposed by \cite{HMTK20}. The Gilbert graph under consideration is a random graph where the nodes constitute a $\lambda$-homogeneous PPP on $[0,1]^2$ and there is an edge between two points if they are within a distance of $2\rbdd/\lambda^\eta$. Therefore, $\pno$ is the probability that there are no edges in the graph. The codes for all the methods discussed in this paper are available at \href{https://github.com/saratmoka/PerfectSampling_HardSpheres}{https://github.com/saratmoka/PerfectSampling\_HardSpheres}. \\

For the implementation of the proposed IS based AR method, the grid is constructed using the cell-edge length $\epsilon = 1/\lfloor \lambda^\eta/\rbdd\rfloor$; see Section~\ref{sec:improv-const} for more details on the cell-edge length selection. 
The complexities of the algorithms are estimated using $1000$ samples. In the simulation results presented below, $\wh S_{\mathsf{NAR}}$ and $\wh S_{\mathsf{ISAR}}$ denotes the sample means of complexities of the naive AR and the IS based AR algorithms. Likewise, $\wh S_{\mathsf{DCM1}}$,  $\wh S_{\mathsf{DCM2}}$ and  $\wh S_{\mathsf{DCM3}}$ are the corresponding estimates for the three dominated CFTP methods $1$, $2$, and $3$ presented in Section~\ref{sec:dcftp}, respectively. \\

A standard software used for generating perfect samples of the hard-sphere model using the dominated CFTP is {\sf rHardcore()}, which is a part of {\sf R} package {\sf Spatstat} that is available at \href{https://spatstat.org/}{https://spatstat.org/}. Experiment~3 provides a perspective on the performance of the proposed method with respect to {\sf rHardcore()} by comparing their expected running times as a function of $\rbdd$ for a fixed $\lambda$. We note that {\sf rHardcore()} does not support the torus-hard-sphere model. However, when selected "expand=TRUE",  it reduces the boundary effects by generating a perfect sample on a larger window, and then clipping the result to the original window $[0,1]^2$.\\

\begin{figure}[H]
 \centering
\includegraphics[ width=0.8\textwidth, height = 0.4\linewidth]{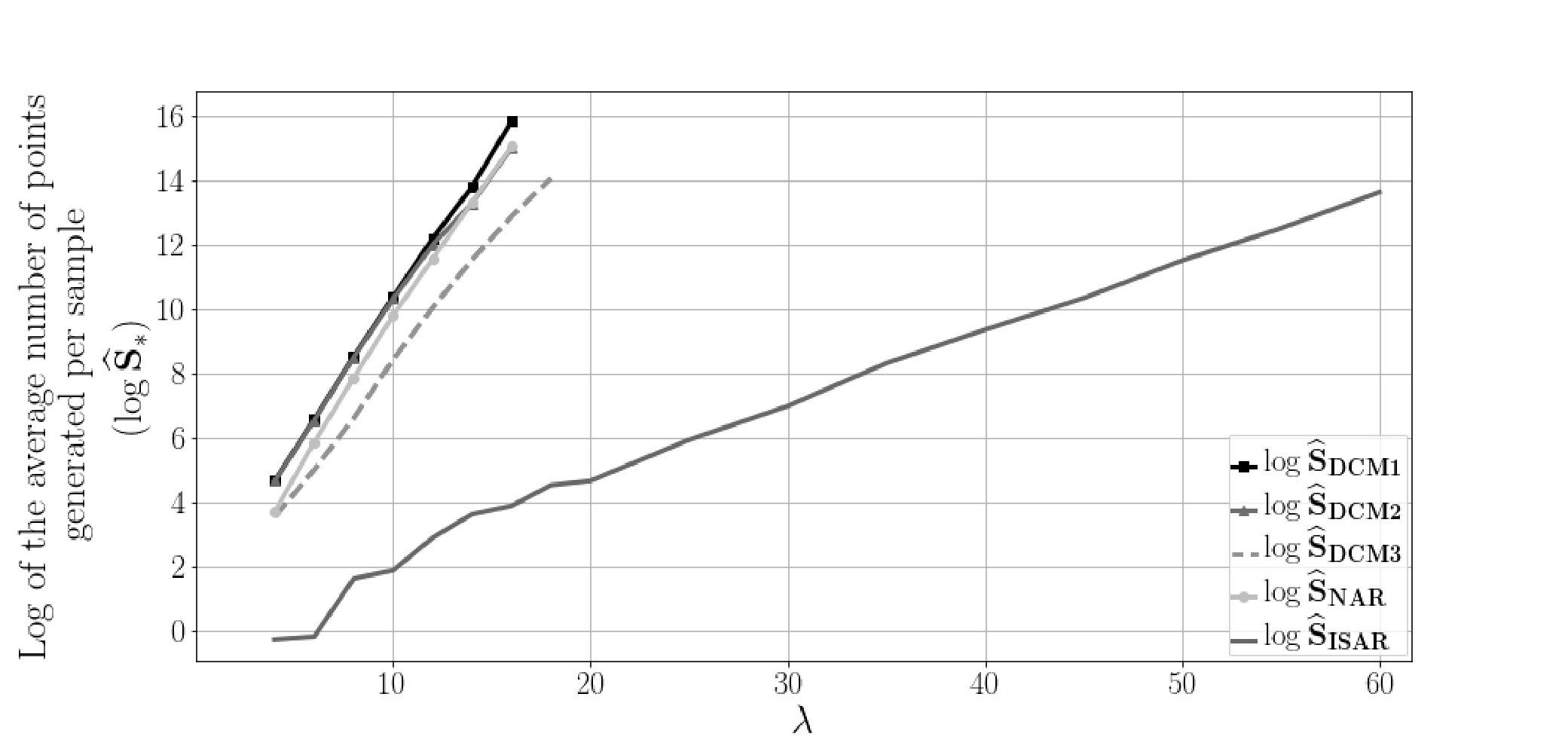}
\caption{Log of the expected number of points generated per a perfect sample of the hard-sphere model, as a function of $\lambda$, in the regime where $\eta = 0.5$, $d = 2$, and $\rbdd = 1$.}
\label{pic:Compare1}
\end{figure}

\noindent
{\bf Experiment 1:} In this experiment, we consider the high density regime.  Figure~\ref{pic:Compare1} compares the performance of all the algorithms for $\eta = 0.5$ (that is $2\eta = 1$) and  $\rbdd = 1$ ({this is identical to the regime where the underlying space is $[0, \sqrt{\lambda}]^2$ and the radius of each sphere is $\rbdd$}). 
This experiment suggests that the proposed IS based AR method can perform significantly better than every other method. To comprehend the rarity of the samples of the hard-sphere configurations under $\mu^0$, we plot $\log \pno$ in Figure~\ref{pic:exp1-plam} and the expected intensity of the hard-sphere model in Figure~\ref{pic:exp1-hcpts}.\\

\begin{figure}[H]
 \centering
\includegraphics[width=0.8\textwidth, height = 0.4\linewidth]{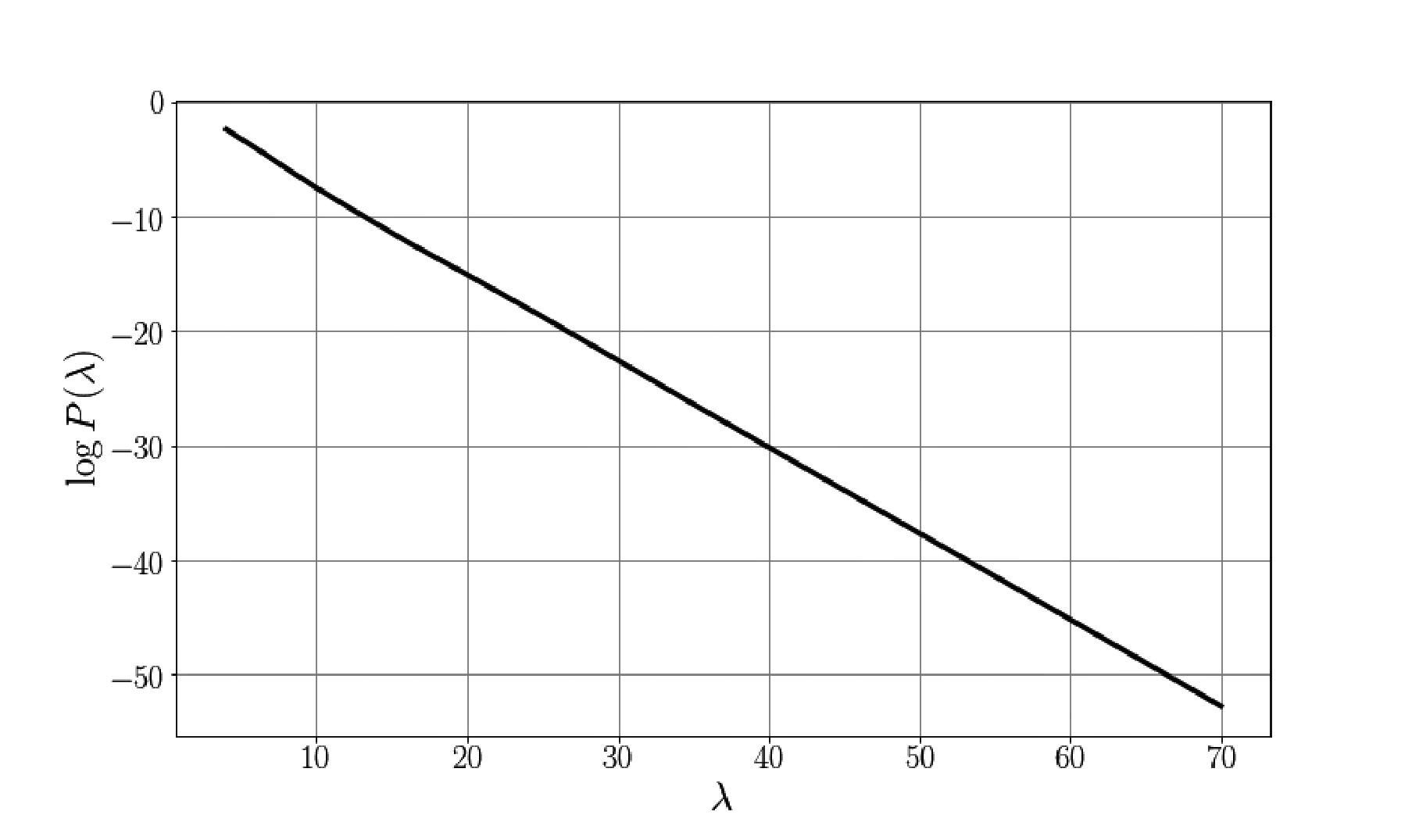}
\caption{$\log \pno$ vs $\lambda$ in the regime where $\eta = 0.5$, $d = 2$, and $\rbdd = 1$, where the non-overlapping probability $\pno$ is estimated using the conditional Monte Carlo method proposed in \cite{HMTK20}. The plot shows that, in the high density regime, the configurations with hard-spheres can be extremely rare under the measure~$\mu^0$.}
\label{pic:exp1-plam}
\end{figure}

Significance of the proposed IS method in the high density regime is more evident when $\eta = 0.25$ (that is, $2\eta = 0.5$) and $\rbdd = 1$. In this case, for values of $\lambda$ greater than $50$,  almost all the times all the dominated CFTP algorithms terminated without giving an output. In particular, 
the {\sf rHardcore} function terminated by producing  the error: {\em memory exhausted (limit reached?)}. On the other hand, the time taken (in secs) for generating $1000$ samples using the proposed method  are $0.13, 0.21, 68.94$ and $271.72$, when $\lambda$ values are $50, 100, 200, 300$ and $400$, respectively.\\

\begin{figure}[H]
 \centering
\includegraphics[ width=0.7\textwidth, height = 0.37\linewidth]{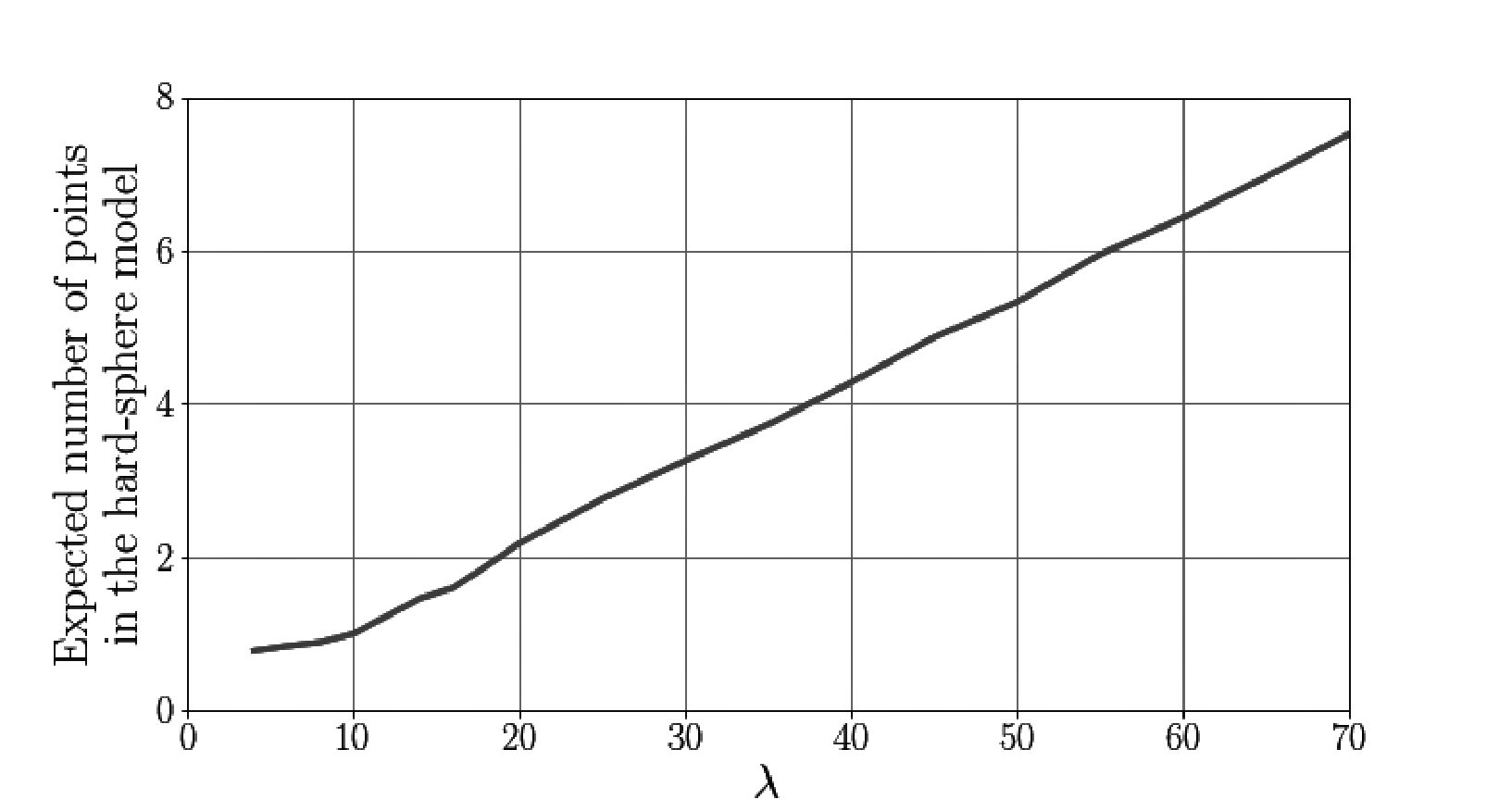}
\caption{The intensity of the hard-sphere model against $\lambda$ in the regime where $\eta = 0.5$, $d = 2$, and $\rbdd = 1$. Due to extreme rarity of the hard-sphere configurations under $\mu^{0}$ as shown in Figure~\ref{pic:exp1-plam}, the intensity of the hard-sphere model is much smaller than the intensity $\lambda$ of the PPP.}
\label{pic:exp1-hcpts}
\end{figure}

\noindent
{\bf Experiment 2:}  In this experiment, we consider the low density regime. Figure~\ref{pic:Compare2} compares the performances of all the methods for $2\eta  = 1.5$ and $\rbdd = 0.5$ to illustrate the case where $1 < \eta d < 2$. As we can see, for large values of $\lambda$, the dominated CFTP methods $2$ and $3$ perform better than the other methods, including the proposed method. Figure~\ref{pic:exp2-plam} is a plot of $\log \pno$ against $\lambda$ while Figure~\ref{pic:exp2-hcpts} is a plot of the intensity of the hard-sphere model against $\lambda$.

\begin{figure}[h]
 \centering
\includegraphics[ width=0.8\textwidth, height = 0.4\linewidth]{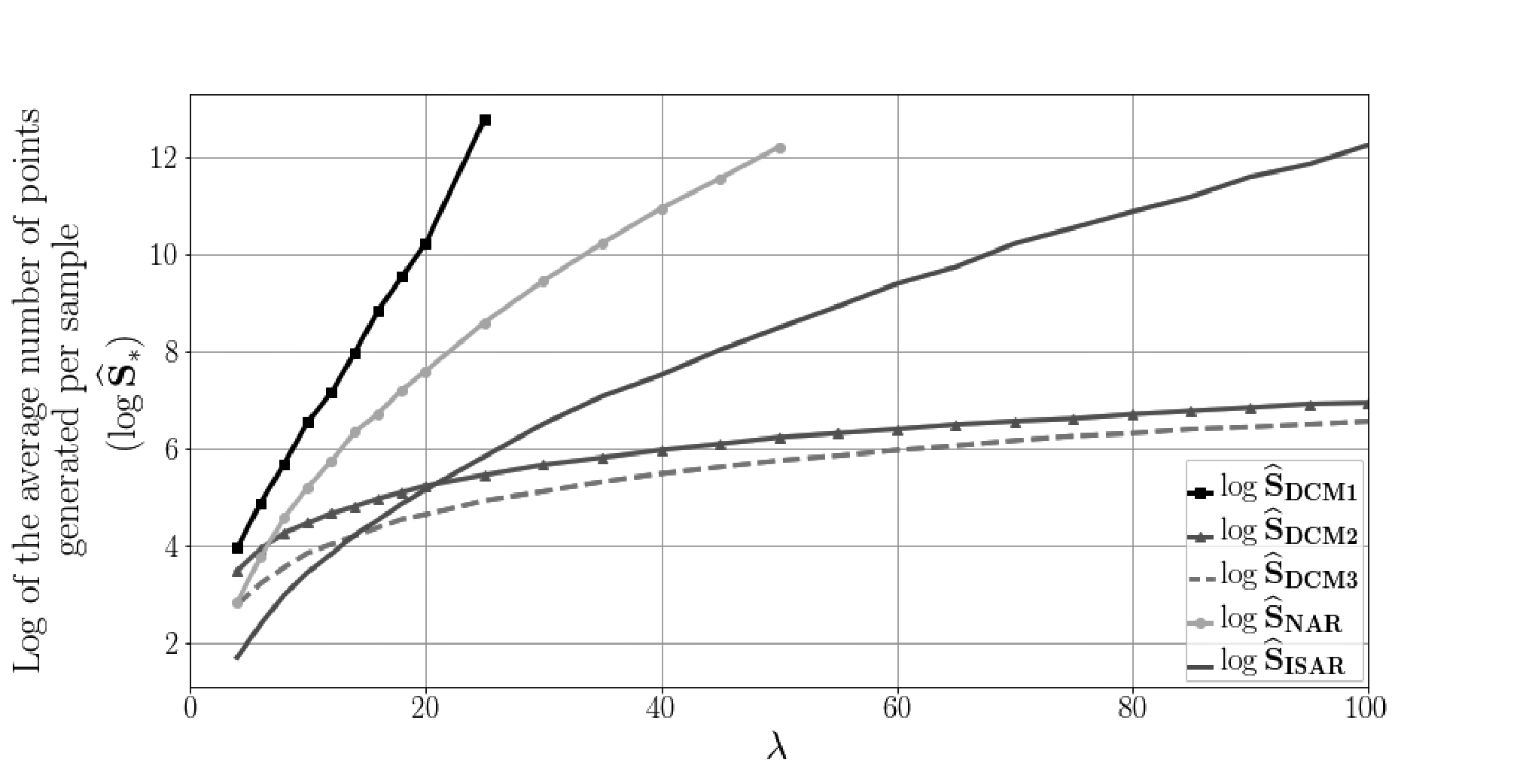}
\caption{Log of the expected number of points generated per a perfect sample of the hard-sphere model, as a function of $\lambda$, in the regime where $\eta = 0.75$, $d = 2$, and $\rbdd = 0.5$.}
\label{pic:Compare2}
\end{figure}
~
\begin{figure}[h]
 \centering
\includegraphics[ width=0.8\textwidth, height = 0.4\linewidth]{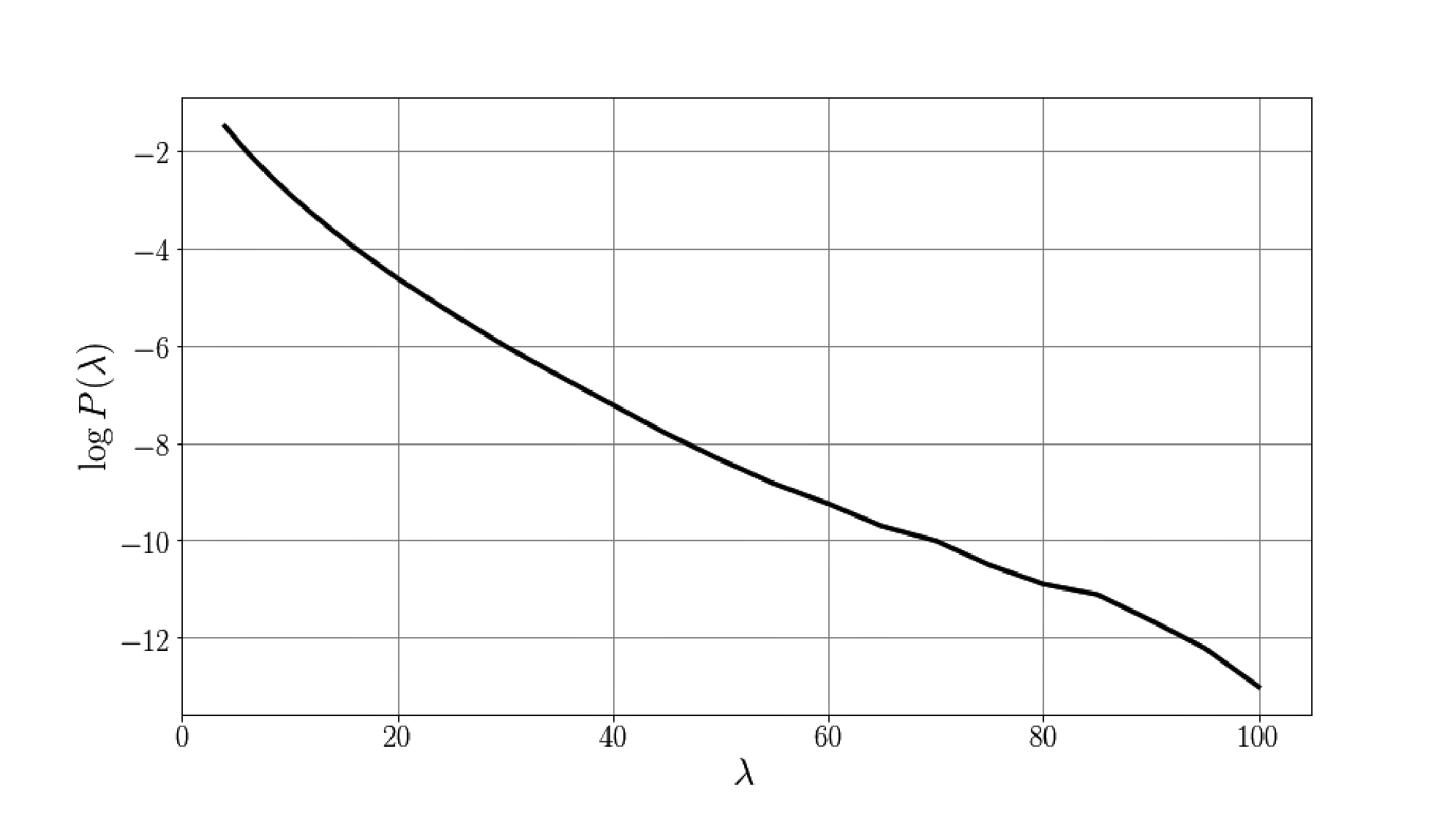}
\caption{$\log \pno$ vs $\lambda$ in the regime where $\eta = 0.75$, $d = 2$, and $\rbdd = 0.5$, where the non-overlapping probability $\pno$ is estimated using the conditional Monte Carlo method proposed in \cite{HMTK20}. Here we notice that the hard-sphere configurations are relatively less rare compared to the scenarios in Experiment~1.}
\label{pic:exp2-plam}
\end{figure}

\begin{figure}[h]
 \centering
\includegraphics[ width=0.8\textwidth, height = 0.45\linewidth]{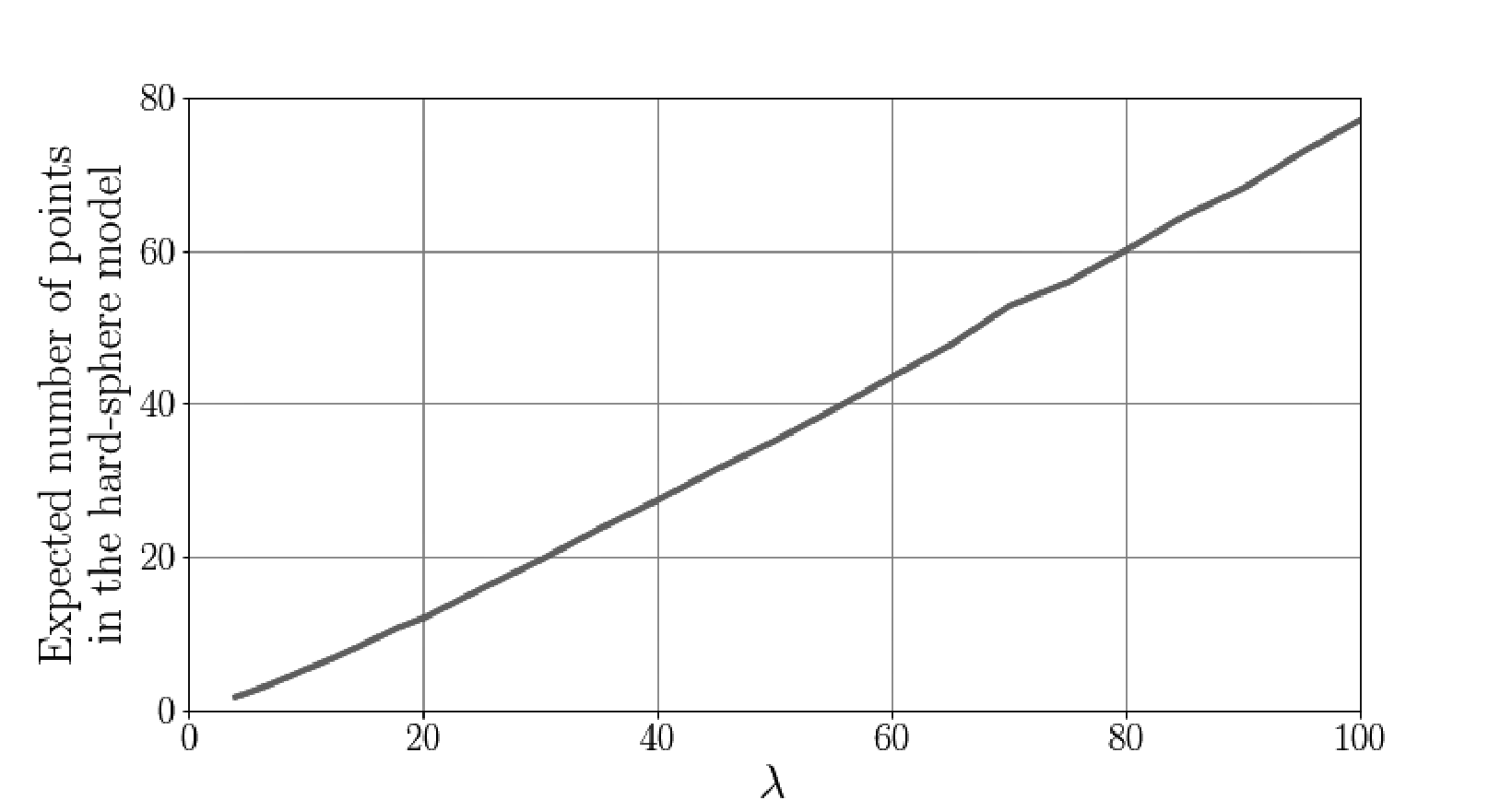}
\caption{The intensity of the hard-sphere model against $\lambda$ in the regime where $\eta = 0.75$, $d = 2$, and $\rbdd = 0.5$. Unlike in Experiment~1, the intensity of the hard-sphere model is relatively close to the intensity $\lambda$ of the PPP.}
\label{pic:exp2-hcpts}
\end{figure}

\noindent
{\bf Experiment 3} Figure~\ref{pic:Compare3} compares the running times  of the proposed IS based AR method and {\sf rHardcore()} for generating $1000$ samples. The same computer is used for running both the softwares. Here, we vary $\rbdd$ while fixing $\lambda = 50$ and $2\eta = 1$. Observe that for large values of $\rbdd$  the density is higher, and the proposed method performs far better than the dominated CFTP.  As we expect for this regime, as the radius $\overline r$ increasing, the intensity of the hard-sphere model is decreasing (Figure~\ref{pic:exp3-hcpts}) while the rarity of the hard-sphere configurations under $\mu^0$ is increasing (Figure~\ref{pic:exp3-plam}). 

\begin{figure}[h]
 \centering
\includegraphics[width=0.8\textwidth, height = 0.4\linewidth]{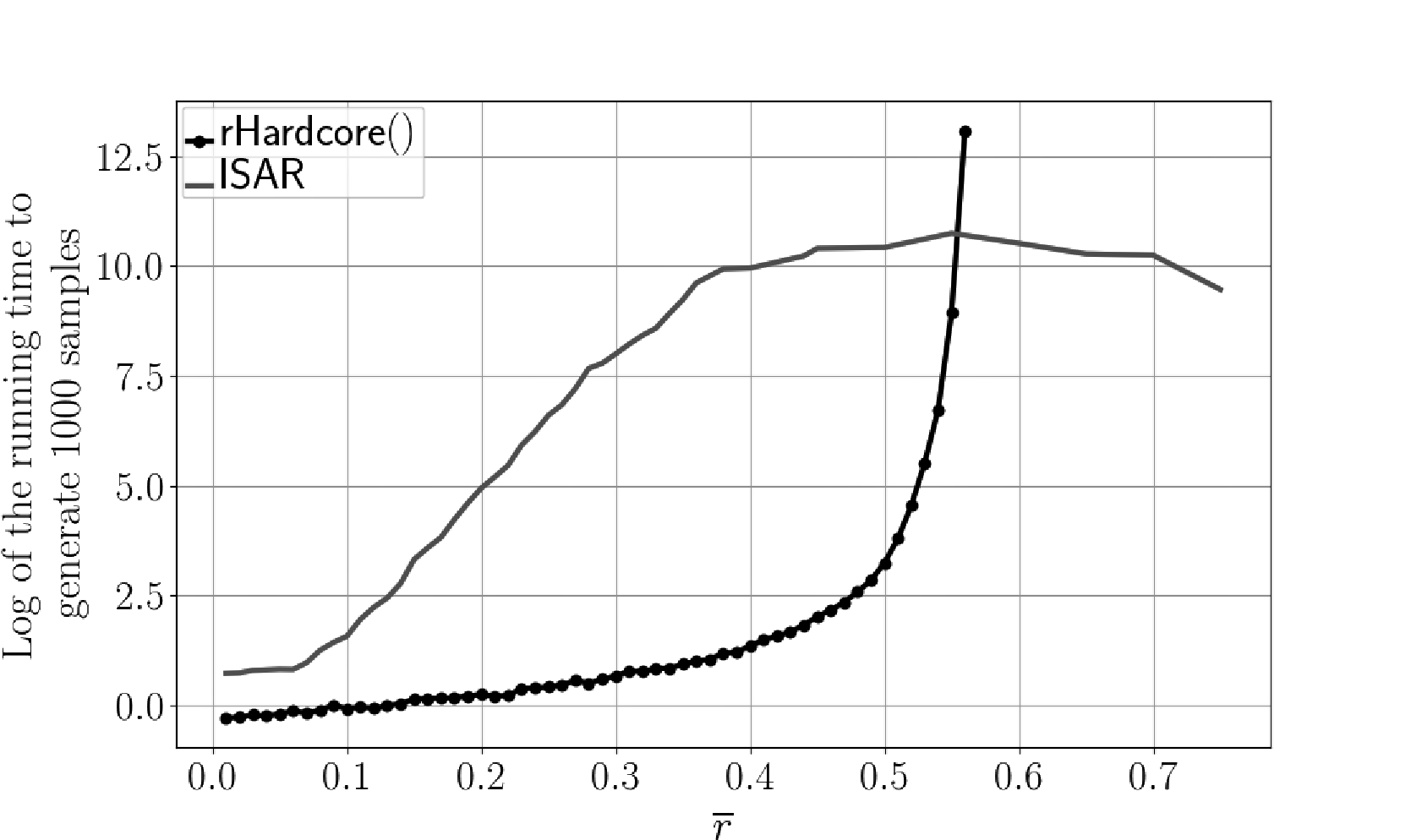}
\caption{Comparison  between the running times  of the proposed IS based AR method and {\sf rHardcore()} for generating $1000$ samples}
\label{pic:Compare3}
\end{figure}
~
\begin{figure}[H]
 \centering
\includegraphics[ width=0.8\textwidth, height = 0.4\linewidth]{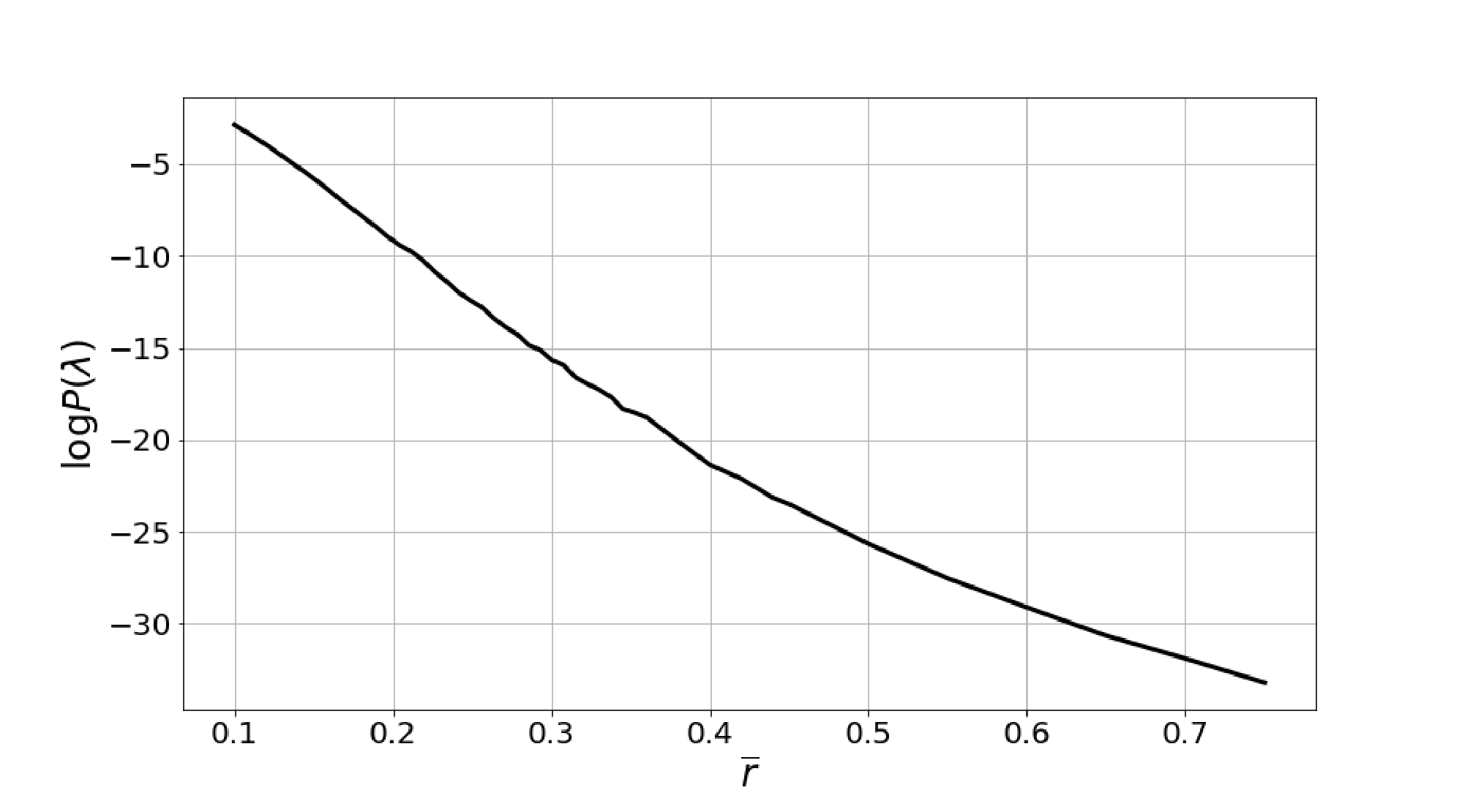}
\caption{$\log \pno$ vs $\rbdd$ in the regime where $\eta = 0.5$, $d = 2$, and $\lambda = 50$, where the non-overlapping probability $\pno$ is again estimated using the conditional Monte Carlo method.}
\label{pic:exp3-plam}
\end{figure}
~
\begin{figure}[H]
 \centering
\includegraphics[ width=0.8\textwidth, height = 0.4\linewidth]{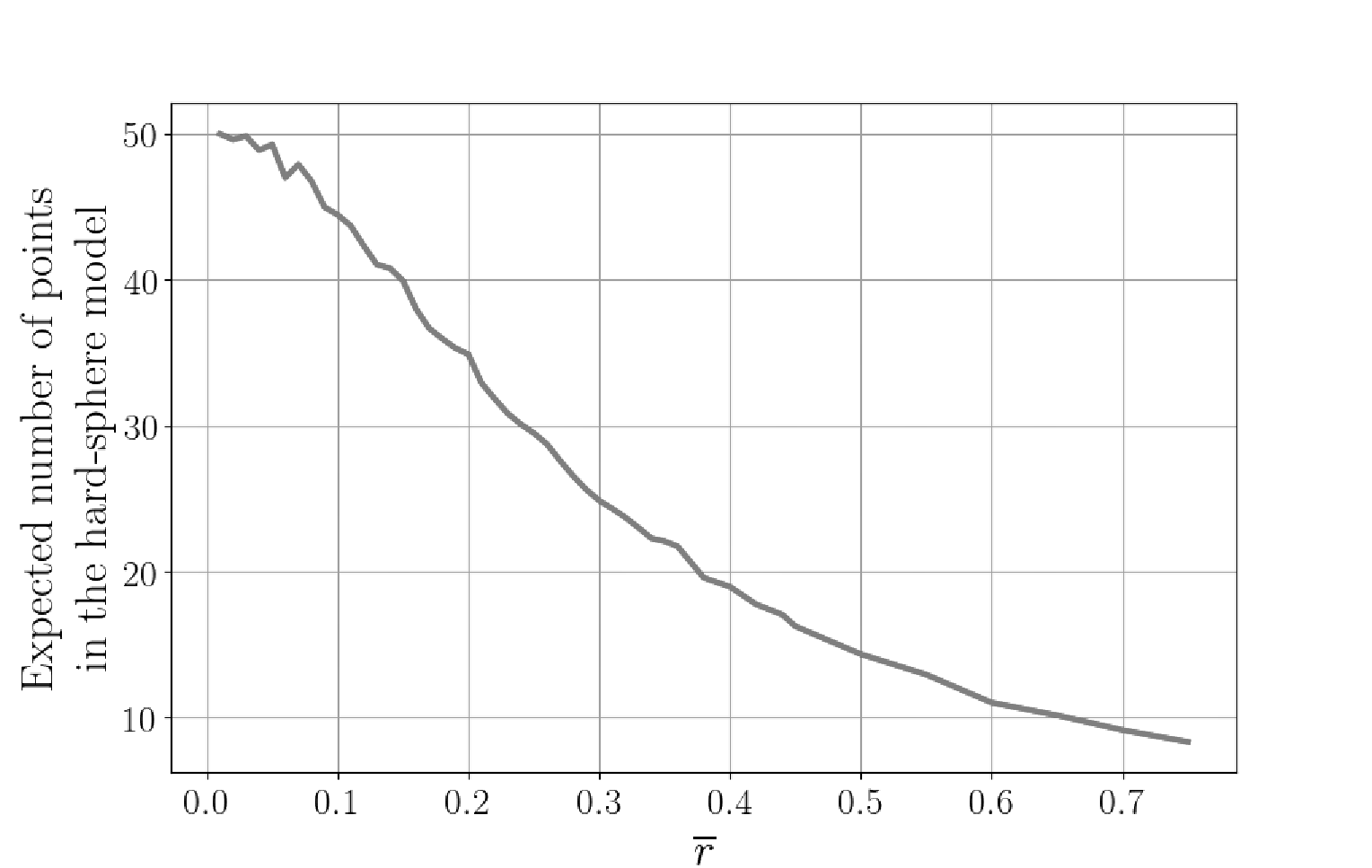}
\caption{Intensity of the hard-sphere model against $\overline{r}$ in the regime where $\eta = 0.5$, $d = 2$, and $\lambda = 50$. }
\label{pic:exp3-hcpts}
\end{figure}

\section{Conclusion}
\label{sec:ConAR}
In this paper we considered the problem of perfect sampling for Gibbs hard-sphere models on $[0,1]^d$.  We discussed the performance of the naive acceptance-rejection method and introduced importance sampling based enhancements to it. We also compared these methods to some of the popular coupling from the past based techniques prevalent in the existing literature. For the performance  analysis and comparison (of expected running time complexity), we developed an asymptotic regime where the intensity 
$\lambda$ of the reference Poisson point process increased to infinity, while the (random) volume of each sphere is an order of $\lambda^{- \eta d}$ decreased to zero, for different regimes of $\eta d > 0$.
One main conclusion is that while the dominated coupling from the past methods perform better for $1< \eta d < 2$ for large $\lambda$,
our importance sampling based methods provide a significant improvement for $\eta d \leq 1$. Enroute, we established large deviations results for the probability that spheres do not overlap with each other when their centers constitute a Poisson point process.
We also conducted numerical experiments to validate our asymptotic results. \\

The  proposed importance sampling based acceptance-rejection methods rely on clever partitioning of the underlying configuration space and arriving at an appropriate change of measure on each partition. While we showed how this may be effectively conducted for hard-sphere models, further research is needed to develop effective implementations for perfect sampling from a broad class of Gibbs point processes.

\section*{Acknowledgement}
SM acknowledges support of the Australian Research Council Centre of Excellence for Mathematical and Statistical Frontiers (ACEMS), under grant number CE140100049. SJ acknowledges support of the Department of Atomic Energy, Government of India, under project no. RTI4001. MM's research is partly funded by the NWO Gravitation project NETWORKS, grant number 024.002.003.

\begin{appendices}

\section{Proofs}
\label{Proofs}
The following lemmas are useful for proving Theorem~\ref{lem:non-ovr-rand}, Proposition~\ref{prop:AR_meth} and Proposition~\ref{prop:imp_samp}. Lemma~\ref{lemma:cherpoi} is a standard Chernoff bound for Poisson random variables and Lemma~\ref{lem:Hoeffding} is Hoeffding's inequality for $U$-statistics \cite{Hoe63}.

\begin{lemma}[Chernoff bound for Poisson]
\label{lemma:cherpoi}
Let $N \sim Poi(\lambda)$. Then, for  any $0 < \epsilon < 1$,
\[
 \pp\left( N \leq (1- \epsilon)\lambda \right) \leq \exp\lt( - \frac{\lambda \epsilon^2}{2}\rt)\,\,\, \text{  and  }\,\,\,\,\, \pp\left( N \geq (1 + \epsilon)\lambda \right) \leq \exp\lt( - \frac{\lambda \epsilon^2}{3} \rt).
\]
\end{lemma}
\begin{lemma}[Hoeffding, 1963]
\label{lem:Hoeffding}
 Suppose that $\xi_1, \xi_2, \dots, \xi_n$ are iid random variables and {$g:\reals^k \rightarrow [0,1]$ is a measurable function.}
 Set
 \[
 Z_n = \sum_{1 \leq i_1 < i_2< \dots < i_k \leq n} g\lt(\xi_{i_1}, \xi_{i_2}, \dots,\xi_{i_k}\rt)
 \] for a positive integer $k \leq n$ (this is known as {a $U$-statistics} of order $k$).
 Then, for any $\epsilon > 0$,
 \[
  \pp \lt(Z_n \geq {n \choose k} \Big(\ee[g(\xi_1,\xi_2,\dots,\xi_k)] + \epsilon\Big)\rt) \leq 2 \exp\lt(- 2 \lfloor n/k \rfloor \epsilon^2\rt).
 \]
 The same estimate holds for $\pp \lt(Z_n \leq {n \choose k} \Big(\ee[g(\xi_1,\xi_2,\dots,\xi_k)] - \epsilon\rt)\Big).$
\end{lemma}

\subsection{Proof of Theorem~\ref{lem:non-ovr-rand}}
\label{sec:LDResults}
Recall that $\lambda > 0$, $\eta > 0$, and  $\frac{R_1}{\lambda^\eta}, \dots, \frac{R_n}{\lambda^\eta}$ are the radii of $n$ spheres whose respective centers $Y_1, \dots, Y_n$ are independently and uniformly generated
on the $d$-dimensional unit cube $[0,1]^d$, where $R_1, \dots, R_n$ are iid positive random variables bounded from above by a constant $\rbdd$ and are independent of $Y_1, \dots, Y_n$. Define {$m_i := \ee\lt[(R_1 + R_2)^{id}\rt]$, for all $i  = 1, 2, \dots$.}
Let $\pnon$ be the probability that these $n$ spheres do not overlap with each other. Since the number of spheres in a $\lambda$-homogeneous marked PPP on $[0,1]^d$ is a Poisson random variable with mean $\lambda$,
the non-overlapping probability
\begin{align}
\label{eqn:pno_lam}
 \pno = \ee\lt[\mathcal{P}_N(\lambda)\rt],
\end{align}
where $N \sim \pois(\lambda)$.

In proving the theorem, we use the following lemmas which exploits the reference IS measure  $\wt \mu_n$ introduced in Section~\ref{sec:AR_GSC}.
By \eqref{eqn:LHRatio1} and the definition of $\pnon$,
\begin{align}
\label{eqn:PnBi}
\pnon &= \pp_{\mu^0} \lt(\state \in \mathscr{A} \big| |\state | = n \rt)\nonumber\\
      &= {\mbb E}_{\wt \mu_n}\left[I\lt(\state \in \mathscr{A} \rt)\prod_{i=1}^{n}\Big(1 - B_{i}\Big)\right]\nonumber\\
      &= {\mbb E}_{\wt \mu_n}\left[\prod_{i=1}^{n}\Big(1 - B_{i}\Big)\right].
\end{align}
The following bound holds  trivially,
\begin{equation}
\label{eqn:bni-ran}
B_{i} \leq \frac{\gamma}{\lambda^{\eta d}}\sum_{j=1}^{i-1}\left(R_j + R_i\right)^d,
\end{equation}
where the sum is taken to be zero when $i = 1$.
Let $\theta_{n,\lambda}  = \frac{\gamma (2\rbdd)^d\, n}{\lambda^{\eta d}}$. We have the following upper and lower bounds on $\pnon$. 
\begin{lemma}
\label{lem:binomial}
Under the above set-up,
 \begin{align}
 \label{eqn:pnon_low}
  \pnon \geq \exp\lt( - n \sum_{j = 1}^\infty \lt(\frac{\gamma n}{\lambda^{\eta d}}\rt)^j \frac{m_j}{j(j+1)}\rt).
 \end{align}
Furthermore, for any $\epsilon > 0$,
\begin{align}
  \label{eqn:pnon_up}
  \pnon \leq N_{n,\lambda} \lt[ \exp\left( - \frac{\gamma n(n-1)(m_1 - \epsilon)}{2\lambda^{\eta d}} \right) + 2\exp\left( - \frac{(n-1)\epsilon^2}{(2\rbdd)^{2d}}\right)\rt],
 \end{align}
 for any $n$ and $\lambda$ such that $\theta_{n, \lambda} < 1$, where $N_{n, \lambda}$ is a function of $n, \lambda$ and $\rbdd$, defined by \eqref{eqn:Nnlam} in the proof below, such that  
\begin{align}
  \lim_{\lambda \rightarrow \infty }\frac{1}{\lambda^{2 -\eta d}} \log N_{\lambda, \lambda} &= 0,\, \text{ if }\, \eta d > 1.\label{eqn:Kr-term22}
 \end{align}
 In particular, for the torus-hard-sphere model,
 \begin{align}
  \lim_{\lambda \rightarrow \infty } N_{\lambda, \lambda} &= 1,\, \text{ if }\, \eta d > 3/2.\label{eqn:Kr-term2}
 \end{align}
\end{lemma}
\begin{proof}
{\bf Lower Bound:} To prove \eqref{eqn:pnon_low} notice that,  by \eqref{eqn:PnBi},
\begin{align*}
 \pnon  = {\mbb E}_{\wt \mu_n}\left[\exp\lt( \sum_{i=1}^{n}\log \Big(1 - B_{i}\Big)\rt)\right]
        = {\mbb E}_{\wt \mu_n}\left[\exp\lt( \sum_{i=1}^{n}\sum_{j = 1}^\infty - \frac{1}{j}B_{i}^j\rt)\right],
\end{align*}
{using} the Taylor's expansion $\log(1 - x) = - \sum_{j =1}^\infty x^j/j$ for $0 \leq x \leq 1$. By Jensen's inequality and \eqref{eqn:bni-ran}, 
\begin{align*}
\pnon  &\ge \exp\left( - \sum_{i=1}^{n}\sum_{j = 1}^\infty \frac{1}{j}{\mbb E}_{\wt \mu_n}\left[B_{i}^j\right]\right)\nonumber\\
       &\geq \exp\left( - \sum_{i=1}^{n}\sum_{j = 1}^\infty \frac{\gamma^j\, }{j\lambda^{j\eta d}}{\mbb E}\left[\lt(\sum_{l = 1}^{i-1} \lt(R_l + R_{i}\rt)^{d}\rt)^j\right]\right)\nonumber\\
       &{\geq \exp\left( - \sum_{i=1}^{n}\sum_{j = 1}^\infty \frac{\gamma^j\, (i-1)^j}{j\lambda^{j\eta d}}{\mbb E}\left[\lt(\frac{1}{i-1}\sum_{l = 1}^{i-1} \lt(R_l + R_{i}\rt)^{d}\rt)^j\right]\right).}\nonumber
\end{align*}
Again by Jensen's inequality, $\lt(\frac{1}{i -1}\sum_{l = 1}^{i-1} \lt(R_l + R_{i}\rt)^{d}\rt)^j \leq \frac{1}{i-1}\lt(\sum_{l = 1}^{i-1} \lt(R_l + R_{i}\rt)^{jd}\rt)$,
and thus $$\pnon  \ge \exp\left( - \sum_{j = 1}^\infty \frac{\gamma^j m_j}{j\lambda^{j\eta d}}\sum_{i=1}^{n}(i-1)^j\right).$$
We establish \eqref{eqn:pnon_low} using $\sum_{i=1}^{n} (i-1)^j \leq \int_{x=0}^n x^j\, dx= \frac{n^{j+1}}{j+1}$.\\

\noindent
{\bf Upper Bound:}
Let $R_{(1)}, R_{(2)}, \dots, R_{(n)}$ be the order statistics of $R_1, R_2, \dots, R_n$. Since the non-overlapping probability $\pnon$ is independent of the order in which the spheres are generated, without loss of generality assume that the $i^{th}$ sphere has radius $R_{(i)}$.
Let, for each $1 \leq j \leq i -1$, $\wt{B}_i(j)$ be the volume of the blocked region for the $i^{th}$ sphere generation when the $(j+1)^{th}, (j+2)^{th}, \dots, (i -1)^{th}$ spheres are ignored,
where $\wt{B}_i(0) = 0$. We can think of $\wt{B}_i(j) - \wt{B}_i({j-1})$ as the blocking volume contributed by the $j^{th}$ sphere for the $i^{th}$ sphere.
Under the new measure $\wt \mu_n$, the blocking volume seen by the $i^{th}$ sphere, $\displaystyle B_{i} = \sum_{j = 1}^{i-1} \lt(\wt{B}_i(j) - \wt{B}_i({j-1})\rt).$
Consider the sets 

$$\displaystyle {\mathscr N}^{(i)} := \left\{ j \in \{1, 2,\dots, i - 1\}: \wt{B}_i(j) - \wt{B}_i({j-1}) =  \frac{\gamma}{\lambda^{\eta d}} \Big(R_{(j)} + R_{(i)}\Big)^d\right\},$$ for $ i \leq n$ and take $\bar{\mathscr N}^{(i)} := \{1,2,\dots, i -1\} \setminus{\mathscr N}^{(i)}$.
Using the inequality $1 - x \leq e^{-x}$ and \eqref{eqn:PnBi},
\begin{align*}
\pnon  &\leq \ee_{\wt \mu_n}\left[  \exp\left(-\sum_{i=1}^{n} B_{i}\right)\right]
    = \ee_{\wt \mu_n}\left[  \exp\left(-\sum_{i=1}^{n} \sum_{j = 1}^{i-1} \lt(\wt{B}_{i}( j) - \wt{B}_{i}( j-1)\rt) \right)\right]\nonumber\\
    &\leq \ee_{\wt \mu_n}\left[  \exp\left(-\frac{\gamma}{\lambda^{\eta d}}\sum_{i=1}^{n} \sum_{j \in {\mathscr N}^{(i)}} \Big(R_{(j)} + R_{(i)}\Big)^d \right)\right]\nonumber\\
    &= \ee_{\wt \mu_n}\left[  \exp\left( - \frac{\gamma\,(2\rbdd)^d}{\lambda^{\eta d}}Z_n +\frac{\gamma }{\lambda^{\eta d}}\sum_{i=1}^{n} \sum_{j \in \bar{\mathscr N}^{(i)}} \Big(R_{(j)} + R_{(i)}\Big)^d \right)\right]\nonumber\\
    &\leq \ee_{\wt \mu_n}\left[  \exp\left( - \frac{\gamma\,(2\rbdd)^d}{\lambda^{\eta d}}Z_n +\frac{\gamma (2\rbdd)^d}{\lambda^{\eta d}}\sum_{i=1}^{n} |\bar{\mathscr N}^{(i)}| \right)\right],\nonumber
\end{align*}
where $Z_n = \sum_{i =1}^{n} \sum_{j = 1}^{i-1} \lt(\frac{R_{(j)} + R_{(i)}}{2\rbdd}\rt)^d,$
and the last inequality holds from the assumption that each $R_i \leq \rbdd$.
Since $R_{(i)}$ is non-decreasing with $i$, from the definition of ${\mathscr N}^{(i)}$, it is easy to see that $|\bar{\mathscr N}^{(i)}|$ is {non-decreasing} with $i$. Therefore, 
\begin{align*}
\pnon  &\leq \ee_{\wt \mu_n}\left[  \exp\left( - \frac{\gamma\,(2\rbdd)^d}{\lambda^{\eta d}}Z_n +\theta_{n, \lambda} |\bar{\mathscr N}^{(n)}| \right)\right]\\
           &= \ee_{\wt \mu_n}\left[  \exp\left( - \frac{\gamma\,(2\rbdd)^d}{\lambda^{\eta d}}Z_n \right)\ee_{\wt \mu_n}\left[ \exp\left(\theta_{n, \lambda} |\bar{\mathscr N}^{(n)}| \right) \Big| R_1, \dots, R_n \right] \right].
\end{align*}
We now show that $|\bar{\mathscr N}^{(n)}|$ is stochastically bounded by a binomial random variable, and as a consequence, the conditional expectation $\ee_{\wt \mu_n}\left[ \exp\left(\theta_{n, \lambda} |\bar{\mathscr N}^{(n)}| \right) \Big| R_1, \dots, R_n \right] $ is uniformly bounded from above by a constant, which is {a function} of $n, \lambda$ and $\rbdd$.
Let 
$$q_{j} = \pp_{\wt \mu_n}\left(\wt B_{n}(j) - \wt B_{n}(j-1) < \frac{\gamma }{\lambda^{\eta d}}(R_{(j)} + R_{(n)})^d\right).$$
Clearly, $q_{j}$ is increasing with $j$, and therefore $q_j \leq q_{n -1}$ for all $j \leq n-1$.
This implies that $|\bar{\mathscr N}^{(n)}|$ is stochastically bounded by a binomial random variable with parameters $n$ and $q_{n-1}$, and thus
\begin{align*}
\ee_{\wt \mu_n}\left[ \exp\left(\theta_{n, \lambda} |\bar{\mathscr N}^{(n)}| \right) \Big| R_1, \dots, R_n \right]  &\leq \Big( q_{n-1} \exp\lt(\theta_{n, \lambda}\rt) + (1 -  q_{n-1}) \Big)^n.
\end{align*}  
Due to the boundary effect, $q_{n-1}$ is not the same for the torus model (where boundary spheres loop over to the opposite boundaries) and the Euclidean model. 
Observe that, for the Euclidean model, $\wt B_{n}(n-1) - \wt B_{n}(n-2) < \frac{\gamma(R_{(j)} + R_{(n)})^d}{\lambda^{\eta d}}$ if either 
\begin{itemize}
\item[(1)] the center of the $(n -1)^{th}$ sphere is within $ (R_{(j)} + R_{(n-1)} + 2R_{(n)})/\lambda^\eta$ distance from the center of $j^{th}$ sphere for some $j \leq n-2$, or 
\item[(2)] the center of $(n-1)^{th}$ sphere is within $(R_{(n-1)} + R_{(n)})/\lambda^\eta$ distance form the boundary of the unit cube. 
\end{itemize}
Note that the boundary event (2) is irrelevant for the torus-hard-sphere model. 
The probability of the event (1) is maximized by $\displaystyle \frac{\gamma}{\lambda^{\eta d}} \sum_{j = 1}^{n-2}\frac{\lt(R_{(j)} + R_{(n-1)} + 2R_{(n)} \rt)^d}{1 - B_{n-1}}$, while that for the event (2) is maximized by $\displaystyle\frac{1 - \lt(1 - 2(R_{(n-1)} + R_{(n)})/\lambda^\eta \rt)^d}{1 - B_{n-1}}$.
{Since the} $R_{i}'$s are bounded from above by $\rbdd$ and $B_{n-1} \leq \theta_{n, \lambda}$ (from \eqref{eqn:bni-ran}), we have
\begin{align*}
 q_{n-1} \leq \bar q_{n,\lambda} := \begin{cases} 
             \frac{2^d \theta_{n,\lambda}}{1 - \theta_{n,\lambda}} + \frac{c}{\lambda^\eta (1 - \theta_{n,\lambda})}  &\quad \text{ for the Euclidean model},  \\
             \frac{2^d \theta_{n,\lambda}}{1 - \theta_{n,\lambda}}  &\quad \text{ for the torus model},  \\
\end{cases} 
\end{align*}
for any $n$ and $\lambda$ such that $\theta_{n,\lambda} < 1$, and for some constant $c$. Let 
\begin{align}
N_{n, \lambda} = \Big( 1 + \bar q_{n,\lambda} \Big(\exp\lt(\theta_{n, \lambda}\rt) - 1\Big) \Big)^n,
\label{eqn:Nnlam}
\end{align}
 then $\pnon  \leq N_{n,\lambda}\, \ee_{\wt \mu_n}\left[  \exp\left( - \frac{\gamma\,(2\rbdd)^d}{\lambda^{\eta d}}Z_n \right)\right].$
Using the definition of $Z_n$, for any $\epsilon > 0$,
\small{
\begin{align*}
\ee_{\wt \mu_n}\left[  \exp\left( - \frac{\gamma\, (2\rbdd)^d}{\lambda^{\eta d}}Z_n\right)\right] \leq \exp\left( - \frac{\gamma n(n-1)(m_1 - \epsilon)}{2\lambda^{\eta d}} \right) + \pp_{\wt \mu_n}\left(Z_n < \frac{n(n-1)(m_1 - \epsilon)}{2(2\rbdd)^d} \right).
\end{align*}}
By Lemma~\ref{lem:Hoeffding} (with $k = 2$), $\pp_{\wt \mu_n}\left(Z_n < \frac{n(n-1)}{2} \lt(\frac{m_1 - \epsilon}{(2\rbdd)^d}\rt)\right) \leq 2\exp\left( - \frac{(n-1)\epsilon^2}{(2\rbdd)^{2d}}\right)$,
 and thus \eqref{eqn:pnon_up} is established.
 
It remains to prove \eqref{eqn:Kr-term22} and \eqref{eqn:Kr-term2} under the assumption that $\eta d > 1$. For this case, we have that $\lim_{\lambda \nearrow \infty}\theta_{\lambda, \lambda} = 0$ and hence $\lim_{\lambda \nearrow \infty}\bar q_{\lambda, \lambda} = 0$. 
Since $$N_{\lambda, \lambda}  \leq \exp\lt( \lambda\, \bar q_{\lambda, \lambda}\, \Big[\exp\lt(\theta_{\lambda, \lambda} \rt) - 1\Big] \rt),$$
using Taylor's expansion {of the} exponential function,
\[
0 \leq \lim_{\lambda \nearrow \infty}\frac{1}{\lambda^{2 - \eta d}} \log N_{\lambda, \lambda} \leq \lim_{\lambda \nearrow \infty }\lt[\bar q_{\lambda, \lambda} \sum_{j\in \mbb{N}_0} \frac{\gamma^{(j+1)} (2\rbdd)^{(j+1)d}}{(j+1)!} \lambda^{j(\eta d - 1)}\rt] = 0.
\]
Thus \eqref{eqn:Kr-term22} holds. In particular, for the torus  model with $\eta d > 3/2$,
\[
\lambda \, \bar q_{\lambda, \lambda}\, \Big[\exp\lt(\theta_{\lambda, \lambda} \rt) - 1\Big] = \lambda \, \frac{2^d  \theta_{\lambda, \lambda}}{1 - \theta_{\lambda, \lambda} }\, \Big[\exp\lt(\theta_{\lambda, \lambda} \rt) - 1\Big] = \frac{2^d}{1 - \theta_{\lambda, \lambda}}\sum_{j= 2}^\infty \frac{\gamma^j (2\rbdd)^{jd}}{j!} \lambda^{1 - j(\eta d - 1)}
\]
goes to $0$ as $\lambda \nearrow \infty$, and hence \eqref{eqn:Kr-term2} holds.
\end{proof}

\begin{lemma}
 \label{lem:pno-upper}
 Suppose that $1 < \eta d \leq 2$. Then, for any $0 < a < 0.5$,
\begin{align}
\label{eqn:PaccL1}
\pno  \geq \exp\left( - \sum_{j = 1}^\infty \frac{\lambda^{j(1 - \eta d)+1} \lt(1+ \frac{1}{\lambda^{a}}\rt)^{j+1}\gamma^j\, m_j}{j(j+1)}\right)\lt[ 1 - o(1) \rt].
\end{align}
Furthermore, let $\bar{\lambda} = \lceil\lambda(1 - \frac{1}{\lambda^a})\rceil$ for some constant $a$ such that $0 < a < \frac{\eta d - 1}{2}$. Then, for any $\epsilon > 0$,
\begin{align}
\label{eqn:PaccL2}
 \pno \leq N_{\lambda, \lambda} \,\exp\left( - \frac{\gamma \bar{\lambda}^2\, (m_1 - \epsilon)}{2\lambda^{\eta d}} \right) \lt[ 1 + o(1) \rt],
\end{align}
where $N_{\lambda, \lambda}$ satisfies \eqref{eqn:Kr-term22} and \eqref{eqn:Kr-term2}.  In particular, \eqref{eqn:PaccL1} holds with $\epsilon = 0$ if $\eta d > 5/3$ and
$2 - \eta d < a < \frac{\eta d - 1}{2}$.
\end{lemma}

\begin{proof}

\noindent
{\bf Lower Bound:}
Fix $a$ such that $0 < a < 0.5$. Since $\pnon$ is a decreasing function of $n$ for any fixed $\lambda$, by Lemma~\ref{lem:binomial}, we can say that for all $n < \lambda\lt(1 + \frac{1}{\lambda^a}\rt)$,
\[
  \pnon \geq \exp\left( - \sum_{j = 1}^\infty \frac{\lambda^{j(1 - \eta d)+1} \lt(1+ \frac{1}{\lambda^{a}}\rt)^{j+1}\gamma^j\, m_j}{j(j+1)}\right),
\]
and from \eqref{eqn:pno_lam} and the Chernoff bound for the Poisson variable $N$ (see Lemma~\ref{lemma:cherpoi}),
\begin{align*}
\pno  &\geq \ee\lt(\mathcal{P}_N(\lambda) ; N <  \lambda\lt(1  + \frac{1}{\lambda^a}\rt) \rt) \nonumber \\
      &\geq  \pp\lt( N <  \lambda\lt(1 + \frac{1}{\lambda^a}\rt) \rt)\,\exp\left( - \sum_{j = 1}^\infty \frac{\lambda^{j(1 - \eta d)+1} \lt(1+ \frac{1}{\lambda^{a}}\rt)^{j+1}\gamma^j\, m_j}{j(j+1)}\right) \nonumber\\
      &\geq \lt(1 - \exp\lt( - \frac{1}{3}\lambda^{1 - 2a}\rt) \rt)\exp\left( - \sum_{j = 1}^\infty \frac{\lambda^{j(1 - \eta d)+1} \lt(1+ \frac{1}{\lambda^{a}}\rt)^{j+1}\gamma^j\, m_j}{j(j+1)}\right).
\end{align*}\\
Now \eqref{eqn:PaccL1} easily because $\exp\lt( - \frac{1}{3}\lambda^{1 - 2a}\rt) = o(1)$ as a function of $\lambda$.\\

\noindent
{\bf Upper Bound:}
From \eqref{eqn:pno_lam},
\begin{align}
\pno  &= \ee [\mathcal{P}_N(\lambda)]\leq \ee[\mathcal{P}_N(\lambda); N \geq \bar{\lambda}] + \pp(N < \bar{\lambda})
         \leq \mathcal{P}_{\bar{\lambda}}\lt(\lambda\rt) + \pp(N < \bar{\lambda}),\label{eqn:pno-upper-temp}
\end{align}
where the last inequality holds due the fact that $\pnon$ is a decreasing function of $n$ for fixed~$\lambda$.
We now analyze $\mathcal{P}_{\bar{\lambda}}\lt(\lambda\rt)$ and $\pp(N < \bar{\lambda})$ separately.\\

By Lemma~\ref{lem:binomial}, for any $\epsilon > 0$,
\begin{align*}
   \mathcal{P}_{\bar{\lambda}}\lt(\lambda\rt)&\leq N_{\lambda, \lambda}\, \lt[\exp\left( - \frac{\gamma \bar{\lambda}(\bar{\lambda} - 1)\, (m_1 - \epsilon)}{2\lambda^{\eta d}} \right) + 2\exp\left( - \frac{\bar{\lambda}\,\epsilon^2}{(2\rbdd)^{2d}}\right) \rt]\\
                                           &\leq N_{\lambda, \lambda}\,\lt[\exp\left( - \frac{\gamma \bar{\lambda}^2\, (m_1 - \epsilon)}{2\lambda^{\eta d}} \right) \exp\left( \frac{\gamma \, m_1}{2\lambda^{\eta d -1}} \right) + 2\exp\left( - \frac{\bar{\lambda}\,\epsilon^2}{(2\rbdd)^{2d}}\right) \rt],
\end{align*}
where we used the fact that $\bar{\lambda} \leq \lambda$. We rewrite the above expression as follows:
\begin{align*}
   \mathcal{P}_{\bar{\lambda}}\lt(\lambda\rt)\leq N_{\lambda, \lambda}\,\exp\left( - \frac{\gamma \bar{\lambda}^2\, (m_1 - \epsilon)}{2\lambda^{\eta d}} \right)\lt(\exp\left( \frac{\gamma \, m_1}{2\lambda^{\eta d -1}} \right)  + 2\exp\left( \frac{\gamma \bar{\lambda}^2\, m_1}{2\lambda^{\eta d}}- \frac{\bar{\lambda}\,\epsilon^2}{(2\rbdd)^{2d}}\right) \rt).
\end{align*}
Note that $\frac{\gamma \bar{\lambda}^2\, m_1}{2\lambda^{\eta d}} = O\lt( \lambda^{2-\eta d}\rt)$ and $\frac{\bar{\lambda}\,\epsilon^2}{(2\rbdd)^{2d}} = \Omega\lt(\lambda \rt)$.
Since $\eta d > 1$,
\begin{align}
\label{eqn:delta1}
 2\exp\left( \frac{\gamma \bar{\lambda}^2\, m_1}{2\lambda^{\eta d}}- \frac{\bar{\lambda}\,\epsilon^2}{(2\rbdd)^{2d}}\right) \leq 2\exp\left( - \bar{\lambda} \lt( \frac{\epsilon^2}{(2\rbdd)^{2d}} - \frac{\gamma \, m_1}{2\lambda^{\eta d  -1}}\rt)\right)
                                                                                                                        \longrightarrow 0, \text{ as }\,  \lambda \rightarrow \infty,
\end{align}
and since $\lim_{\lambda \rightarrow \infty}\exp\left( \frac{\gamma \, m_1}{2\lambda^{\eta d -1}} \rt)= 1$,
we can say that the first term $\mathcal{P}_{\bar{\lambda}}\lt(\lambda\rt)$ in \eqref{eqn:pno-upper-temp} satisfies the following inequality,
\begin{align*}
   \mathcal{P}_{\bar{\lambda}}\lt(\lambda\rt)\leq N_{\lambda, \lambda}\,\exp\left( - \frac{\gamma \bar{\lambda}^2\, (m_1 - \epsilon)}{2\lambda^{\eta d}} \right) \lt[1  + o(1)\rt].
\end{align*}
 By Lemma~\ref{lemma:cherpoi},
\begin{align}
\label{eqn:poi-temp1}
 \pp\lt(N \leq \bar{\lambda}\rt) \leq \pp\lt(N \leq \lambda\lt(1 - \frac{1}{\lambda^a} \rt)\rt) \leq \exp\lt( - \frac{\lambda^{1 - 2a}}{2}\rt).
\end{align}
Since $2a < 1$ (because $\eta d \leq 2$),  we have $1 - 2a > 2 - \eta d,$ and hence using \eqref{eqn:poi-temp1} and the fact that $N_{\lambda, \lambda} \geq 1$,
\begin{align}
 \label{eqn:delta2}
  \frac{\exp\left(\frac{\gamma \bar{\lambda}^2 (m_1 - \epsilon)}{\lambda^{\eta d}}\right)\pp\lt(N \leq \bar{\lambda}\rt)}{N_{\lambda, \lambda}}
   \leq \exp\left(\frac{ \gamma \bar{\lambda}^2 (m_1 - \epsilon)}{2\lambda^{\eta d}} - \frac{\lambda^{1 - 2a}}{2}\right)\,
   \longrightarrow 0,\, \text{  as  }\, \, \lambda \rightarrow \infty,
 \end{align}
and hence \eqref{eqn:PaccL2} follows from \eqref{eqn:pno-upper-temp} and \eqref{eqn:delta2}. 

In particular if $\eta d > 5/3$, we choose $a$ such that $2 - \eta d < a < \frac{\eta d - 1}{2}$. Let $\epsilon = 1/\lambda^a$.
Then \eqref{eqn:delta1} and \eqref{eqn:delta2} holds. We complete the proof using the fact that $\lim_{\lambda \rightarrow \infty}\exp\left( \frac{\gamma \bar{\lambda}^2\, \epsilon}{2\lambda^{\eta d}} \right) = 1$.
\end{proof}

\begin{proof}[Proof of Theorem~\ref{lem:non-ovr-rand}]
The following upper and lower bounds together complete the proof.\\

\noindent
{\bf Lower Bounds:} Consider the inequality \eqref{eqn:PaccL1}.\\

\noindent
{\bf \underline{Case: $\boldsymbol{\eta d > 2}$.}} Since $R \leq \rbdd$,
we have $m_j \leq (2\rbdd)^{jd}$. Thus, for $\eta d > 2$, all the terms in the exponent of the right-hand side of \eqref{eqn:PaccL1} go to zero asymptotically. In other words, for any $0 < a < 0.5$,
\[
 \lim_{\lambda \rightarrow \infty } \sum_{j = 1}^\infty \frac{\lambda^{j(1 - \eta d)+1} \lt(1+ \frac{1}{\lambda^{a}}\rt)^{j+1}\gamma^j\, m_j}{j(j+1)} = 0.
\]
That means, $ \lim_{\lambda \rightarrow \infty} \pno = 1,\,\, \text{ for }\, \eta d > 2.$\\

\noindent
{\bf \underline{Case: $\boldsymbol{3/2 < \eta d \leq 2}$.}}
Using \eqref{eqn:PaccL1}, $\pno \exp\lt( \frac{\gamma m_1}{2}\lambda^{2 - \eta d} \rt)$ is bounded from below by 
\[
\exp\left(  O\lt(\lambda^{2 - \eta d - a}\rt) - \sum_{j = 2}^\infty \frac{\lambda^{j(1 - \eta d)+1} \lt(1+ \frac{1}{\lambda^{a}}\rt)^{j+1}\gamma^j\, m_j}{j(j+1)}\right)\lt[ 1 - o(1) \rt].
\]
By fixing $a > 2 - \eta d$, we see that the right-hand side of the above expression goes to one as $\lambda \nearrow \infty$. Thus,
$\liminf_{\lambda \rightarrow \infty } \lt[\pno \exp\lt( \frac{\gamma m_1}{2}\lambda^{2 - \eta d} \rt)\rt] \geq 1.$\\

\noindent
{\bf \underline{Case: $\boldsymbol{1 < \eta d \leq 3/2}$.}} By applying $\log$ on both the sides of \eqref{eqn:PaccL1}, we have for any $0 < a < 0.5$ that
\begin{align*}
 \log\pno \geq \log \lt(1 - o(1) \rt) - \sum_{j = 1}^\infty \frac{\lambda^{j(1 - \eta d)+1} \lt(1+ \frac{1}{\lambda^{a}}\rt)^{j+1}\gamma^j\, m_j}{j(j+1)},
\end{align*}
and see that
\begin{align*}
 \frac{1}{\lambda^{2 - \eta d}}\sum_{j = 1}^\infty \frac{\lambda^{j(1 - \eta d)+1} \lt(1+ \frac{1}{\lambda^{a}}\rt)^{j+1}\gamma^j\, m_j}{j(j+1)} = \frac{\lt(1 + \frac{1}{\lambda^\epsilon} \rt)^2 \gamma m_1}{2} + \sum_{j = 2}^\infty \frac{\lt(1+ \frac{1}{\lambda^{a}}\rt)^{j+1}\gamma^j\, m_j}{\lambda^{(j-1)(\eta d -1 )} j(j+1)}.
\end{align*}
Thus $\liminf_{\lambda\rightarrow \infty} \frac{1}{\lambda^{2 - \eta d}} \log\pno \geq - \frac{\gamma m_1}{2}$
for $\eta d > 1$.\\

\noindent
{\bf \underline{Case: $\boldsymbol{0 < \eta d \leq 1}$.}}
Configurations with one sphere or no sphere is always accepted, that is, $\mathcal{P}_1(\lambda) = \mathcal{P}_0(\lambda) = 1$.
The probability of generating no sphere is $e^{- \lambda}$. Consequently, $\pno  > e^{-\lambda}$ and for any $\eta d > 0$,
\begin{align}
\label{eqn:Triv_LB}
 \liminf_{\lambda \rightarrow \infty}\lt[\frac{1}{\lambda}\log \pno \rt] \geq - 1.
\end{align}
In particular, assume that $\eta d = 1$.
For this case, first we show that the limit $\delta := \lim_{\lambda \rightarrow \infty}\lt[\frac{1}{\lambda}\log \pno \rt]$ exists.
To prove this, partition the cube $[0,1]^d$ into a cubic grid of cell-edge length $x^{1/d} \in (0,1)$. Ignore the cells at the boundary that have the edge length strictly smaller than $x^{1/d}$. So, the total intensity of the underlying PPP over a cell is $\lambda x$. 

When $\eta d = 1$, {the radius} of each sphere is identical in distribution to $R/\lambda$. Observe that the non-overlapping probability of the spheres restricting to a cell is $\mathcal{P}(\lambda x)$ (see the definition of the non-overlapping probability). Since the total number of cells is at least $1/x$,
the non-overlapping probability $\pno \leq \lt(\mathcal{P}(\lambda x)\rt)^{\frac{1}{x}},$ and thus
$\frac{1}{\lambda} \log \pno \leq \frac{1}{\lambda x}\log \mathcal{P}(\lambda x).$
We can increase $\lambda$ and decrease the cell-edge length $x^{1/d}$ such that $y :=\lambda x$ is fixed. Then the following inequality holds
\begin{align*}
 \limsup_{\lambda \rightarrow \infty} \lt[ \frac{1}{\lambda} \log \pno \rt] \leq \frac{1}{y}\log \mathcal{P}(y) < 0.
\end{align*}
Now the existence of the required limit is established by applying limit on $y$:
\begin{align*}
 \limsup_{\lambda \rightarrow \infty} \lt[ \frac{1}{\lambda} \log \pno \rt] \leq \liminf_{y \rightarrow \infty} \lt[\frac{1}{y}\log \mathcal{P}(y)\rt].
\end{align*}
To show that $\delta \nearrow 0$ as $\gamma m_1 \searrow 0$, assume that $\gamma m_1  < \epsilon$ for a constant $\epsilon \in (0,1)$.
By \eqref{eqn:PnBi} and \eqref{eqn:bni-ran},
\begin{align*}
\pnon \geq \ee\lt[\prod_{i=1}^{n-1} \lt(1 - \frac{\gamma}{\lambda} \sum_{k=1}^{i}(R_k + R_i)^d \rt)^+ \rt].
\end{align*}
Consider the following partial order on $\reals_+^n$: for any $y,\, y' \in \reals_+^n$, we say that $y \preceq y'$ if $y_i \leq y_i'$ for all $i=1,\dots,n$.
A function $f: \reals_+^n \rightarrow \reals$ is called {\it increasing} (or {\it decreasing}) if $f(y) \leq f(y')$ (or $f(y) \geq f(y')$) for all $y, \, y' \in \reals_+^n$
such that $y \preceq y'$. If $f$ and $g$ are either increasing or decreasing functions then Theorem~2.4 of \cite{Grimm99} (FKG inequality) {can be} trivially extended to show that
$\ee[f(Y)g(Y)] \geq \ee[f(Y)] \ee[g(Y)]$. Clearly the following function $f_i$ is a decreasing function {on $\reals_+^n$:}
\[
f_{i}(y) = \lt(1 - \frac{\gamma}{\lambda} \sum_{k=1}^{i}(y_k + y_i)^d \rt)^+.
\]
Therefore,
\[
\ee\lt[\prod_{i=1}^{n-1} \lt(1 - \frac{\gamma}{\lambda} \sum_{k=1}^{i}(R_k + R_i)^d \rt)^+ \rt] \geq \prod_{i=1}^{n-1} \ee\lt[\lt(1 - \frac{\gamma}{\lambda} \sum_{k=1}^{i}(R_k + R_i)^d \rt)^+ \rt].
\]
Using the convexity of the function $x^+$ and Jensen's inequality, for each $i$,
\begin{align*}
\ee\lt[\lt(1 - \frac{\gamma}{\lambda} \sum_{k=1}^{i}(R_k + R_i)^d \rt)^+ \rt] \geq \lt(1 - i\,\frac{\gamma m_1}{\lambda} \rt)^+,
\end{align*}
and thus $\pnon \geq \prod_{i=1}^{n-1} \lt(1 - i\,\frac{\gamma m_1}{\lambda} \rt)^+.$ With $\underline \lambda = \lfloor \lambda + \lambda^{0.75}\rfloor$ and $N \sim Poi(\lambda)$,
\begin{align*}
\pno = \sum_{n=0}^\infty e^{-\lambda} \frac{\lambda^n}{n!} \pnon \geq \sum_{n=0}^{\underline \lambda} e^{-\lambda} \frac{\lambda^n}{n!} \pnon \geq \mathcal{P}_{\underline \lambda}(\lambda) \pp\lt(N \leq \underline{\lambda} \rt).
\end{align*}
By applying $\log$ on both the sides of the above inequality and scaling with $1/\lambda$,
\[
\frac{1}{\lambda} \log \pno \geq \frac{1}{\lambda} \log \mathcal{P}_{\underline \lambda}(\lambda) + \frac{1}{\lambda} \log \pp\lt(N \leq \underline{\lambda} \rt).
\]
From the definition of $\underline \lambda$ and Lemma~\ref{lemma:cherpoi}, the second term, $\frac{1}{\lambda} \log \pp\lt(N \leq \underline{\lambda} \rt)$, goes to zero as $\lambda \nearrow \infty$. We now focus on the first term, $\frac{1}{\lambda} \log \mathcal{P}_{\underline \lambda}(\lambda)$. Since $\gamma m_1  < \epsilon < 1$, for all $i \leq \underline \lambda$,
\[
\frac{i}{\lambda}\frac{\gamma m_1}{\epsilon} < \frac{\underline \lambda}{\lambda}\frac{\gamma m_1}{\epsilon} \leq \lt(1 + \frac{1}{\lambda^{0.25}}\rt)\frac{\gamma m_1}{\epsilon} \leq 1,
\]
for large values of $\lambda$. Thus, we can write using Bernoulli's inequality that
\begin{align*}
\mathcal{P}_{\underline \lambda}(\lambda) \geq \prod_{i=1}^{\underline \lambda} \lt(1 - i\,\frac{\gamma m_1}{\lambda} \rt) = \prod_{i=1}^{\underline \lambda} \lt(1 - \epsilon \,\frac{i \gamma m_1}{\epsilon \lambda} \rt)
\geq \prod_{i=1}^{\underline \lambda}\lt(1 - \epsilon\rt)^{\,\frac{i \gamma m_1}{\epsilon \lambda}}
\end{align*}
for large values of $\lambda$. Therefore, by combining the trivial bound \eqref{eqn:Triv_LB} and the above conclusions,
\begin{align*}
\delta \geq \max\lt( -1, \frac{\gamma m_1}{2} \lt[ \frac{\log(1 - \epsilon)}{\epsilon}\rt] \rt) \longrightarrow 0\,\, \text{ as }\,\, \gamma m_1 \searrow 0.
\end{align*}

\noindent
{\bf Upper Bounds:} We have a complete proof of the large deviation of $\pno$ for the case $\eta d > 2$.
So, it remains to prove the theorem for $0 < \eta d \leq 2$. We first prove the required upper bounds for the case $1 < \eta d \leq 2$.
If $0 < a < 0.5$ and $\bar{\lambda} = \lceil\lambda(1 - \frac{1}{\lambda^a})\rceil$, then from Lemma~\ref{lem:pno-upper}, for any $\epsilon > 0$,
\begin{align}
\label{eqn:pno-upper1}
 \pno  \leq N_{\lambda, \lambda}\, \exp\left( - \frac{\gamma \bar{\lambda}^2\, (m_1 - \epsilon)}{2\lambda^{\eta d}} \right) \lt[ 1 + o(1) \rt].\\
\end{align}

\noindent
{\bf \underline{Case: $\boldsymbol{ \eta d  > 1}$.}} By applying $\log$ on both the sides of \eqref{eqn:pno-upper1} and then dividing by $\lambda^{2 -\eta d}$, we see that
\[
 \frac{1}{\lambda^{2 - \eta d}}\log \pno \leq - \frac{\gamma \, (m_1 - \epsilon)}{2}\lt(1 + \frac{1}{\lambda^a} \rt)^2 + \frac{1}{\lambda^{2 - \eta d}}\log N_{\lambda, \lambda} + \frac{1}{\lambda^{2 - \eta d}}\log [1 + o(1)].
\]
As a consequence of Lemma~\ref{lem:binomial}, $\limsup_{\lambda \rightarrow \infty} \frac{1}{\lambda^{2 - \eta d}}\log \pno \leq - \frac{\gamma(m_1 - \epsilon)}{2}.$ Now take $\epsilon \searrow 0$.

In particular, consider torus-hard-sphere model with  $5/3 < \eta d  \leq 2$.  We can fix $a$ such that $2 - \eta d < a < \frac{\eta d  - 1}{2}$. From Lemma~\ref{lem:pno-upper},
\eqref{eqn:pno-upper1} holds with $\epsilon = 0$. Therefore,
\[
 \pno \exp\lt( \frac{\gamma m_1}{2}\lambda^{2 - \eta d} \rt) \leq N_{\lambda, \lambda}\,\exp\left(  O\lt(\lambda^{2 - \eta d - a}\rt) \rt) \lt[ 1 + o(1) \rt],
\]
and hence $\limsup_{\lambda \rightarrow \infty}\lt[ \pno \exp\lt(\frac{\gamma m_1}{2}\lambda^{2 - \eta d} \rt)\rt] \leq 1$ from Lemma~\ref{lem:binomial}.\\

\noindent
{\bf \underline{Case: $\boldsymbol{0 < \eta d  < 1}$.}} Let $\underline{\lambda} = \lfloor \lambda^{\frac{1+\eta d}{2}}\rfloor$ and $N \sim Poi(\lambda)$.
From the definition
\begin{align}
\label{eqn:pno_etad_low}
\pno = \ee \lt[ \mathcal{P}_N(\lambda)\rt]  \leq \pp \lt( N \leq \underline{\lambda}\rt) + \ee \lt[ \mathcal{P}_N(\lambda) ; N \geq \underline{\lambda} + 1\rt].
\end{align}
For any $\epsilon > 0$, let $H_n(\epsilon) := \lt\{ \frac{1}{n} \sum_{i=1}^n R_i^d > \epsilon \rt\}$. From \eqref{eqn:bni},
\begin{align*}
\mathcal{P}_{n+1}(\lambda) \leq \ee\lt[\prod_{i=1}^{n} \lt(1 - \frac{\gamma'}{\lambda^{\eta d}} \sum_{j=1}^i R_j^d\rt)^+ \rt] \leq \pp\lt(H^c_n\lt( \frac{\lambda^{\eta d}}{\gamma' n}\rt) \rt) \leq \pp\lt(H^c_n\lt( \frac{\lambda^{\eta d}}{\gamma' \underline{\lambda}}\rt) \rt),
\end{align*}
where the second inequality holds because $\lt(1 - \frac{\gamma'}{\lambda^{\eta d}} \sum_{j=1}^n R_j^d\rt)^+ = 0$ on $H_n\lt( \frac{\lambda^{\eta d}}{\gamma' n}\rt)$.
Since $\eta d < (\eta d +1)/2 < 1$, see that $\frac{\lambda^{\eta d}}{\gamma' \underline{\lambda}} \searrow 0$ as $\lambda \nearrow \infty$, and thus for every $\epsilon > 0$ there exists
$\lambda_\epsilon$ such that
\begin{align*}
\mathcal{P}_{n+1}(\lambda) \leq \pp\lt(H^c_n\lt( \epsilon\rt) \rt),
\end{align*}
for all $\lambda > \lambda_\epsilon$ and $n > \underline{\lambda}$.

Suppose there is a constant $c > 0$ such that $R \geq c$. Then for all sufficiently small values of $\epsilon$,
$\pp\lt(H^c_n\lt(\epsilon\rt) \rt) = 0$ for all $n > \underline{\lambda}$.  Thus for large values of $\lambda$, $\pno \leq \pp \lt( N \leq \underline{\lambda}\rt) \leq e^{-\lambda} \lambda^{\underline{\lambda}}$, and from the definition of $\underline{\lambda}$,
\[
\limsup_{\lambda \rightarrow \infty} \frac{1}{\lambda} \log \pno \leq - 1 + \limsup_{\lambda \rightarrow \infty} \lt[\frac{\underline{\lambda} \log \lambda}{\lambda} \rt] = -1.
\]
So we can assume that $\pp(R < \epsilon)> 0$ for every $\epsilon> 0$. Recall that $\pp(R > 0) = 1$ and  $\Lambda(\theta)$ is the logarithmic moment generating function of $R^d$. As a consequence of positivity of $R$, we see that $\Lambda(\theta) \searrow -\infty$ as $\theta \searrow -\infty$.
Let $\Lambda^*(x) = \sup_{\theta \in \reals} \lt\{\theta x - \Lambda(\theta)\rt\}$. As a consequence of the assumption that $\pp(R < \epsilon)> 0$ for every $\epsilon> 0$,
we can show $\Lambda^*(x) \nearrow \infty$ as $x \searrow 0$. From
Theorem~2.2.3 of \cite{DZ10},
$$\pp\lt(H^c_n\lt( \epsilon\rt) \rt) \leq 2 \exp\lt( - n \inf_{x \leq \epsilon} \Lambda^*(x)\rt) = 2 \exp\lt( - n \Lambda^*(\epsilon)\rt)$$ for all $n > \underline{\lambda}$ and $\epsilon < \ee[R_1^d]$, where the last inequality holds because $\Lambda^*(x)$ is non-decreasing over $0 < x \leq \ee[R_1^d]$. By \eqref{eqn:pno_etad_low},
\begin{align*}
\pno &\leq \pp \lt( N \leq \underline{\lambda}\rt) + \pp \lt(H^c_N\lt(\epsilon\rt); N \geq \underline{\lambda} + 1\rt) \leq \pp \lt( N \leq \underline{\lambda}\rt) + \pp \lt(H^c_N\lt(\epsilon\rt)\rt)\\
&\leq \pp \lt( N \leq \underline{\lambda}\rt) + 2\exp\lt(- \lambda\lt(1 - e^{-\Lambda^*(\epsilon)} \rt) \rt),
\end{align*}
for all $\lambda \geq \lambda_\epsilon$. To conclude that $\limsup_{\lambda \rightarrow \infty} \frac{1}{\lambda} \log \pno \leq -1$, see from the definition of Poisson distribution and $\underline{\lambda}$ that
$$\pp \lt( N \leq \underline{\lambda}\rt) = \sum_{n= 0}^{\underline{\lambda}} e^{-\lambda} \frac{\lambda^n}{n!} \leq \underline{\lambda} e^{-\lambda} \lt(\frac{\lambda^{\underline{\lambda}}}{\underline{\lambda}!}\rt) \leq e^{-\lambda} \lambda^{\underline{\lambda}},$$ where we used the fact that $\lambda^{n-1}/(n-1)! < \lambda^{n}/n!$ for all $n < \lambda$. Hence,
\begin{align*}
\pno &\leq 2\exp\lt(- \lambda\lt(1 - e^{-\Lambda^*(\epsilon)} \rt) \rt) \lt( 1 + \lambda^{\underline{\lambda}} \, \exp\lt(- \lambda\Lambda^*(\epsilon)\rt)\rt)\\
     &= 2\exp\lt(- \lambda\lt(1 - e^{-\Lambda^*(\epsilon)} \rt) \rt) \lt( 1 + \exp\lt(- \lambda\lt(\Lambda^*(\epsilon) - \frac{\underline{\lambda}}{\lambda} \log\lambda\rt)\rt)\rt).
\end{align*}
From the definition of $\underline{\lambda}$, see that $\frac{\underline{\lambda}}{\lambda} \log\lambda \searrow 0$ as $\lambda \nearrow \infty$. As a consequence, as $\lambda \nearrow \infty$, $$\exp\lt(- \lambda\lt(\Lambda^*(\epsilon) - \frac{\underline{\lambda}}{\lambda} \log\lambda\rt)\rt)$$ goes to zero.
Therefore,
\begin{align*}
\limsup_{\lambda \rightarrow \infty} \frac{1}{\lambda} \log \pno \leq -\lt(1 - e^{-\Lambda^*(\epsilon)} \rt).
\end{align*}
We have the required result by taking $\epsilon \searrow 0$.\\

{\bf \underline{Case: $\boldsymbol{\eta d  = 1}$.}}
{It remains} to show that $\delta \leq -\frac{1}{2}\lt(1 - \frac{1}{\gamma' \rbdd^d}\rt)^2$ if $R \equiv \rbdd$ and $\gamma' \rbdd^d > 1$.
Since, from \eqref{eqn:bni}, $\mathcal{P}_{n+1}(\lambda) \leq \prod_{i=1}^{n} \lt(1 - \frac{\gamma'}{\lambda} i \rbdd^d\rt)^+  = 0,$ for all $n > \lambda \frac{\lambda}{\gamma' \rbdd^d},$ we have $\pno \leq \pp \lt( N \leq \lambda \frac{\lambda}{\gamma' \rbdd^d}\rt).$ Now the proof is complete by Lemma~\ref{lemma:cherpoi}.
\end{proof}

\subsection{Proof of Proposition~\ref{prop:Pack_int}}
From \cite{MMSWD01}, the intensity of the torus-hard-sphere model is given by 
\[ 
\rho(\lambda) = \frac{\sum_{n = 1}^\infty n \, \frac{\lambda^n}{n!} \mathcal{P}_{n}(\lambda)}{\sum_{n \in \mbb{N}_0} \frac{\lambda^n}{n!} \mathcal{P}_{n}(\lambda)} = \lambda \frac{\sum_{n \in \mbb{N}_0} \frac{\lambda^n}{n!} \mathcal{P}_{n+1}(\lambda)}{\sum_{n \in \mbb{N}_0} \frac{\lambda^n}{n!} \mathcal{P}_{n}(\lambda)},
\]
where $\mathcal{P}_{n}(\lambda)$ is the non-overlapping probability of $n$ uniformly and independently generated spheres with radius $\rbdd/\lambda^\eta$.\\

\noindent
\underline{\bf Case $\boldsymbol{\eta d >1}$.} In this regime, we show that  $\rho(\lambda)$ is of the order of $\gamma \rbdd^d \lambda^{1 - \eta d}$.
Using inequalities \eqref{eqn:bni} and \eqref{eqn:bni-ran},
\[
\lt(1 - n \gamma \rbdd^d \lambda^{-\eta d}\rt) \mathcal{P}_{n}(\lambda) \geq \mathcal{P}_{n+1}(\lambda) \geq \lt(1 - n \gamma 4\rbdd^d\lambda^{-\eta d}\rt) \mathcal{P}_{n}(\lambda).
\]
Therefore,
\begin{align*}
\rho(\lambda) \geq  \lambda \frac{\sum_{n \in \mbb{N}_0} \frac{\lambda^n}{n!}\lt(1 - n  4\rbdd^d\lambda^{-\eta d}\rt) \mathcal{P}_{n}(\lambda) }{\sum_{n \in \mbb{N}_0} \frac{\lambda^n}{n!} \mathcal{P}_{n}(\lambda)} \geq \lambda - \gamma 4\rbdd^d \lambda^{1 - \eta d} \rho(\lambda),
\end{align*}
and 
\begin{align*}
\rho(\lambda) \leq  \lambda \frac{\sum_{n \in \mbb{N}_0} \frac{\lambda^n}{n!}\lt(1 - n \rbdd^d\lambda^{-\eta d}\rt) \mathcal{P}_{n}(\lambda) }{\sum_{n \in \mbb{N}_0} \frac{\lambda^n}{n!} \mathcal{P}_{n}(\lambda)} \leq \lambda - \gamma \rbdd^d \lambda^{1 - \eta d}  \rho(\lambda).
\end{align*}
Consequently, 
\[
\lt(\frac{1}{ 1 + \gamma 4\rbdd^d \lambda^{1 - \eta d}} \rt) \gamma \rbdd^d \lambda^{1- \eta d} \leq \rho(\lambda)\gamma \rbdd^d \lambda^{- \eta d} \leq \lt(\frac{1}{ 1 + \gamma \rbdd^d \lambda^{1 - \eta d}} \rt) \gamma \rbdd^d \lambda^{1- \eta d},
\]
and thus  $\lim_{\lambda \nearrow \infty} \frac{\mathsf{VF}(\lambda)}{\gamma \rbdd^d \lambda^{1 - \eta d}} = 1$.\\

\noindent
\underline{\bf Case $\boldsymbol{\eta d \leq 1}$.}  We know show that $\lim_{\lambda \nearrow \infty} \mathsf{VF}(\lambda) \leq \rho^{\max} \gamma$ with equality if and only if $\eta d < 1$. Towards this end, we first consider another torus-hard-sphere model on $[0, \lambda^\eta/\rbdd]^d$ with unit radius spheres  and  absolutely continuous  with respect to a $\kappa$-homogeneous Poisson point process for some intensity $\kappa > 0$. Let  $\rho(\kappa, \lambda)$ be the intensity of this new hard-sphere model. We can easily see that when $\kappa = \rbdd^d \lambda^{1-  \eta d}$, the fraction of the volume occupied by the spheres in the new hard-sphere model is also $ \mathsf{VF}(\lambda)$.

The proof of Proposition~1 and 2 of \cite{MMSWD01} can be easily modified to show that  $\rho(\kappa, \lambda)$ is strictly increasing in $\kappa$ for any fixed $\lambda > 0$, and
\[
\lim_{\kappa \to \infty} \rho(\kappa, \lambda) = \rho^{max}, 
\]
where $\rho^{max}$ is the closest packing density.
On the other hand, by fixing $\kappa$,we can further using \cite{MMSWD01} show that the limit $\lim_{\lambda \to \infty} \rho(\kappa, \lambda)$ exists and is equal to the intensity of the {\em stationary} hard-sphere model on $\reals^d$ with unit radius spheres and the reference PPP being $\kappa$-homogeneous. (In fact, the difference between $\rho(\kappa, \lambda)$ and the limit $\lim_{\lambda \to \infty} \rho(\kappa, \lambda)$ is known to be insignificant for large values of $\lambda$; see, for example, \cite{BN12}.) 

From the above discussion, when $\eta d < 1$, for sufficiently small $\epsilon > 0$, there exist constants $\kappa_\epsilon$ and $\lambda_\epsilon$ such that 
$\rho(\kappa, \lambda) > \rho^{max} - \epsilon,$
for all $\lambda > \lambda_\epsilon$ and $\kappa > \kappa_\epsilon$.
If we take $\kappa = \rbdd^d\,\lambda^{1- \eta d}$, since $\eta d< 1$,
\[
 \lim_{\lambda \to \infty} \mathsf{VF}(\lambda) = \lim_{\lambda \to \infty} \lt[\rho(\rbdd^d\lambda^{1- \eta d}, \lambda) \gamma  \rt] = \rho^{max} \gamma,
\] 
which is the maximum packing intensity.

Finally, if $\eta d = 1$ and $\kappa = \rbdd^d$, the limit $\lim_{\lambda \to \infty} \rho(\rbdd^d, \lambda)$ is strictly less than $\rho^{max}$. Hence, we have $\lim_{\lambda \to \infty} \mathsf{VF}(\lambda) < \rho^{max} \gamma$.

\subsection{Proof of Proposition~\ref{prop:AR_meth}}
\label{sec:ComResults}

\newcommand{\pnonn}{\mathcal{P}_{n-1}}
Let $N'$ be the number of spheres generated sequentially, independently and identically before seeing an overlap. {Let $N \sim \pois(\lambda)$ independently of $N'$.}
Then from the construction of Algorithm~\ref{alg:ext1},
\begin{align}
\label{eqn:citr_1}
c\,\ee\lt[\sum_{n=1}^{\min(N, N')} (n-1)\rt]  \leq C_{\mathsf{itr}}(\lambda) \leq  c'\,\lt(\log(\lambda) + \ee\lt[\sum_{n=1}^{\min(N, N')} (n-1) \rt] \rt),
\end{align} 
for some positive constants $c$ and $c'$. In the above expression, $\log(\lambda)$ appears because the cost to generate a sample of $N$ is at most an order of $\log(\lambda)$ (see, e.g., \cite{Dev86}). 
Observe that
\begin{align}
\label{eqn:C_itr}
\ee\lt[\sum_{n=1}^{\min(N, N')} (n-1) \rt] = \ee\lt[\sum_{n=0}^{N -1} n I(N' > n) \rt] = \ee\lt[\sum_{n=1}^{N-1} n \pnon \rt],
\end{align}
where the last equality follows from the fact that $\pp(N' > n) = \pnon$.\\

\noindent
{\bf Upper bound:} For $\eta d \geq 2$, since $\pnon \leq 1$, we can upper bound \eqref{eqn:C_itr} by a constant times $\ee[N^{2}]$, which is further bounded from  above by a constant times $\lambda^2$.
Therefore we just need to consider the case $\eta d < 2$. 
From \eqref{eqn:PnBi}, ${\pnon = {\mbb E}_{\wt \mu_n}\left[\prod_{i=1}^{n}\Big(1 - B_{i}\Big)\right]}$. As a consequence of \eqref{eqn:bni},
$$\displaystyle\pnon \leq \ee\lt[\exp\lt(- \frac{\gamma'}{\lambda^{\eta d}} \sum_{1 \leq j < i \leq n} R_j^d \rt)  \rt] = \ee\lt[\exp\lt(- \frac{\gamma' \rbdd^d}{\lambda^{\eta d}} \sum_{1 \leq j < i \leq n} \frac{R_j^d}{\rbdd^d} \rt)  \rt].$$
Let $\alpha = \ee[R_1^d]$. Since $\rbdd$ is an upper bound on the $R_i'$s,  by Hoeffding's inequality (Lemma~\ref{lem:Hoeffding}) on the sequence $\{R^d_1/\rbdd^d, \dots, R^d_n/\rbdd^d\}$ with $\epsilon = \alpha/2\rbdd^d$, $k = 2$ and $g(x,y) = x$,
\begin{align*}
 \pnon &\leq \exp\lt( - \frac{\gamma'}{ 2\lambda^{\eta d}} \frac{n(n-1)}{2} \alpha\rt)  + \pp\lt( \sum_{1 \leq j < i \leq n} \frac{R_j^d}{\rbdd^d}  \leq  \frac{\alpha}{2\rbdd^d} \rt) \\
     &\leq \exp\lt( - \frac{\gamma' (n-1)^2}{ 4\lambda^{\eta d}}  \alpha\rt)  + \exp\lt(-\frac{n\alpha^2}{4\rbdd^{2d}} \rt).
\end{align*}
Let $a = \sqrt{\frac{2\lambda^{\eta d}}{\gamma' \alpha}}$.   Then from the above expression,
\begin{align}
\label{eqn:upbdd_pn}
 \sum_{n=1}^\infty n\, \pnon \leq  \sum_{n = 1}^\infty n \exp\lt( -\frac{(n-1)^2}{2 a^2} \rt) + \sum_{n=1}^\infty n\exp\lt(- \frac{n\alpha^2}{4\rbdd^{2d}}\rt).
\end{align}
Select $\lambda$ large enough so that $b > 0$. Then with $p = 1 - \exp(-\alpha^2/(4\rbdd^{2d}))$, the second term on the right side of \eqref{eqn:upbdd_pn} is $1/p$ times $\ee\lt[Z\rt]$ for a geometric random variable $Z$ with success probability $p$ and support $\{1,2,3,\dots\}$. Since $\ee\lt[Z\rt] = 1/p$, the term $\sum_{n=1}^\infty n\exp\lt(- n \alpha^2/(4\rbdd^{2d})\rt)$ bounded from above by a constant.

On the other hand, since $n \exp\lt( - (n-1)^2/(2 a^2) \rt) \leq \int_{n - 2}^{n - 1} (x+2) \exp\lt( - \frac{x^2}{2 a^2} \rt)$ for any non-negative integer $n$, we can write that
\begin{align*}
 \sum_{n = 1}^\infty  n \exp\lt( -\frac{(n-1)^2}{2 a^2} \rt) &\leq 1 + \int_{0}^\infty (x+2) \exp\lt( -\frac{x^2}{2 a^2} \rt) \, dx = 1 + a\sqrt{\frac{\pi}{2}}\,\, \ee\lt[\lt(\big|Z\big| +2\rt)\rt],
\end{align*}
where $Z$ is a Gaussian random variable with mean $0$ and variance $a^2$. 
{Since $\ee[|Z|] = a \sqrt{2/\pi}$}, using the definition of $a$, the first term in \eqref{eqn:upbdd_pn} is bounded from above by a constant times $\lambda^{\eta d}$.
Thus, the required upper bound on \eqref{eqn:citr_1} established as a consequence of \eqref{eqn:C_itr} and \eqref{eqn:upbdd_pn}.\\

\noindent
{\bf Lower bound:} Let $\epsilon' = \min(1, \eta d/2)$. Then from \eqref{eqn:C_itr},
$$\displaystyle C_{\mathsf{itr}}(\lambda) \geq \pp\lt(N > \lambda^{\epsilon'}/2\rt) \sum_{n=1}^{ \lfloor\lambda^{\epsilon'}/2 \rfloor + 1} n \pnon \geq c \,\mathcal{P}_{\lfloor\lambda^{\epsilon'}/2 \rfloor}(\lambda)\,\lt(\frac{\lambda^{\epsilon'}}{2}\rt)^{2},$$ for a constant $c > 0$.
From \eqref{eqn:pnon_low},
\[
 \mathcal{P}_{\lfloor\lambda^{\epsilon'}/2 \rfloor}(\lambda) \geq \exp\lt(- \frac{\lambda^{\epsilon'}}{2}  \sum_{j=1}^\infty \lt(\frac{\gamma \lambda^{\epsilon'} }{2\lambda^{\eta d}} \rt)^j \frac{m_j}{j(j+1)}\rt) \geq \exp\lt(- \frac{\lambda^{\epsilon'}}{2}  \sum_{j=1}^\infty \lt(\frac{\gamma \lambda^{\epsilon'} }{2\lambda^{\eta d}} \rt)^j \frac{m_j}{j}\rt),
\]
where $m_j = \ee\lt[(R_1 + R_2)^{j d}\rt]$. Note that $m_j \leq (2\rbdd)^{jd}$ since $R \leq \rbdd$. Therefore,
$\displaystyle\mathcal{P}_{\lfloor\lambda^{\epsilon'}/2 \rfloor}(\lambda) \geq  \exp\lt(- \frac{\lambda^{\epsilon'}}{2}  \sum_{j=1}^\infty \frac{1 }{j} \lt( \frac{\gamma (2\rbdd^d) }{2\lambda^{\eta d - \epsilon'}} \rt)^j \rt)$.
Using Taylor's expansion of $\log(1 - x)$ for $0< x <1$, and the fact that $\frac{\gamma (2\rbdd^d) }{2\lambda^{\eta d - \epsilon'}} < 1 $ for sufficiently large values of $\lambda$,
$$\displaystyle\mathcal{P}_{\lfloor\lambda^{\epsilon'}/2 \rfloor}(\lambda) \geq \exp\lt(\frac{\lambda^{\epsilon'}}{2} \log\lt(1 - \frac{\gamma (2\rbdd^d) }{2\lambda^{\eta d - \epsilon'}} \rt) \rt) = \lt(1 - \frac{\gamma (2\rbdd^d) }{2\lambda^{\eta d - \epsilon'}} \rt)^{\frac{\lambda^{\epsilon'}}{2}}.$$
From the definition of $\epsilon'$,
\[
 \lim_{\lambda \rightarrow \infty} \lt[ \lt(1 - \frac{\gamma (2\rbdd^d) }{2\lambda^{\eta d - \epsilon'}} \rt)^{\frac{\lambda^{\epsilon'}}{2}}\rt] = \begin{cases}
                                                                                                                                              1, & \quad \text{ if }\, \eta d > 2,\\
                                                                                                                                              \exp\lt( - \gamma (2\rbdd)^d/4 \rt), & \quad \text{ if }\, 0 < \eta d \leq 2.
                                                                                                                                             \end{cases}
\]
{In addition, from  Lemma~\ref{lemma:cherpoi}, ${\lim_{\lambda \rightarrow \infty}\pp\lt(N > \lambda^{\epsilon'}/2\rt) = 1}$. Therefore,
there exists a constant $c$ such that ${C_{\mathsf{itr}}(\lambda) \geq c\, \lambda^{2\epsilon'}}$. }
The proof of Proposition~\ref{prop:AR_meth} is complete using Theorem~\ref{lem:non-ovr-rand} and~\eqref{eqn:TAR}.

\subsection{Proof of Proposition~\ref{prop:imp_samp}}
\label{sec:proof_imp_samp}
First note that the sphere volume is at most a constant time the cell volume for all $\lambda$. Thus, after generating a sphere, the complexity of relabelling cells around the center of the new sphere plus the complexity of overlap check is a constant. For $\eta d > 1$, {the number of} spheres generated in {an iteration} of
Algorithm~\ref{alg:ext11} is stochastically dominated by a Poisson random variable with rate~$\lambda$. 
Therefore, there exists a constant $c$ such that $\wt C_{\mathsf{itr}}(\lambda) \leq c\, \lambda$.
On the other hand, if $0 < \eta d \leq 1$, the expected number of spheres generated per iteration is of order $\lambda^{\eta d}$ because the expected volume
of each sphere is an order of $1/\lambda^{\eta d}$. It is clear that there exists a constant $c > 0 $ such that
$\wt C_{\mathsf{itr}}(\lambda) \leq c\, \lambda^{\eta d } $.
Thus, by \eqref{eqn:ISAR_cost} and Proposition~\ref{prop:AR_meth},
\begin{align*}
 \mathcal{T}_{\mathsf{ISAR}} &\leq c \,  \frac{\lambda^{\min(1, \eta d)}}{P_{\mathsf{acc}} (\lambda)} = c\, \ee[\wt \sigma(N)] \frac{\lambda^{\min(1, \eta d)}}{\pno}.
\end{align*}
Thus, \eqref{eqn:ISAR_cost_grid} holds, because $\wt \sigma(n) = \delta_n$ for each $n \in \mbb{N}_0$. Furthermore, from the definition of $\wt \sigma(\cdot)$ and~$N$,
 \begin{align*}
  \ee\lt[\delta_N\rt] \leq  \ee\lt[ \exp\lt(- \sum_{i=0}^{N}(i-1) \gamma' \frac{\rbdd^d}{\lambda^{\eta d}} \rt) \rt] = \ee\lt[ \exp\lt(- \gamma' \frac{\rbdd^d}{2\lambda^{\eta d}}  \frac{(N-1)N}{2} \rt) \rt].
 \end{align*}
{By the Chernoff bound} (Lemma~\ref{lemma:cherpoi}), for any $0 < \epsilon < 1$,
\begin{align*}
 \ee\lt[ \exp\lt(- \gamma' \frac{\rbdd^d}{\lambda^{\eta d}} \frac{(N-1)N}{2} \rt) \rt] &\leq  \ee\lt[ \exp\lt(- \gamma' \frac{\rbdd^d}{\lambda^{\eta d}} \frac{(N-1)N}{2} \rt); N > \lambda\lt(1 - \epsilon\rt) \rt]\\
                                                       &\hspace{3cm} + \pp\lt( N \leq \lambda\lt(1 - \epsilon\rt)\rt)\\
                                                       &\leq  \exp\lt(- \frac{\gamma' \rbdd^d}{2}  \lt(1 - \epsilon \rt) \lambda^{2 - \eta d}\rt) + \exp\lt( -\frac{\lambda \epsilon^2}{2}\rt).
\end{align*}
If $\eta d > 1$, then the second term on the right-hand side of the above expression decreases faster than the first term, and thus the claim holds true.
For $\eta d = 1$, take $\epsilon = 1/2$, then we have the required result with $b = \min\lt(1/8, \gamma \rbdd^d/4\rt)$.
Furthermore, if $0 < \eta d < 1$ then the first term decreases faster than the second one, and  hence the proof is completed by taking $b = 1/2$.

\subsection{Proof of Theorem~\ref{thm:loss-system}}
To derive the lower bound on $\comDCone$, we view the entire dominating process $\mD$
as a Poisson Boolean model on a higher dimensional space and use an extension of the FKG inequality \cite{MR96}
(alternatively, see Theorem~2.2 in \cite{MR96}). 
Let $s_0 = 0$ and $s_i$ be the instant of the $i^{th}$ arrival in the dominating process after time zero.
Let $C(\lfs, \lfs^u, \lfs^l)$ be the running time complexity of updating the dominating, upper bound and lower bound processes
at the instant of an arrival when their respective states are $\lfs, \lfs^u$ and~$\lfs^l$.

Since $U_0(0) = D(0)$ and $L_0(0) = \varnothing$, on $\bigcap_{j=0}^i \{D(s_j) \notin \mathscr{A}\}$, for all $t \leq s_i,$
\begin{align}
\label{eqn:LowerNothing}
 L_0(t) = \varnothing\,\, \text{ and } \,\, U_0(t) = D(t).
\end{align}
Thus,  $L_0(t) \neq U_0(t)$ for all $t \leq s_i$
on $\bigcap_{j=0}^i \{D(s_j) \notin \mathscr{A}\}$.
Now take $$\tau = \inf\lt\{i \in \mbb{N}_0: D(s_i) \in \mathscr{A}\rt\}.$$ From the above conclusion, it  is clear that $N^f \geq \tau$.
Then,
\begin{align*}
 \mathcal{T}_{\mathsf{DC1}} &\geq \ee\lt[\sum_{i=0}^{N^f} C\Big(D(s_i), U_0(s_i), L_0(s_i)\Big) \rt]\\
 & \geq \ee\lt[\sum_{i=0}^{\tau } C\Big(D(s_i), U_0(s_i), L_0(s_i)\Big) \rt]\nonumber\\
                  &= \sum_{i=0}^\infty \ee\lt[ C\Big(D(s_i), U_0(s_i), L_0(s_i)\Big) ; \tau \geq i \rt]\nonumber\\
                  &= \sum_{i=0}^\infty \ee\lt[ C\Big(D(s_i), U_0(s_i), L_0(s_i)\Big) ; \bigcap_{j=0}^{i -1} \lt\{D(s_j) \notin \mathscr{A}\rt\}\rt]\nonumber\\
                  &= \sum_{i=0}^\infty \ee\lt[ C\Big(D(s_i), D(s_i), \varnothing\Big) ; \bigcap_{j=0}^{i -1} \lt\{D(s_j) \notin \mathscr{A}\rt\}\rt],\nonumber
\end{align*}
where $I\lt(\cap_{j=0}^{-1} \lt\{D(s_j) \notin \mathscr{A}\rt\} \rt) = 1$ and the last equality follows from \eqref{eqn:LowerNothing}.

Suppose that $\mathscr{D}$ is the state space of the entire process $\lt\{ D(t) : t \in \reals \rt\}$. Then we can define a simple partial order on $\mathscr{D}$ as follows:
For any $\omega, \omega' \in \mathscr{D}$, we say $\omega \preceq \omega'$ if and only if every sphere present in $\omega$ is also present in $\omega'$, that is, either $\omega' = \omega$ or $\omega'$ is obtained by
adding spheres to $\omega$. Define the following notion of increasing functions: A real valued function on $\mathscr{D}$ is increasing if $f(\omega) \leq f(\omega')$ for all
$\omega, \omega' \in \mathscr{D}$ such that $\omega \preceq \omega'$.

At each arrival,  if there  are $n$ points in  the upper bounding process, the cost to decide whether to accept the new point is at least the the cost to check overlap condition in the upper bounding process and that cost is an order of $n$. Therefore, $C\Big(D(s_i), U_0(s_i), \varnothing\Big) = c |U_0(s_i)|$ for some constant. Since $|U_0(s_i)|$ is a non-decreasing function on $\mathscr{D}$ as per the partial order stated above, by FKG inequality (Theorem~2.2 in \cite{MR96}), 
\[
 \ee\lt[ C\Big(D(s_i), D(s_i), \varnothing\Big) ; \bigcap_{j=0}^{i-1} \lt\{D(s_j) \notin \mathscr{A}\rt\} \rt]
\]
is bounded from below by
\[
 \ee\lt[ C\Big(D(s_i), D(s_i), \varnothing\Big)\rt] \prod_{j=0}^{i-1} \pp\lt( D(s_j) \notin \mathscr{A} \rt).
\]
Thus,
\begin{align*}
 \mathcal{T}_{\mathsf{DC1}} \geq  \frac{\ee\lt[ C\Big(D(s_1), D(s_1), \varnothing\Big)\rt]}{1 - \pp\lt( D(s_1) \notin \mathscr{A} \rt)} = \frac{\ee_{\mu^0}\lt[ |\state| \rt]}{\pp_{\mu^0} \lt( \state \in \mathscr{A} \rt)} = c\,\frac{\lambda}{\pno},
\end{align*}
for some constant $c > 0$. Then \eqref{eqn:TDC_vs_TAR} follows from \eqref{eqn:TAR} and \eqref{eqn:TAR-vs-Pno}. The proof is completed using Theorem~\ref{lem:non-ovr-rand}.

\end{appendices}

\bibliographystyle{abbrv}
\bibliography{Ref}
\end{document}